\documentclass[12pt]{amsart}
\usepackage{amssymb}
\usepackage{amsbsy}
\usepackage{amscd}
\usepackage{lineno}

\usepackage[mathscr]{eucal}
\usepackage{nicefrac}

\DeclareSymbolFont{rsfs}{U}{rsfs}{m}{n}
\DeclareSymbolFontAlphabet{\mathrsfs}{rsfs}

\usepackage{a4wide}
\usepackage[all]{xy}
\usepackage{tikz}
\usetikzlibrary{arrows,decorations.pathmorphing,backgrounds,positioning,fit,petri}
\usetikzlibrary{decorations.markings}

\usepackage{verbatim}
\usepackage{version}
\usepackage{color}
\definecolor{darkspringgreen}{rgb}{0.09, 0.45, 0.27}
\definecolor{deepjunglegreen}{rgb}{0.0, 0.29, 0.29}
\newenvironment{NB}{
\color{red}{\bf NB}. \footnotesize
}{}
\newenvironment{NB2}{
\color{blue}{\bf NB2}. \footnotesize
}{}
\newenvironment{NB3}{
\color{blue}{\bf new NB}. \footnotesize
}{}
\excludeversion{NB}
\excludeversion{NB2}
\excludeversion{NB3}
\makeatletter

\usepackage{url}
\usepackage{ifmtarg}

\usepackage{xr-hyper}
\usepackage[
pdftex,
bookmarks=true,
colorlinks=true,
citecolor=deepjunglegreen,
filecolor=deepjunglegreen,
debug=true,
pdfnewwindow=true]{hyperref}
\usepackage{cleveref}
\let\RR\relax
\let\bN\relax
\let\SL\relax
\let\PGL\relax
\usepackage{xspace}

\AtBeginDocument{%
\makeatletter
\let\oldref=\ref
\renewcommand{\ref}[1]{%
  \def\@mystring{affine_pre-#1}%
  \@ifundefined{r@\@mystring}{%
    \def\@mystring{blowup_pre-#1}%
    \@ifundefined{r@\@mystring}{%
      \def\@mystring{Coulomb2-#1}%
      \@ifundefined{r@\@mystring}{%
        \cref{#1}}{%
        \cite[\namecref{Coulomb2-#1}\oldref{Coulomb2-#1}]{main}%
      }%
    }%
    {%
      \cite[\namecref{blowup_pre-#1}\oldref{blowup_pre-#1}]{blowup}%
   }%
 }%
 {%
   \cite[\namecref{affine_pre-#1}\oldref{affine_pre-#1}]{affine}%
 }%
}
\renewcommand{\eqref}[1]{%
  \def\@mystring{affine_pre-#1}%
  \@ifundefined{r@\@mystring}{%
    \def\@mystring{blowup_pre-#1}%
    \@ifundefined{r@\@mystring}{%
      \def\@mystring{Coulomb2-#1}%
      \@ifundefined{r@\@mystring}{%
        \textup{\tagform@{\oldref{#1}}}}{%
        \text{\cite[(\oldref{Coulomb2-#1})]{main}}%
      }%
    }%
    {%
      \text{\cite[(\oldref{blowup_pre-#1})]{blowup}}%
    }%
  }%
  {%
    \text{\cite[(\oldref{affine_pre-#1})]{affine}}%
  }%
}
\makeatother
}

\crefname{Theorem}{Theorem\xspace}{Theorems}
\crefformat{Theorem}{Theorem~#2#1#3}
\crefname{section}{\S}{\S\S}
\crefformat{section}{\S#2#1#3} 
\crefmultiformat{section}{\S\S#2#1#3}{, #2#1#3}{, #2#1#3}{, #2#1#3}
\crefname{Lemma}{Lemma\xspace}{Lemmas\xspace}
\crefformat{Lemma}{Lemma~#2#1#3}
\crefname{Proposition}{Proposition\xspace}{Propositions\xspace}
\crefformat{Proposition}{Proposition~#2#1#3}
\crefname{Corollary}{Corollary\xspace}{Corollaries\xspace}
\crefformat{Corollary}{Corollary~#2#1#3}
\crefname{Definition}{Definition}{Definitions}
\crefformat{Definition}{Definition~#2#1#3}
\crefname{Remark}{Remark\xspace}{Remarks\xspace}
\crefformat{Remark}{Remark~#2#1#3}
\crefname{Remarks}{Remark\xspace}{Remarks\xspace}
\crefformat{Remarks}{Remark~#2#1#3}
\crefname{Conjecture}{Conjecture\xspace}{Conjectures\xspace}
\crefformat{Conjecture}{Conjecture~#2#1#3}
\crefname{figure}{Figure\xspace}{Figure\xspace}
\crefformat{figure}{Figure~#2#1#3}

\crefname{appendix}{Appendix\xspace}{Appendices\xspace}
\crefformat{appendix}{\S#2#1#3}
\crefmultiformat{appendix}{\S\S#2#1#3}{, #2#1#3}{, #2#1#3}{, #2#1#3}
\crefformat{enumi}{#2#1#3}

\crefname{equation}{}{}
\crefformat{equation}{(#2#1#3)}

\usepackage{stmaryrd}
\SetSymbolFont{stmry}{bold}{U}{stmry}{m}{n}
\usepackage{pigpen} 
\usepackage{enumitem}
\usepackage{calc}

\usepackage{aliascnt}
\usepackage{textgreek}

\renewcommand{\thesubsection}{\thesection(\@roman\c@subsection)}
\newenvironment{aenume}{%
  \begin{enumerate}%
  }{\end{enumerate}}
\newcounter{number}
\makeatother
\newtheorem{Theorem}[equation]{Theorem}
\newtheorem{Corollary}[equation]{Corollary}
\newtheorem{Lemma}[equation]{Lemma}
\newtheorem{Proposition}[equation]{Proposition}

\theoremstyle{definition}

\theoremstyle{remark}
\newtheorem{Remark}[equation]{Remark}

\numberwithin{equation}{section}

\newcommand{\secref}[1]{\ref{#1}}

\newcommand{\defeq}{\overset{\operatorname{\scriptstyle def.}}{=}}
\newcommand{\CC}{{\mathbb C}}
\newcommand{\ZZ}{{\mathbb Z}}
\newcommand{\QQ}{{\mathbb Q}}
\newcommand{\RR}{{\mathbb R}}
\newcommand{\proj}{{\mathbb P}}

\newcommand{\SL}{\operatorname{\rm SL}}

\newcommand{\GL}{\operatorname{GL}}
\newcommand{\PGL}{\operatorname{PGL}}

\newcommand{\gl}{\operatorname{\mathfrak{gl}}}

\newcommand{\Spec}{\operatorname{Spec}\nolimits}
\newcommand{\Proj}{\operatorname{Proj}\nolimits}
\newcommand{\End}{\operatorname{End}}
\newcommand{\Hom}{\operatorname{Hom}}

\newcommand{\Ima}{\operatorname{Im}}

\newcommand{\id}{\operatorname{id}}

\renewcommand{\MR}[1]{}

\newcommand{\dslash}{/\!\!/}
\newcommand{\bM}{\mathbf M}
\newcommand{\bN}{\mathbf N}

\newcommand{\wt}{\operatorname{wt}}

\newcommand{\shfO}{\mathcal O}
\newcommand{\tslash}{/\!\!/\!\!/}
\newcommand{\tslabar}{\mathbin{
\setbox0=\hbox{/\!\!/\!\!/}\rule[0.4\ht0]{\wd0}{.3\dp0}\kern-\wd0\box0}}

\newcommand{\Gr}{\mathrm{Gr}}

\newcommand{\cR}{\mathcal R}
\newcommand{\cF}{\mathcal F}

\newcommand{\cO}{\mathcal O}

\newcommand{\scP}{\mathscr P}
\newcommand{\scA}{\mathscr A}

\makeatletter
\newcommand{\cA}[1][{}]{%
  \@ifmtarg{#1}%
  {\mathcal A}
  {\mathcal A(#1)}
}
\newcommand{\cAh}[1][{}]{%
  \@ifmtarg{#1}%
  {\mathcal A_\hbar}
  {\mathcal A_\hbar(#1)}
}
\makeatother

\newcommand{\ft}{\mathfrak t}

\newcommand{\bNT}{\bN_T}

\newcommand{\scAfor}{\scA^{\mathrm{for}}}

\newcommand{\cL}{\mathcal L}

\newcommand{\po}{\ar@{}[dr]|{\text{\pigpenfont R}}}
\newcommand{\pb}{\ar@{}[dr]|{\text{\pigpenfont J}}}
\newcommand{\pp}{\ar@{}[dr]|{\text{\pigpenfont P}}}

\usepackage{accents}
\makeatletter
\newcommand{\sbullet}{%
  \hbox{\fontfamily{lmr}\fontsize{.4\dimexpr(\f@size pt)}{0}\selectfont\textbullet}}
\DeclareRobustCommand{\mathbullet}{\accentset{\sbullet}}
\makeatother

\newcommand{\cM}{\mathcal M}

\newcommand{\BA}{{\mathbb{A}}}
\newcommand{\BC}{{\mathbb{C}}}

\newcommand{\BG}{{\mathbb{G}}}
\newcommand{\BN}{{\mathbb{N}}}

\newcommand{\BP}{{\mathbb{P}}}

\newcommand{\BZ}{{\mathbb{Z}}}

\newcommand{\bp}{{\mathbf{p}}}

\newcommand{\CB}{{\mathcal{B}}}
\newcommand{\CalD}{{\mathcal{D}}}
\newcommand{\CF}{{\mathcal{F}}}
\newcommand{\CG}{{\mathscr{G}}}

\newcommand{\CL}{{\mathcal{L}}}
\newcommand{\CM}{{\mathcal{M}}}

\newcommand{\CO}{{\mathcal{O}}}

\newcommand{\CR}{{\mathcal{R}}}
\newcommand{\cS}{{\mathcal{S}}}

\newcommand{\CW}{{\mathcal{W}}}
\newcommand{\oW}{\overline{\mathcal{W}}{}}
\newcommand{\iso}{\overset{\sim}{\longrightarrow}}

\newcommand{\sz}{{\mathsf{z}}}
\newcommand{\sy}{{\mathsf{y}}}

\newcommand{\sw}{{\mathsf{w}}}

\newcommand{\unl}{\underline}
\newcommand{\ol}{\overline}
\newcommand{\on}{\operatorname}
\newcommand{\wit}{\widetilde}
\newcommand{\oZ}{\mathring{Z}}
\newcommand{\bZ}{\mathbullet{Z}}
\newcommand{\oE}{\mathring{E}}
\newcommand{\oA}{\mathring{\mathbb A}}
\newcommand{\oG}{\mathring{\mathbb G}}

\newcommand{\alphavee}{\alpha^{\!\scriptscriptstyle\vee}}
\newcommand{\tilscA}{\scA}

\newcommand{\bz}{\mathbf z}
\newcommand{\bv}{\mathbf v}
\newcommand{\bw}{\mathbf w}

\newcommand{\dil}{\mathrm{dil}}
\newcommand{\ialpha}{\alpha}
\newcommand{\iialp}{\mathsf n}
\newcommand{\ibeta}{\beta}
\newcommand{\bpi}{{\boldsymbol{\pi}}}

\newcommand{\vk}{\varkappa}

\begin{document}

\title[Line bundles on Coulomb branches]{Line bundles on Coulomb branches
}
\author[A.~Braverman]{Alexander Braverman}
\address{
Department of Mathematics,
University of Toronto and Perimeter Institute of Theoretical Physics,
Waterloo, Ontario, Canada, N2L 2Y5
}
\email{braval@math.toronto.edu}
\author[M.~Finkelberg]{Michael Finkelberg}
\address{National Research University Higher School of Economics,
  Russian Federation, Department of Mathematics, 6 Usacheva st, Moscow 119048;
  Skolkovo Institute of Science and Technology; 
Institute for Information Transmission Problems}
\email{fnklberg@gmail.com}
\author[H.~Nakajima]{Hiraku Nakajima}
\address{Research Institute for Mathematical Sciences,
Kyoto University, Kyoto 606-8502,
Japan}
\email{nakajima@kurims.kyoto-u.ac.jp}
\curraddr{Kavli Institute for the Physics and Mathematics of the Universe
  (WPI), The University of Tokyo,
  5-1-5 Kashiwanoha, Kashiwa, Chiba, 277-8583,
  Japan
}
\email{hiraku.nakajima@ipmu.jp}

\subjclass[2000]{}
\begin{abstract}
  This is the third companion paper of \cite{main}. When a gauge
  theory has a flavor symmetry group, we construct a partial
  resolution of the Coulomb branch as a variant of the definition. We
  identify the partial resolution with a partial resolution of a
  generalized slice in the affine Grassmannian, Hilbert scheme of
  points, and resolved Cherkis bow variety for a quiver gauge theory
  of type $ADE$ or affine type $A$.
\end{abstract}

\maketitle

\setcounter{tocdepth}{2}

\section{Introduction}

Let $G$ be a complex reductive group and $\bM$ its symplectic
representation of a form $\bN\oplus\bN^*$. ($\bN$ will be fixed
hereafter.) In \cite{2015arXiv150303676N,main} we gave a
mathematically rigorous definition of the Coulomb branch of a $3d$
$\mathcal N=4$ gauge theory associated with $(G,\bM)$ as follows. We
introduce an infinite dimensional variety $\cR = \cR_{G,\bN}$ (the
variety of triples), and define a convolution product on its
$G_\cO = G[[z]]$-equivariant homology $H^{G_\cO}_*(\cR)$, which is
commutative. Then we define the Coulomb branch
$\cM_C\equiv \cM_C(G,\bN)$ as the spectrum of $H^{G_\cO}_*(\cR)$. It
is an affine algebraic variety.

Suppose that we have a flavor symmetry, i.e.\ $\bN$ is a
representation of a larger group $\tilde G$ containing $G$ as a normal
subgroup. We further assume $G_F := \tilde G/G$ is a torus. Then we
can consider the Coulomb branch $\cM_C(\tilde G,\bN)$ for the larger
group $\tilde G$. We showed that the original $\cM_C(G,\bN)$ is the
Hamiltonian reduction $\cM_C(\tilde G,\bN)\tslash G_F^\vee$
of $\cM_C(\tilde G,\bN)$ by the dual torus $G_F^\vee$, see
\ref{prop:reduction}.
See \cite[\S5]{2015arXiv150303676N} for a motivation of this
statement, and references in physics literature.
Since $\cM_C(G,\bN)$ is a hamiltonian reduction by a torus, one can
take the reduction at a different value of the moment map, or can
consider a GIT quotient $\cM_C^\vk(G,\bN)$ with respect to a
stability condition, which is a character
$\vk\colon G_F^\vee\to\CC^\times$.
The former gives a deformation of $\cM_C(G,\bN)$ parametrized by
$\Spec H^*_{G_F}(\mathrm{pt})$. The latter gives a quasi-projective
variety $\cM_C^\vk(G,\bN)$ equipped with a projective morphism
$\bpi\colon\cM_C^\vk(G,\bN)\to\cM_C(G,\bN)$. This is birational.
See \ref{rem:birational} below.

We could understand this construction as follows. (See
\ref{subsec:flav-symm-group2}.) Let us denote the variety of triples
for the larger group $(\tilde G,\bN)$ by $\tilde\cR$. Let $\tilde\pi$
be the natural projection $\tilde\cR\to \Gr_{G_F}$. We identify
$\Gr_{G_F}$ with the coweight lattice of $G_F$, which is the weight
lattice of $G_F^\vee$. For a coweight $\vk$ of $G_F$, the
inverse image $\pi^{-1}(\vk)$ is denoted by
$\tilde\cR^\vk$. (In \ref{subsec:flav-symm-group2} a coweight
was denoted by $\lambda_F$.) Note that $\tilde\cR^0$ is nothing but
the original variety of triples $\cR$.
The convolution product defines a multiplication
\begin{equation*}
  H^{G_\cO}_*(\tilde\cR^\vk)\otimes_\CC
  H^{G_\cO}_*(\tilde\cR^{\vk'}) \to
  H^{G_\cO}_*(\tilde\cR^{\vk+\vk'}).
\end{equation*}
In particular $H^{G_\cO}_*(\tilde\cR^\vk)$ is an
$H^{G_\cO}_*(\cR)$-module, hence defines a sheaf on
$\cM_C(G,\bN)= \Spec(H^{G_\cO}_*(\cR))$. We only take coweights in
$\ZZ_{\ge 0}\vk$ for a fixed $\vk$, and consider
$\Proj(\bigoplus_{n\ge 0} H^{G_\cO}_*(\tilde\cR^{n\vk}))$. This is
nothing but the GIT quotient $\cM_C^\vk(G,\bN)$.
It is a quasi projective variety, equipped with a natural projective
morphism $\bpi\colon\cM_C^\vk(G,\bN)\to\cM_C(G,\bN)$.
We have
$H^{G_\cO}_*(\tilde\cR^\vk) =
\Gamma(\cM_C(G,\bN),\bpi_*\shfO_{\cM_C^\vk(G,\bN)}(1))$.

In this paper, we study $\cM^\vk_C(G,\bN)$ for a framed quiver
gauge theory of type $ADE$ or affine $A$. The original Coulomb branch
$\cM_C(G,\bN)$ was identified with a generalized slice in the affine
Grassmannian \cite{blowup}, and a Cherkis bow variety
\cite{2016arXiv160602002N} respectively. In both cases the variety has
a natural partial resolution (actual resolution for type $A$ or affine
type $A$), and we identify it with the GIT quotient.

The paper is organized as follows. In \ref{sec:mult} we show that the
multiplication on $\bigoplus H^{G_\cO}_*(\tilde\cR^{n\vk})$ is equal
to one given by the tensor product of line bundles for a framed quiver
gauge theory of type $A_1$. This case was studied in detail in
\ref{line Klein}$\sim$\ref{Klein via}, and this section is its
supplement.
In \ref{hilbert} we show that the determinant line bundle on the
Hilbert scheme of points in $\BA^2$ arises in our construction.
In \ref{line Cherkis} we study the Coulomb branch of a framed quiver
gauge theory of affine type $A$ and identify our construction of a
partial resolution with a bow variety with an appropriate stability
condition.
In \ref{det slice} we study the Coulomb branch of a framed quiver
gauge theory of type $ADE$ and identify our construction of a partial
resolution with a convolution diagram over a generalized slice in the
affine Grassmannian.

\begin{Remark}\label{rem:birational}
  Let us show that $\bpi$ is birational. By \ref{sec:z} we can replace
  the representation $\bN$ by $0$. Thus we need to compare
  $\cM_C(\tilde G,0)\tslash_{\varkappa} G_F^\vee$ and
  $\cM_C(G,0)$. Note that we have a finite covering $G'_F$ of $G_F$
  such that the corresponding covering of $\tilde G$ becomes the
  product $G\times G'_F$. Moreover we can replace $\varkappa$ by its
  positive power, hence we may assume it lifts to $G'_F$. Then we get
  $\cM_C(G\times G_F,0)\tslash_{\varkappa} G_F^\vee = \cM_C(G,0)
  \times \cM_C(G_F,0) \tslash_{\varkappa} G_F^\vee$, which is
  obviously $\cM_C(G,0)$.
\end{Remark}

\subsection*{Notation}

We basically follow the notation in \cite{main}, \cite{blowup} and
\cite{affine}.


\subsection*{Acknowledgments}

We thank
M.~Bershtein,
R.~Bezrukavnikov,
D.~Gaitsgory,
and
A.~Ob\-lomkov
for the useful discussions.
%

A.B.\  was partially supported by the NSF grant DMS-1501047.
M.F.\ was partially funded within the framework of the HSE University Basic Research Program and
the Russian Academic Excellence Project `5-100'.
The research of H.N.\ is supported by JSPS Kakenhi Grant Numbers
25220701, 
16H06335. 
A part of this work was done while H.N.\ was in residence at the
Mathematical Sciences Research Institute in Berkeley, California
during the semester of 2018 Spring with support by the National Science
Foundation under Grant No.~140140.

\section{Multiplication morphism}\label{sec:mult}

This section is a supplement to \ref{line Klein}$\sim$\ref{Klein via}.

Let $N$ be an integer greater than $1$. Let $\cS_N$ denote the
hypersurface $ZY=W^N$ in $\BA^3$, $\bpi\colon\wit\cS_N\to\cS_N$ its
minimal resolution, and $\cS_N^\circ := \cS_N\setminus \{0\}$.
We change $z$, $y$, $w$ to capital letters to avoid a confusion later.
A weight $\lambda$ of $\SL(N)$ defines a line bundle $\cL_\lambda$
over $\wit\cS_N$. Let $\cF_\lambda$ denote the torsion free sheaf
$\bpi_* \cL_\lambda$ on $\cS_N$ for dominant $\lambda$.  (To be
consistent with other parts of this paper, we should denote a weight
by $\vk$, but we keep notation in \ref{monopole}.)
Let us recall the notation briefly. We identify $\cS_N$ with
$\BA^2\dslash (\ZZ/N\ZZ)$, where $\zeta\in\ZZ/N\ZZ$ acts on $\BA^2$ by
$\zeta\cdot(u,v) = (\zeta u, \zeta^{-1}v)$. We have $W=uv$, $Z=u^N$,
$Y=v^N$. The line bundle $\cL_{\omega_i}$ for a fundamental root
$\omega_i$ is defined so that
$\Gamma(\wit\cS_N,\cL_{\omega_i}) = \Gamma(\cS_N,\cF_{\omega_i})$ is
the space of the semi-invariants $\CC[\BA^2]^{\chi_i}$ with
$\chi_i(\zeta) = \zeta^i$ ($i=1,\dots,N-1$). If we identify a weight
$\lambda$ of $\SL(N)$ with $(\lambda_1\ge\dots\ge\lambda_N)$ up to
simultaneous shifts of all $\lambda_i$, we have
$\cL_\lambda = \bigotimes
\cL_{\omega_i}^{\otimes(\lambda_i-\lambda_{i+1})}$.

We realize $\cS_N$ and $\wit\cS_N$ as Coulomb branches as follows: $V$
with $\dim V = 1$, $W$ with $\dim W = N$, $G = \GL(V) = \CC^\times$,
$\tilde G = (\GL(V)\times T(W))/\CC^\times$, where $T(W)$ is a maximal torus of
$\GL(W)$ consisting of diagonal matrices, $\CC^\times$ is the diagonal scalar
subgroup, $G_F = T(W)/\CC^\times$, and $\bN = \Hom(W,V)$. Then
$\cM_C(G,\bN)$ is $\cS_N$ and
$\Gamma(\cS_N,\cF_\lambda) \cong H^{G_\cO}_*(\tilde\cR^\lambda)$. Note that
$H^{G_\cO}_*(\tilde\cR^\lambda)$ is denoted by
$i^!_\lambda\tilscA^{\on{for}}$ in \ref{monopole}, as it is a costalk
of a ring object $\tilscA^{\on{for}}$ at $\lambda$.

We choose isomorphisms 
$\Gamma(\cS_N,\CF_\lambda)\iso H^{G_\cO}_*(\tilde\cR^\lambda)
\begin{NB}
i^!_\lambda\tilscA^{\on{for}}  
\end{NB}%
$ for any $\lambda$ (defined uniquely up to multiplication by a
scalar).

\begin{Lemma}
\label{tensor surj}
The multiplication morphism $\Gamma(\cS_N,\CF_\lambda)\otimes\Gamma(\cS_N,\CF_\mu)
\to\Gamma(\cS_N,\CF_{\lambda+\mu})$ \textup(resp.\
$H^{G_\cO}_*(\tilde\cR^\lambda)\otimes H^{G_\cO}_*(\tilde\cR^\mu)
\to
H^{G_\cO}_*(\tilde\cR^{\lambda+\mu})$
\begin{NB}
$i^!_\lambda\tilscA^{\on{for}}
\otimes i^!_\mu\tilscA^{\on{for}}\to i^!_{\lambda+\mu}\tilscA^{\on{for}}$  
\end{NB}%
\textup) is surjective
for any dominant $\lambda,\mu$.
\end{Lemma}

\begin{proof}
  It suffices to consider the case
  $\mu=\omega_n=(1,\ldots,1,0,\ldots,0)$ ($n$ 1's) for
  $1\le n\le N-1$.
  \begin{NB}
    Added by HN, May 4.
  \end{NB}%
Recall that the $\BC^\times\times\BC^\times$-character of 
$\Gamma(\cS_N,\CF_{\lambda+\mu})$ given by~\ref{monopole Klein} is multiplicity
free. So it suffices to represent each summand
$x^{\sum_{i=1}^N((\lambda+\omega_n)_i-m)}t^{\sum_{i=1}^N|(\lambda+\omega_n)_i-m|}$ as a product
of two summands $x^{\sum_{i=1}^N(\lambda_i-m')}t^{\sum_{i=1}^N|\lambda_i-m'|}$ and
$x^{\sum_{i=1}^N((\omega_n)_i-m'')}t^{\sum_{i=1}^N|(\omega_n)_i-m''|}$. Now if 
$m\geq\lambda_n+1$, we take $m'=m-1,\ m''=1$, and if $m\leq\lambda_n$, we take
$m'=m,\ m''=0$.
\begin{NB}
  If $m\ge\lambda_n+1$, we have $m\ge\lambda_i+1$ for $n\le i\le
  N-1$. Therefore
  $|\lambda_i - (m-1)| = m - 1 - \lambda = |\lambda_i - m| - 1$ for
  $n \le i \le N-1$. So we miss $N-1-n$. It is compensated by
  $|(\omega_n)_i-1|$.

  If $m\le\lambda_n$, we have $m\le \lambda_i$ for $1\le i\le n$. Then
  $|\lambda_i-m| = |\lambda_i + 1 - m| - 1$. Hence we miss $n$. It is
  compensated by $|(\omega_n)_i|$.
\end{NB}%
The same argument works for $H^{G_\cO}_*(\tilde\cR^?)$
\begin{NB}
the costalks of $\tilscA^{\on{for}}$
\end{NB}%
due to the monopole formula.
\begin{NB}
  HN on May 4: It is implicitly used that the multiplication is
  $\CC^\times\otimes\CC^\times$-equivariant, and the image of
  $\text{nonzero}\otimes\text{nonzero}$ is nonzero. For example, if
  the multiplication morphism is just zero. Where do you find a
  contradiction in your argument ? The following explanation is fine.
\end{NB}%
Indeed, the morphism $H^{G_\cO}_*(\tilde\cR^\lambda)\otimes H^{G_\cO}_*(\tilde\cR^\mu)
\to H^{G_\cO}_*(\tilde\cR^{\lambda+\mu})$ respects the bigrading. And
the induced morphism $H^{G_\cO}_*(\tilde\cR^\lambda)\otimes_{H^{G_\CO}_*(\CR)}
H^{G_\cO}_*(\tilde\cR^\mu)\to H^{G_\cO}_*(\tilde\cR^{\lambda+\mu})$ is an isomorphism
generically due to the localization theorem.
\end{proof}

\begin{Lemma}
\label{tensor Klein}
The diagram 
$$\begin{CD}
\Gamma(\cS_N,\CF_\lambda)\otimes_{\BC[\cS_N]}\Gamma(\cS_N,\CF_\mu) @>\sim>>
H^{G_\cO}_*(\tilde\cR^\lambda)\otimes_{H^{G_\cO}_*(\cR)} H^{G_\cO}_*(\tilde\cR^\mu) \\
@VVV @VVV \\
\Gamma(\cS_N,\CF_{\lambda+\mu}) @>\sim>> H^{G_\cO}_*(\tilde\cR^{\lambda+\mu})
\end{CD}
\begin{NB}
\begin{CD}
\Gamma(\cS_N,\CF_\lambda)\otimes_{\BC[\cS_N]}\Gamma(\cS_N,\CF_\mu) @>\sim>>
i^!_\lambda\tilscA^{\on{for}}\otimes_{i^!_0\tilscA^{\on{for}}}i^!_\mu\tilscA^{\on{for}} \\
@VVV @VVV \\
\Gamma(\cS_N,\CF_{\lambda+\mu}) @>\sim>> i^!_{\lambda+\mu}\tilscA^{\on{for}}
\end{CD}
\end{NB}%
$$
commutes up to multiplication by a scalar for any dominant $\lambda,\mu$.
\end{Lemma}

\begin{proof}
The kernels of both vertical morphisms coincide with the torsion in the upper
row. Thus it suffices to check the claim generically. But generically all the
four modules in question are free of rank 1. So it suffices to check the
commutativity for a single $\BC^\times\times\BC^\times$-eigensection of 
$\Gamma(\cS_N,\CF_\lambda)\otimes_{\BC[\cS_N]}\Gamma(\cS_N,\CF_\mu)$, and this 
follows from the multiplicity free property of
$H^{G_\cO}_*(\tilde\cR^{\lambda+\mu})$.
\begin{NB}
$i^!_{\lambda+\mu}\tilscA^{\on{for}}$.
\end{NB}%
\end{proof}

\begin{Remark}\label{rem:ambiguity}
  At the end of \ref{Klein via}, we wrote down an explicit isomorphism
  $\Gamma(\cS_N,\cF_\lambda)\xrightarrow{\sim}
  H^{G_\cO}_*(\tilde\cR^\lambda)$ when $\lambda$ is a fundamental
  coweight $\omega_i$ as
  \begin{equation*}
     r^{(m,\omega_i)} \mapsto
     \begin{cases}
       v^{N-i} Y^{m-1} & \text{if $m > 0$},
       \\
       u^i Z^{-m} & \text{if $m\le 0$},
     \end{cases}
   \end{equation*}
   where $r^{(m,\omega_i)}$ (denoted by $r^m$ in \ref{Klein via}) is
   the fundamental cycle of the fiber of $\tilde\cR\to\Gr_{\tilde G}$
   over $(m,\underbrace{1,\dots,1}_{\text{$i$
       times}},\underbrace{0,\dots,0}_{\text{$N-i$ times}})$.
  \begin{NB}
    Comparison of the notation: $y_i^\ialpha$ is $v^{N-\ialpha}$, and
    $\sy_i^\ialpha$ is $r^{(1,\omega_i)}$ for $\iialp=\ialpha$, i.e.\ the
    fundamental class of the fiber over
    $(1,\underbrace{1,\dots,1}_{\text{$\ialpha$ times}},
    \underbrace{0,\dots,0}_{\text{$N-\ialpha$ times}})$.
  \end{NB}%
  Thanks to \ref{tensor Klein} we generalize it for general
  dominant $\lambda$ by products. Then \ref{tensor Klein} holds
  without ambiguity of a scalar under the generalized isomorphism.
  Namely it is characterized by
  $
  \bigotimes_{i=1}^{N-1} (r^{(1,\omega_i)})^{\otimes(\lambda_i - \lambda_{i+1})}
  \mapsto \bigotimes_{i=1}^{N-1}(v^{N-i})^{\otimes(\lambda_i - \lambda_{i+1})}.
  $
  By \ref{sec:abelian} the left hand side is nothing but the
  fundamental class over
  $(\lambda_1-\lambda_N,\lambda_1-\lambda_N,\lambda_2-\lambda_N,\dots,
  \lambda_{N-1}-\lambda_N,0) =
  (\lambda_1,\lambda_1,\lambda_2,\dots,\lambda_{N-1},\lambda_N)$ (the
  first entry corresponds to $\GL(V)$ of $\tilde G$ and others to
  $T(W)$).
\end{Remark}

\begin{Remark}
  We have another way\footnote{H.N.\ thanks Alexei Oblomkov for
    motivating him to considering this approach.} to understand
  $\cM_C^\varkappa(G,\bN)$. We identify
  $\tilde G = \CC^\times\times (\CC^\times)^N/\CC^\times$ with
  $(\CC^\times)^N$ by
  $(r,r_1,\dots,r_N)\bmod \CC^\times\mapsto
  (r_1/r,\dots,r_N/r)$.
  \begin{NB}
    The inverse is $(t_1,\dots,t_N)\to (1,t_1,\dots,t_N)\bmod \CC^\times$.
  \end{NB}%
  The projection $\tilde G\to G_F$ is just the quotient by the
  diagonal subgroup $\CC^\times$. Then $\bN\cong\CC^N$ is just the
  product of $N$ copies of the dual of the standard representation of
  $\CC^\times$, hence the Coulomb branch
  $\cM_C(\tilde G,\bN) \cong \CC^{2N}$. The action of
  $\pi_1(\tilde G)^\wedge$ is the $(\CC^\times)^N$-action on
  $\CC^{2N}$ given by
  $(s_1,\dots, s_N)\cdot (x_1,y_1,\dots,x_N, y_N) = (s_1 x_1, s_1^{-1}
  y_1,\dots, s_N x_N, s_N^{-1} y_N)$. See \ref{sec:abelian}. We note
  that
  $(\CC^\times)^{N-1} \cong \pi_1(G_F)^\wedge \to \pi_1(\tilde
  G)^\wedge \cong (\CC^\times)^N$ is given by
  $(t_1,\dots, t_{N-1})\mapsto (t_1,t_2/t_1,t_3/t_2,\dots,
  t_{N-1}/t_{N-2},1/t_{N-1})$. Hence $\cM_C(G,\bN)$ is the hamiltonian
  reduction of $\CC^{2N}$ by the action
  $(t_1 x_1, t_1^{-1} y_1, t_2/t_1 x_2, t_1/t_2 y_2, \dots,
  t_{N-1}^{-1} x_N, t_{N-1} y_N)$.
  This is nothing but a quiver variety of type $A_{N-1}$ with
  dimension vectors $\bv = (1,\dots,1)$, $\bw = (1,0,\dots,0,1)$,
  which is known to be $\cS_N$. It is also known that the GIT quotient
  gives a minimal resolution of $\cS_N$ such that the tautological
  line bundle for the $i$-th $\CC^\times$ is identified with
  $\cL_{\omega_i}$.
\end{Remark}

\section{Determinant line bundle on the Hilbert scheme}
\label{hilbert}
In this section we identify the determinant line bundle on the Hilbert
scheme $\on{Hilb}^n(\BA^2)$, or rather global sections of its
pushforward to $\on{Sym}^n\BA^2$, with the module over the Coulomb
branch of the Jordan quiver gauge theory arising from the construction
of \ref{subsec:flav-symm-group2}. (See also 
\ref{sec:sheav-affine-grassm}, though it is not essentially used.)

\subsection{Degree $2$}
\label{deg 2}
We consider the case of the Hilbert scheme $\on{Hilb}^2(\BA^2)$
of two points in this subsection. 
We have the dilatation action of $\BC^\times$ on
$\BA^2\colon t(u,v)=(t^{-1}u,t^{-1}v)$. It induces a $\BC^\times$-action on
$\on{Hilb}^2(\BA^2)$. The determinant line bundle $\CL$ on
$\on{Hilb}^2(\BA^2)$ carries a natural $\BC^\times$-equivariant structure.
We have $\on{Hilb}^2(\BA^2)\simeq\wit\cS_2\times\BA^2$, and
$\CL\simeq\CO_{\wit\cS_2}(1)\boxtimes\CO_{\BA^2}$. Hence,
from~\ref{monopole Klein}, for $l\in\BN$, the character of
$\Gamma(\on{Hilb}^2(\BA^2),\CL^l)$ equals
$$(1-t^2)^{-1}(1-t)^{-2}\sum_{m\in\BZ}t^{|l-m|+|m|}.$$

On the other hand, we consider
$G=\GL(V)=\GL(2),\ G_F=\BC^\times,\ \tilde{G}=G\times G_F$.
The $G=\GL(V)$-module $\bN=V\oplus\gl(V)$ carries a commuting dilatation
$G_F$-action; these two actions together give rise to the action of
$\tilde{G}$ on $\bN$. According to~\ref{prop:ad_taut}, the Coulomb branch
$\CM_C(G,\bN)$ is identified with $\on{Sym}^2(\BA^2)$.
Recall the setup of \ref{subsec:flav-symm-group2}.
(See also \ref{subsec:affG_flavor} and~\ref{subsec:line-bundles-via}.)
We consider the variety of triple $\tilde\cR$ for the larger group
$\tilde{G}$ and $\bN$, regarded as a representation of $\tilde G$. Let
$\tilde\pi\colon\tilde\cR\to\Gr_{G_F}$ be the projection. The affine
Grassmannian $\Gr_{G_F}$ is identified with $\BZ$. We denote the fiber
over $l$ by $\tilde\cR^l$. The fiber
$\tilde\cR^{0}$ over $0$ is nothing but the original variety of
triple $\cR$ whose equivariant Borel-Moore homology $H^{G_\cO}_*(\cR)$
is the coordinate ring of the Coulomb branch, i.e.\
$\BC[\on{Sym}^2(\BA^2)]=\BC[\on{Hilb}^2(\BA^2)]$ in this case. For
$l\in\BN\subset\BZ=\Gr_{G_F}$, the homology $H^{G_\cO}_*(\tilde\cR^l)$
is a module over $H^{G_\cO}_*(\cR)$, see
\secref{subsec:flav-symm-group2}.
\begin{NB}
We have the complex $\tilscA$ on
$\Gr_{G_F}$. The affine Grassmannian $\Gr_{G_F}$ is identified with $\BZ$.
The algebra $i^!_0\tilscA^{\on{for}}$ is nothing but the Coulomb branch
$H_*^{G_\CO}(\CR_{G,\bN})\simeq\BC[\on{Sym}^2(\BA^2)]=\BC[\on{Hilb}^2(\BA^2)]$.
For $l\in\BN\subset\BZ=\Gr_{G_F}$, the costalk $i^!_l\tilscA^{\on{for}}$ forms
a module over the algebra $i^!_0\tilscA^{\on{for}}$, see~\eqref{eq:47}.
\end{NB}%
We will denote the coherent sheaf on $\on{Sym}^2(\BA^2)$ associated to this
module by $\CG_l$.

We want to identify this module with
$\Gamma(\on{Hilb}^2(\BA^2),\CL^l)$. The module
$H_*^{G_\CO}(\tilde\cR^l)$ is nothing but the costalk
$i^!_l\tilscA^{\on{for}}$ in the setup in \ref{subsec:affG_flavor}.
\begin{NB}
The costalk $i^!_l\tilscA^{\on{for}}$ is nothing but
$H_*^{G_\CO}(\tilde\cR^l)$ where
$\tilde\pi\colon \CR_{\tilde G,\bN}\to\Gr_{G_F}$, see~\eqref{eq:52}.
\end{NB}
By the monopole formula~\eqref{eq:modified} for the character of
$H_*^{G_\CO}(\tilde\cR^l)$, we have
$$P_t^{\on{mod}}=(1-t^2)^{-2}\sum_{\lambda_1>\lambda_2\in\BZ}t^{-2|\lambda_1-\lambda_2|+
|\lambda_1-\lambda_2+l|+|\lambda_2-\lambda_1+l|+2l+|\lambda_1+l|+|\lambda_2+l|}+
(1-t^2)^{-1}(1-t^4)^{-1}\sum_{\lambda\in\BZ}t^{4l+2|\lambda+l|}.$$

\begin{Lemma}
\label{monopole Hilbert}
$P_t^{\on{mod}}=t^{2l}(1-t^2)^{-1}(1-t)^{-2}\sum_{m\in\BZ}t^{|l-m|+|m|}.$
\end{Lemma}

\begin{proof}
The sum in the RHS splits into 3 summands according to $m\leq0,\ 0<m\leq l,\ m>l$,
equal respectively, to $\frac{t^l}{1-t^2},\ lt^l,\ \frac{t^{l+2}}{1-t^2}$.
The second sum in the LHS splits into 2 summands according to
$\lambda\leq-l,\ \lambda>-l$, equal respectively, to
$\frac{t^{4l}}{1-t^2},\ \frac{t^{4l+2}}{1-t^2}$.
The first sum in the LHS splits into 6 summands according to
$-l\geq\lambda_1>\lambda_2,\ \lambda_1-\lambda_2\geq l$, or
$-l\geq\lambda_1>\lambda_2,\ \lambda_1-\lambda_2<l$, or
$\lambda_1>\lambda_2\geq-l,\ \lambda_1-\lambda_2\geq l$, or
$\lambda_1>\lambda_2\geq-l,\ \lambda_1-\lambda_2<l$, or
$\lambda_1>-l>\lambda_2,\ \lambda_1-\lambda_2\geq l$, or
$\lambda_1>-l>\lambda_2,\ \lambda_1-\lambda_2<l$.
These summands are equal respectively, to
$\frac{t^{3l}}{(1-t^2)(1-t)},\ \frac{t^{3l+1}(1-t^{l-1})}{(1-t^2)(1-t)},\
\frac{t^{3l}}{(1-t^2)(1-t)},\ \frac{t^{3l+1}(1-t^{l-1})}{(1-t^2)(1-t)},\
\frac{t^{3l}}{(1-t)^2}+\frac{(l-2)t^{3l}}{1-t},\
\frac{(l-2)t^{3l+1}}{1-t}-\frac{t^{3l+2}(1-t^{l-1})}{(1-t)^2}$.
Now a straightforward calculation finishes the proof.
\end{proof}

The evident action of $\BG_a^2$ on $\BA^2$ induces the natural free action of
$\BG_a^2$ on $\on{Sym}^2\BA^2$ such that $\BG_a^2\backslash\on{Sym}^2\BA^2=\cS_2$.
Moreover, we have a projection $\on{add}\colon \on{Sym}\BA^2\to\BA^2,\
((u_1,v_1),(u_2,v_2))\mapsto(u_1+u_2,v_1+v_2)$;
altogether we obtain an isomorphism $\on{Sym}^2\BA^2\iso\cS_2\times\BA^2$.

\begin{Proposition}
\label{costalk Hilbert}
Under the identification
$
\begin{NB}
  i^!_0\tilscA^{\on{for}}=
\end{NB}%
  H_*^{G_\CO}(\cR)\simeq\BC[\on{Sym}^2\BA^2]$, the
  \begin{NB}
    $i^!_0\tilscA^{\on{for}}$-module
  \end{NB}
  $H_*^{G_\CO}(\cR)$-module
$
\begin{NB}
i^!_l\tilscA^{\on{for}}=  
\end{NB}
H_*^{G_\CO}(\tilde\cR^l)$ is isomorphic
to the $\BC[\on{Sym}^2\BA^2]$-module $\Gamma(\on{Hilb}^2(\BA^2),\CL^l)$.
More precisely,

\textup{(a)} The restriction $\CG_l^\circ$ of $\CG_l$ to
$\cS_2^\circ\times\BA^2\subset\cS_2\times\BA^2=\on{Sym}^2\BA^2$
is a line bundle isomorphic to $\CL^l|_{\cS_2^\circ\times\BA^2}$.

\textup{(b)} An isomorphism in \textup{(a)} is
defined uniquely up to multiplication by a scalar.

\textup{(c)} An isomorphism in \textup{(a)} extends to an isomorphism
$
H^{G_\cO}_*(\tilde\cR^l)
\begin{NB}
= i^!_l\tilscA^{\on{for}}  
\end{NB}%
\iso\Gamma(\on{Hilb}^2(\BA^2),\CL^l)$.
\end{Proposition}

\begin{proof}
We consider the elements $E_1[1]$ and $F_1[1]$ of~\eqref{eq:79}
in $H_*^{G_\CO}(\cR)\simeq\BC[\on{Sym}^2\BA^2]$. They have degree $1/2$
with respect to the modified grading as in~\ref{discrepancy}(2),
see~\eqref{eq:Deltadeg}. Clearly, $E_1[1]=u_1+u_2,\ F_1[1]=v_1+v_2$. The
corresponding hamiltonian vector fields $H_{E_1[1]}$ and $H_{F_1[1]}$ on
$\CM_C=\on{Sym}^2\BA^2$ commute since the Poisson bracket $\{E_1[1],F_1[1]\}$
acts as multiplication by $2$ (the number of points), and its hamiltonian
vector field vanishes. The degrees of both $H_{E_1[1]}$ and $H_{F_1[1]}$ are $-1/2$
since the degree of the Poisson bracket is $-1$. Since the degrees of
$H^{G_\cO}_*(\cR)
\begin{NB}
  =
  i^!_0\tilscA^{\on{for}}
\end{NB}%
$ and $H^{G_\cO}_*(\tilde\cR^l)
\begin{NB}
=i^!_l\tilscA^{\on{for}}  
\end{NB}%
$ are all nonnegative by the
monopole formula, both $H_{E_1[1]}$ and $H_{F_1[1]}$ are locally nilpotent. Hence
they integrate to the action of $\BG_a^2$ on
$H^{G_\cO}_*(\cR)
\begin{NB}
  =
  i^!_0\tilscA^{\on{for}}
\end{NB}%
$ and $H^{G_\cO}_*(\tilde\cR^l)
\begin{NB}
=i^!_l\tilscA^{\on{for}}  
\end{NB}%
$. The action of $\BG_a^2$ on
$H^{G_\cO}_*(\cR)
\begin{NB}
=
i^!_0\tilscA^{\on{for}}  
\end{NB}%
=\BC[\on{Sym}^2\BA^2]$ comes from the action on
$\on{Sym}^2\BA^2$ discussed before the proposition. We conclude that the
coherent sheaf $\CG_l$ on $\on{Sym}^2\BA^2$ is
$\BG_m\ltimes\BG_a^2$-equivariant (the action of
$\BG_m$ comes from the modified grading).

In particular, $\CG_l$ is a pullback of a $\BG_m$-equivariant sheaf
$\CF_l$ on $\BG_a^2\backslash\on{Sym}^2\BA^2=\cS_2$. Both $\CG_l$ and
$\CF_l$ are generically of rank $1$; hence both $\CF_l|_{\cS_2^\circ}$
and $\CG_l^\circ:=\CG_l|_{\cS_2^\circ\times\BA^2}$ are line bundles. Recall that
$\on{Pic}(\cS_2^\circ)=\BZ/2\BZ$; the trivial line bundle is denoted
$\CF_{\bar0}$, and the nontrivial one is denoted $\CF_{\bar1}$ in accordance with
notations of~\ref{push Klein}. \ref{monopole Hilbert} and the argument in the
proof of~\ref{charac} show that $\CF_l|_{\cS_2^\circ}\simeq\CF_{\bar{l}}$, where
$\bar{l}=l\pmod{2}$. This proves (a), and the same argument as in the
proof of~\ref{charac} establishes (b).

For (c), we have to identify $\CF_l\subset j_*\CF_{\bar{l}}$
and $\CF_\lambda\subset j_*\CF_{\bar{l}}$ in notations
of~\ref{push Klein}, where $\lambda=(l,0)$. We start with $l=1$ case.
Then $\CF_\lambda=j_*\CF_{\bar{1}}$, and the character of
(the global sections of) $\CF_1$ coincides with the character of
$j_*\CF_{\bar{1}}$. Hence $\CF_1=j_*\CF_{\bar{1}}=\CF_\lambda$.

For $l>1$ we have to identify $\Gamma(\cS_2,\CF_l)$ inside
$\Gamma(\cS_2^\circ,\CF_{l\pmod{2}})$ with $\Gamma(\wit\cS_2,\CL_\lambda)=
\Gamma(T^*\BP^1,\CO(l))=\bigoplus_{k\geq0}\Gamma(\BP^1,\CO(l+2k))$.
However, the latter submodule is clearly characterized by its $t$-character
which coincides with the $t$-character of $\Gamma(\cS_2,\CF_l)$
by~\ref{monopole Hilbert}. Hence
$\Gamma(\on{Hilb}^2(\BA^2),\CL^l)=H_*^{G_\CO}(\tilde\CR{}^l)$.
\begin{NB}
  I cannot follow the last part of the argument. If we ignore the
  $x$-character, the module is no longer multiplicity free.
\end{NB}%
\end{proof}

\begin{NB}For $l>1$, the multiplication morphism
$\Gamma(\on{Hilb}^2(\BA^2),\CL^1)^{\otimes l}\to\Gamma(\on{Hilb}^2(\BA^2),\CL^l)$
is surjective. It suffices to check for $\CL=\CO(1)$ on $T^*\BP^1$ in place of
$\on{Hilb}^2(\BA^2)$. Then
$\Gamma(T^*\BP^1,\CO(l))=\bigoplus_{k\geq0}\Gamma(\BP^1,\CO(l+2k))$.
Clearly, the multiplication morphism
$$\left(\bigoplus_{k\geq0}\Gamma(\BP^1,\CO(l+2k))\right)\otimes
\left(\bigoplus_{j\geq0}\Gamma(\BP^1,\CO(1+2j))\right)\to
\bigoplus_{i\geq0}\Gamma(\BP^1,\CO(l+1+2i))$$ is surjective, and we are done
by induction.

This multiplication is compatible (up to a multiplicative constant) with
$(i^!_1\tilscA^{\on{for}})^{\otimes l}\to i^!_l\tilscA^{\on{for}}$.
\begin{NB2}
  Why ? HN
\end{NB2}%
In effect, $\Gamma(\on{Hilb}^2(\BA^2),\CL^1)\otimes_{\BC[\on{Sym}^2\BA^2]}\otimes
\ldots\otimes_{\BC[\on{Sym}^2\BA^2]}\Gamma(\on{Hilb}^2(\BA^2),\CL^1)=
\Gamma(\on{Hilb}^2(\BA^2),\CL^l)$, and hence
$i^!_l\tilscA^{\on{for}}\otimes_{\BC[\on{Sym}^2\BA^2]}\otimes
\ldots\otimes_{\BC[\on{Sym}^2\BA^2]}i^!_l\tilscA^{\on{for}}=
\Gamma(\on{Hilb}^2(\BA^2),\CL^l)$, but the automorphisms of the submodule
$\Gamma(\on{Hilb}^2(\BA^2),\CL^l)\subset\Gamma(\cS_2^\circ\times\BA^2,\CL^l)$
reduce to scalars.
\begin{NB2}
It follows that
$\Gamma(\on{Hilb}^2(\BA^2),\CL^l)\subset i^!_l\tilscA^{\on{for}}$
(inside $\Gamma(\cS_2,j_*\CF_{\bar{l}})\otimes\BC[\BA^2]$).  Now the
equality of characters implies
$\Gamma(\on{Hilb}^2(\BA^2),\CL^l)= i^!_l\tilscA^{\on{for}}$.
\end{NB2}%
\end{NB}%

Recall that the $\tilde{G}$-module $\bN=V\oplus\gl(V)$ splits as a direct sum.
If we set $'\bN=\gl(V)$, then from~\ref{rem:further} we obtain a homomorphism
$
\begin{NB}
i_0^!\tilscA^{\on{for}}=  
\end{NB}%
H_*^{G_\CO}(\cR)\hookrightarrow
H_*^{G_\CO}({}'\cR)
\begin{NB}
=i_0^!\,{}'\!\!\tilscA^{\on{for}}  
\end{NB}%
$ of algebras and a
compatible homomorphism of modules
$
\begin{NB}
i_l^!\tilscA^{\on{for}}=  
\end{NB}%
H_*^{G_\CO}(\tilde\cR^{l})\hookrightarrow
H_*^{G_\CO}({}'\tilde\cR^{l})
\begin{NB}
=i_l^!\,{}'\!\!\tilscA^{\on{for}}  
\end{NB}
$ (where ${}'\cR$, ${}'\tilde\cR$ are varieties of triples for $(G,{}'\bN)$,
$(\tilde G,{}'\bN)$ respectively, ${}'\tilde\cR^l$ is the fiber
of the projection ${}'\tilde\cR\to\Gr_{G_F}$ over $l$).
According to~\ref{prop:ad_taut}, the Coulomb branch $\CM_C(G,{}'\bN)$
is identified with $\on{Sym}^2(\cS_0)$, and the homomorphism
$H_*^{G_\CO}(\cR)\hookrightarrow
H_*^{G_\CO}({}'\cR)$
\begin{NB}
$i_0^!\tilscA^{\on{for}}\hookrightarrow i_0^!\,{}'\!\!\tilscA^{\on{for}}$
\end{NB}%
corresponds to the morphism
$\jmath^2\colon \on{Sym}^2(\cS_0)\hookrightarrow\on{Sym}^2(\BA^2)$ arising from
the open embedding
$\jmath\colon \cS_0\hookrightarrow\BA^2,\ (u,v)\mapsto(u,u^{-1}v),\ u\ne0$.
We denote by $'\!\CG_l$ the coherent sheaf on $\on{Sym}^2\cS_0$ associated to
the $H^{G_\cO}_*({}'\cR)
\begin{NB}
i_0^!\,{}'\!\!\tilscA^{\on{for}}  
\end{NB}%
$-module
$H^{G_\cO}_*({}'\tilde\cR^l))
\begin{NB}
i_l^!\,{}'\!\!\tilscA^{\on{for}}  
\end{NB}%
$. We would like to identify the coherent sheaves
$\on{pr}_*\CL^l$ and $'\!\CG_l$ on $\on{Sym}^2(\cS_0)$, where
$\on{pr}\colon \on{Hilb}^2(\cS_0)\to\on{Sym}^2(\cS_0)$ is the Hilbert-Chow
morphism. The localization of the morphism
$H^{G_\cO}_*(\tilde\cR^l)\hookrightarrow 
H^{G_\cO}_*({}'\tilde\cR^l)$
\begin{NB}
$i_l^!\tilscA^{\on{for}}\hookrightarrow i_l^!\,{}'\!\!\tilscA^{\on{for}}$
\end{NB}
factors through $\CG_l\hookrightarrow\jmath^2_*\jmath^{2*}\CG_l=
\jmath^2_*\on{pr}_*\CL^l\hookrightarrow\jmath^2_*{}'\!\CG_l$. The restriction of
the latter morphism to $\on{Sym}^2\cS_0$ is denoted by
$\theta\colon \on{pr}_*\CL^l\hookrightarrow{}'\!\CG_l$.

\begin{Corollary}
\label{Hilbert S0}
The morphism $\theta\colon \on{pr}_*\CL^l\hookrightarrow{}'\!\CG_l$ of
coherent sheaves on $\on{Sym}^2\cS_0$ is an isomorphism.
\end{Corollary}

\begin{proof}
Let $T\subset\GL(V)=\GL(2)$ be the diagonal torus with Lie algebra
$\ft\subset\gl(V)=\gl(2)$, with coordinates $w_1,w_2$.
The canonical projection
$\on{Sym}^2\BA^2=\CM_C(G,\bN)\to\ft/S_2=\on{Sym}^2\BA^1$ is the symmetric
square of the morphism $\BA^2\to\BA^1,\ (u,v)\mapsto uv$. The generalized
roots in $\ft^\vee$ are $w_1,w_2,w_1-w_2$.
We change the base to $\ft\to\ft/S_2$ and
localize at a general point $t$ of the diagonal $w_1-w_2=0$. The corresponding
fixed point sets coincide:
$({}'\tilde\cR^l)^t = (\tilde\cR^l)^t$;
\begin{NB}
$('\!\tilde\pi{}^{-1}(l))^t=(\tilde\pi{}^{-1}(l))^t$;  
\end{NB}%
hence $\theta$ is an isomorphism over the general points of diagonal.

Now let $t$ be a general point of the divisor $w_2=0$. Then the fixed point set
$({}'\tilde\cR^l)^t$ (resp.\ $(\tilde\cR^l)^t$)
\begin{NB}
$('\!\tilde\pi{}^{-1}(l))^t$ (resp.\ $(\tilde\pi{}^{-1}(l))^t$)   
\end{NB}%
splits as a
product $\Gr_{T_1}\times\Gr_{T_2}$ (resp.\ $\Gr_{T_1}\times\CR_{T_2,\bN'}$). Here
$T_1$ (resp.\ $T_2$) is a $1$-dimensional torus with coordinate $\sw_1$ (resp.\
$\sw_2$) with differential $w_1$ (resp.\ $w_2$), and $\bN'$ is the
$1$-dimensional representation of $T_2$ with character $\sw_2$. Note that the
flavor group disappeared since its action is absorbed into the action of $T_2$.
The morphism of localizations
$$\left(\BC[\ft_1\times T_1^\vee]\otimes\BC[\BA^2]\right)_t=
H_*^{T_\CO}((\tilde\cR^{l})^t)_t\to
H_*^{T_\CO}(({}'\tilde\cR^{l})^t)_t=
\left(\BC[\ft_1\times T_1^\vee]\otimes\BC[\cS_0]\right)_t$$
\begin{NB}
$$\left(\BC[\ft_1\times T_1^\vee]\otimes\BC[\BA^2]\right)_t=
\left(H_*^{T_\CO}(\tilde\pi{}^{-1}(l))^t\right)_t\to
\left(H_*^{T_\CO}('\!\tilde\pi{}^{-1}(l))^t\right)_t=
\left(\BC[\ft_1\times T_1^\vee]\otimes\BC[\cS_0]\right)_t$$
\end{NB}%
at the level of spectra is nothing but $(\on{id}\times\jmath)_t$.
The same argument takes care of the general points of the divisor $w_1=0$.
Hence the base change of $\theta$ is an isomorphism over the general points
of all the root hyperplanes. We conclude that $\theta$ is an isomorphism.
\end{proof}

\subsection{Factorization}
\label{factorization2}
The projection $\varpi_1\colon \cS_1=\BA^2\to\BA^1,\ (u,v)\mapsto w=uv$, induces
the projection $\varpi_n\colon\on{Hilb}^n(\cS_1)=
\on{Hilb}^n(\BA^2)\stackrel{\bpi_n}{\longrightarrow}\on{Sym}^n\BA^2
\stackrel{\varPi_n}{\longrightarrow}\on{Sym}^n\BA^1=\BA^{(n)}$.
The embedding $\BG_m\subset\BA^1$ induces the embedding
$\BG_m^{(n)}\subset\BA^{(n)}$. We denote by $\oG_m^{(n)}\subset\BG_m^{(n)}$ the
open subset formed by the complement to all the diagonals; we have a Galois
$S_n$-covering $\oG_m^n\to\oG_m^{(n)}$. We have
\begin{equation}
\label{gen fac}
\oG_m^n\times_{\oG_m^{(n)}}\varpi_n^{-1}(\oG_m^{(n)})=
\oG_m^n\times_{\oG_m^{(n)}}\varPi_n^{-1}(\oG_m^{(n)})=\oG_m^n\times\BG_m^n
\end{equation}
with coordinates $w_1,\ldots,w_n$ on the first factor, and $v_1,\ldots,v_n$
on the second factor. We denote the base change
$\BA^n\times_{\BA^{(n)}}\on{Hilb}^n(\BA^2)$ (resp.\
$\BA^n\times_{\BA^{(n)}}\on{Sym}^n\BA^2$) by $\unl{\on{Hilb}}^n(\BA^2)$ (resp.\
$\unl{\on{Sym}}^n\BA^2$).
We have factorization isomorphisms for $n=n'+n''$:
$$\unl{\on{Hilb}}^n(\BA^2)|_{(\BA^{n'}\times\BA^{n''})_{\on{disj}}}\iso
(\unl{\on{Hilb}}^{n'}(\BA^2)\times
\unl{\on{Hilb}}^{n''}(\BA^2))|_{(\BA^{n'}\times\BA^{n''})_{\on{disj}}},$$
$$\unl{\on{Sym}}^n\BA^2|_{(\BA^{n'}\times\BA^{n''})_{\on{disj}}}\iso
(\unl{\on{Sym}}^{n'}\BA^2\times
\unl{\on{Sym}}^{n''}\BA^2)|_{(\BA^{n'}\times\BA^{n''})_{\on{disj}}}$$
compatible with~\ref{gen fac}. By the definition of the determinant line
bundle, we also have the following factorization isomorphisms:
$$\left(\unl{\on{Hilb}}^n(\BA^2)|_{(\BA^{n'}\times\BA^{n''})_{\on{disj}}},\CL^l\right)\iso
\left((\unl{\on{Hilb}}^{n'}(\BA^2)\times
\unl{\on{Hilb}}^{n''}(\BA^2))|_{(\BA^{n'}\times\BA^{n''})_{\on{disj}}},
\CL^l\boxtimes\CL^l\right),$$
\begin{equation}
\label{fac}
\left(\unl{\on{Sym}}^n\BA^2|_{(\BA^{n'}\times\BA^{n''})_{\on{disj}}},\bpi_{n*}\CL^l\right)
\iso\left((\unl{\on{Sym}}^{n'}\BA^2\times
\unl{\on{Sym}}^{n''}\BA^2)|_{(\BA^{n'}\times\BA^{n''})_{\on{disj}}},
\bpi_{n'*}\CL^l\boxtimes\bpi_{n''*}\CL^l\right)
\end{equation}
compatible with the $S_n$-equivariant trivialization
\begin{equation}
\label{gen triv}
\left(\oG_m^n\times_{\oG_m^{(n)}}\varpi_n^{-1}(\oG_m^{(n)}),\CL^l\right)=
\left(\oG_m^n\times_{\oG_m^{(n)}}\varPi_n^{-1}(\oG_m^{(n)}),\bpi_{n*}\CL^l\right)=
\left(\oG_m^n\times\BG_m^n,\CO_{\overset{\circ}\BG{}_m^n\times\BG_m^n}\right)
\end{equation}
arising from the factorization and the identification
\begin{equation}
\label{n=1}
\left(\varpi_1^{-1}(\BG_m),\CL^l\right)=
\left(\varPi_1^{-1}(\BG_m),\bpi_{1*}\CL^l\right)=
\left(\BG_m\times\BG_m,\CO_{\BG_m\times\BG_m}\right).
\end{equation}
We will need the following particular case of the above factorization
isomorphisms:
$$\left((\BG_m^{(n-1)}\times\BA^1)_{\on{disj}}\times_{\BA^{(n)}}
\on{Hilb}^n(\BA^2),\CL^l\right)\iso$$
$$\left((\BG_m^{(n-1)}\times\BA^1)_{\on{disj}}
\times_{\BA^{(n-1)}\times\BA^1}(\on{Hilb}^{n-1}(\BA^2)\times\BA^2),
\CL^l\boxtimes\CL^l\right),$$
\begin{multline}
\label{origin fac}
\left((\BG_m^{(n-1)}\times\BA^1)_{\on{disj}}\times_{\BA^{(n)}}
\on{Sym}^n\BA^2,\bpi_{n*}\CL^l\right)\iso\\
\left((\BG_m^{(n-1)}\times\BA^1)_{\on{disj}}
\times_{\BA^{(n-1)}\times\BA^1}(\on{Sym}^{n-1}\BA^2\times\BA^2),
\bpi_{n-1,*}\CL^l\boxtimes\bpi_{1*}\CL^l\right).
\end{multline}

\subsection{Determinant sheaves via homology groups of fibers}
\label{det via}
We change slightly the setup of~\ref{deg 2}: we consider
$G=\GL(V)=\GL(n),\ G_F=\BC^\times,\ \tilde{G}=G\times G_F$.
The $G=\GL(V)$-module $\bN=V\oplus\gl(V)$ carries a commuting dilatation
$G_F$-action; these two actions together give rise to the action of
$\tilde{G}$ on $\bN$. According to~\ref{prop:ad_taut}, the Coulomb branch
$\CM_C(G,\bN)$ is identified with $\on{Sym}^n(\BA^2)$.
In this case we have
$H_*^{G_\CO}(\cR)\cong\BC[\on{Sym}^n(\BA^2)]=\BC[\on{Hilb}^n(\BA^2)]$,
see \ref{prop:ad_taut}.
\begin{NB}
The algebra $i^!_0\tilscA^{\on{for}}$ is nothing but the
Coulomb branch
$H_*^{G_\CO}(\cR)\cong\BC[\on{Sym}^n(\BA^2)]=\BC[\on{Hilb}^n(\BA^2)]$.
\end{NB}%
For $l\in\BN\subset\BZ=\Gr_{G_F}$, $H^{G_\cO}_*(\tilde\cR^l)$
\begin{NB}
the costalk $i^!_l\tilscA^{\on{for}}$  
\end{NB}%
forms
a module over the algebra $H^{G_\cO}_*(\cR)$ as in the case $n=2$,
\begin{NB}
  $i^!_0\tilscA^{\on{for}}$,
\end{NB}%
and we want to identify this
module with
$\Gamma(\on{Hilb}^n(\BA^2),\CL^l)=\Gamma(\on{Sym}^n\BA^2,\bpi_{n*}\CL^l)$.
\begin{NB}
The costalk $i^!_l\tilscA^{\on{for}}$ is nothing but
$H_*^{G_\CO}(\tilde\cR^l)$ where
$\tilde\pi\colon \CR_{\tilde G,\bN}\to\Gr_{G_F}$, see~\eqref{eq:52}.
\end{NB}%
Recall that $\on{Spec}H^*_{G_\CO}(\on{pt})=\BA^{(n)}\leftarrow\BA^n=
\on{Spec}H^*_{T_\CO}(\on{pt})$, and the base change under
$\BA^{(n)}\leftarrow\BA^n$ gives $H_*^{T_\CO}(\tilde\cR^l)$, where $T$
is a Cartan torus of $G$.
If we further localize to $\oG_m^n\subset\BA^n$, we have a localization
isomorphism $\bz^*\iota_*^{-1}\colon H_*^{T_\CO}(\tilde\cR^l)_{\on{loc}}\iso
H_*^{T_\CO}(\hat\pi^{-1}(l))_{\on{loc}}$ where
$\hat\pi\colon\Gr_{T\times G_F}\to\Gr_{G_F}$ is the obvious projection.
But $H_*^{T_\CO}(\hat\pi^{-1}(l))\cong H_*^{T_\CO}(\Gr_T)=\BC[\BA^n\times\BG_m^n]$
by~\ref{rem:W-cover}(2). All in all, we obtain an $S_n$-equivariant trivialization
\begin{equation}
  \label{costalk triv}
  H^{T_\cO}_*(\tilde\cR^l)
\begin{NB}
  = i^!_l\tilscA^{\on{for}}|_{\overset{\circ}\BG{}_m^n}  
\end{NB}%
\cong
\CO_{\overset{\circ}\BG{}_m^n\times\BG_m^n}.
\end{equation}
Composing with the trivialization~\ref{gen triv}, we obtain a rational
isomorphism of $\BC[\on{Sym}^n\BA^2]$-modules
$\theta\colon \Gamma(\on{Sym}^n\BA^2,\bpi_{n*}\CL^l)\dasharrow
H^{G_\cO}_*(\tilde\cR^l)
\begin{NB}
= i^!_l\tilscA^{\on{for}}  
\end{NB}%
$.

\begin{Theorem}
\label{det via hom}
The rational isomorphism $\theta\colon \Gamma(\on{Sym}^n\BA^2,\bpi_{n*}\CL^l)
\dasharrow H^{G_\cO}_*(\tilde\cR^l)
\begin{NB}
=  i^!_l\tilscA^{\on{for}}
\end{NB}%
$ extends to the regular isomorphism of
$\BC[\on{Sym}^n\BA^2]$-modules
$\theta\colon \Gamma(\on{Sym}^n\BA^2,\bpi_{n*}\CL^l)\iso
H^{G_\cO}_*(\tilde\cR^l)
\begin{NB}
=  i^!_l\tilscA^{\on{for}}
\end{NB}%
$.
\end{Theorem}

\begin{proof}
We follow the standard scheme, see e.g.\ the proof of~\ref{Coulomb_quivar}.
We have to check that $\theta$ extends through the general points of the
boundary divisor $\BA^n\setminus\oG_m^n$. If a point lies on a diagonal divisor
$w_r=w_s$, we are reduced by localization and factorization~\ref{fac}
to~\ref{Hilbert S0}.
\begin{NB}
  HN: I do not see why $\bz^*\iota_*^{-1}$ is the same as the rational
  isomorphism $\theta$ in \ref{Hilbert S0}.
\end{NB}%
If a point lies
on a coordinate hyperplane $w_r=0$, we are reduced by localization and
factorization~\ref{origin fac},~\ref{n=1} to the evident case $n=1$.
We conclude by an application of~\ref{prop:flat} and~\ref{rem:conditions}.
The condition $\varPi_{n*}\bpi_{n*}\CL^l\iso
j_*\varPi_{n*}\bpi_{n*}\CL^l|_{\on{Hilb}^n(\BA^2)^\bullet}$
of~\ref{rem:conditions} is satisfied since the complement of
$\on{Hilb}^n(\BA^2)^\bullet$
in $\on{Hilb}^n(\BA^2)$ is of codimension $2$. The latter claim follows from
the semismallness of $\bpi_n\colon \on{Hilb}^n(\BA^2)\to\on{Sym}^n\BA^2$.
\end{proof}

\section{Line bundles on Cherkis bow varieties}
\label{line Cherkis}

We can modify the proof of the last section to the case of quiver gauge
theories of affine type $A_{n-1}$ replacing Hilbert schemes by Cherkis
bow varieties, and using results in \cite{2016arXiv160602002N}.
%
%
We use the notation in \cite{2016arXiv160602002N}, hence we assume the reader
is familiar with it.

\subsection{Resolution for bow varieties}
\label{resolution bow}

Given dimension vectors $\underline{\bv} = (\bv_0,\dots,\bv_{n-1})$,
$\underline{\bw} = (\bw_0,\dots,\bw_{n-1})$ we consider
\begin{equation*}
  G \equiv \GL(\underline{\bv}) \defeq \prod_{i=0}^{n-1} \GL(\bv_i), \quad
  \bN \equiv \bN(\underline{\bv},\underline{\bw}) =
  \bigoplus_{i=0}^{n-1} \Hom(\CC^{\bv_i}, \CC^{\bv_{i+1}}) \oplus
  \Hom(\CC^{\bw_i}, \CC^{\bv_i}) 
\end{equation*}
with the natural $G$-action on $\bN$. Let
$\ell = \sum_{i=0}^{n-1} \mathbf w_i$. The Coulomb branch
$\cM_C(G,\bN)$ is isomorphic to a bow variety
$\cM(\underline{\bv},\underline{\bw})$ with a balanced condition,
defined as in \cite[\S2.2]{2016arXiv160602002N}. The definition of
\cite[\S2.2]{2016arXiv160602002N} is more general: we have parameters
$\vk_\sigma\in\QQ$ ($\sigma=1,\dots,\ell$) of the stability condition
for the GIT quotient, where $\cM_C(G,\bN)$ corresponds to the case
$\vk_\sigma = 0$ for $\sigma=1,\dots,\ell$.%
\footnote{%
  It was denoted by $\nu_\sigma^\RR$ in \cite{2016arXiv160602002N}, as we also
  have complex parameters $\nu^\CC = (\nu^\CC_\sigma)_\sigma$, which
  we set $0$ for brevity here.
}
%
%
We have a $\QQ$-line bundle 
from the construction, which is an actual line bundle if
$\vk_\sigma\in\ZZ$ for $\sigma=1,\dots,\ell$. We suppose
$\vk_\sigma\in\ZZ$ hereafter.

There is one more extra parameter $\vk_*\in\ZZ$, which was not
explicitly explained in \cite{2016arXiv160602002N}.\footnote{In the
  original description \cite[\S2.1]{2016arXiv160602002N} of bow
  varieties as solutions of Nahm's equations, parameters $\vk_\sigma$,
  $\vk_*$ are put as the level of real part of the hyper-K\"ahler
  moment map.} It corresponds to the quotient where either one of the
stability conditions (C-S1) or (C-S2) is required in
\cite[Prop.~6.4]{2016arXiv160602002N}.

Let us number vector spaces appearing in the definition of bow
varieties as in \cite[\S6.1]{2016arXiv160602002N}.
\begin{align*}
&\xymatrix@C=1em{ & V_{i-1}^{\bw_{i-1}} \ar@(ur,ul)_{B_{i-1}} \ar[rr]^{A_{i-1}} \ar[dr]_{b_{i-1}} \ar@{.}[l] && V_{i}^0  \ar@(ur,ul)_{B'_{i}} \ar@<-.5ex>[rr]_{C_{1,i}} && V_{i}^1 \ar@<-.5ex>[ll]_{D_{1,i}} \ar@<-.5ex>[rr]_{C_{2,i}} && \ar@<-.5ex>[ll]_{D_{2,i}} \ar@{.}[r] & \ar@<-.5ex>[rr]_{C_{\bw_i-1,i}} && V_{i}^{\bw_i-1} \ar@<-.5ex>[ll]_{D_{\bw_i-1,i}} \ar@<-.5ex>[rr]_{C_{\bw_i,i}} && V_i^{\bw_i} \ar@<-.5ex>[ll]_{D_{\bw_i,i}} \ar@(ur,ul)_{B_{i}} \ar[rr]^{A_{i}} \ar[dr]_{b_{i}} && V_{i+1}^0  \ar@(ur,ul)_{B'_{i+1}} \ar@{.}[r] & \\
 & & \CC \ar[ur]_{a_{i}} & && && & && && & \CC \ar[ur]_{a_{i+1}} & & }.
\end{align*}
In particular, $\sigma$ ($\sigma=1,\dots,\ell$) is indexed as $(\ialpha,i)$
($i=0,\dots,n-1$, $\ialpha=1,\dots,\bw_i$). We introduce the character
corresponding to parameters\footnote{%
  We consider the `corresponding' complex parameter $\nu^\CC_*$ in
  \cite[6.2]{2016arXiv160602002N}, but we put it for all $i$. But the sum over
  $i$ only matters, so our $\vk_*$ should be compared with
  $n\nu^\CC_*$.} $\vk_*$, $\vk_{\ialpha,i}$ by
\begin{multline}\label{eq:bow1}
  \prod_{i=0}^{n-1}
    (\det V_i^1)^{-\vk_{1,i} + \vk_{2,i}} \cdots
    (\det V_i^\ialpha)^{- \vk_{\ialpha,i} + \vk_{\ialpha+1,i}} \cdots
    \\
    (\det V_i^{\bw_i-1})^{-\vk_{\bw_i-1,i} + \vk_{\bw_i,i}}
    (\det V_i^{\bw_i})^{- \vk_{\bw_i,i} + \vk_{1,i+1} + \delta_{i+1,0}\vk_*}.
\end{multline}
\begin{NB}
  another possible choice :
  \begin{equation*}
  (\det
  V_i^0)^{\delta_{i0}\vk_*-\vk_{\bw_{i-1},i-1}+\vk_{1,i}} (\det
  V_i^1)^{-\vk_{1,i} + \vk_{2,i}} \cdots (\det
  V_i^\ialpha)^{-\vk_{\ialpha,i} + \vk_{\ialpha+1,i}} \cdots (\det
  V_i^{\bw_i-1})^{-\vk_{\bw_i-1,i} + \vk_{\bw_i,i}}.
  \end{equation*}
  \end{NB}%
Note that the simultaneous shift
$\vk_{\ialpha,i}\mapsto \vk_{\ialpha,i}+s$, while keeping $\vk_*$, is
irrelevant.
\begin{NB}
  The convention in \cite[\S2.2(c)]{2016arXiv160602002N} is
  $\chi(g_\zeta) = \prod \det (g_\zeta)^{-\nu^\RR_\zeta}$. The
  numerical criterion \cite[\S2.4.2]{2016arXiv160602002N} is written in this
  convention.
\end{NB}%

\begin{NB3}
  The assumption is moved to here.
\end{NB3}%
We assume
\begin{equation}\label{eq:bow3}
    \vk_{1,i}\ge \vk_{2,i}\ge\dots\ge \vk_{\bw_i,i}.
\end{equation}
In particular, all powers appearing in \eqref{eq:bow1} except the last
one are nonpositive.
This assumption is not essential, as it is satisfied if we renumber
$\vk_{\alpha,i}$. Alternatively we apply reflection functors for
quiver varieties \cite{Na-reflect} on two way parts. Here we regard
$V_i^0$ and $V_i^{\bw_i}$ as framing vector spaces, and do not touch
for reflection functors.

These powers, especially the last one, look slightly different from
\cite[(6.3)]{2016arXiv160602002N}, where the corresponding complex
parameters $\nu^\CC_*$, $\nu^\CC_{\ialpha,i}$ are put in the defining
equation.
\begin{NB}
It becomes closer
\begin{equation*}
    \prod_{i=0}^{n-1} (\det V_i^0)^{\delta_{i0}\vk_*+\vk_{1,i}}
    (\det V_i^1)^{-\vk_{1,i} + \vk_{2,i}} \cdots
    (\det V_i^k)^{-\vk_{k,i} + \vk_{k+1,i}} \cdots
    (\det V_i^{\bw_i})^{-\vk_{\bw_i,i}}.
\end{equation*}
\end{NB}%
But it is implicit in the proof of \cite[Prop.~2.9]{Takayama} (see also
\cite[Prop.~3.2]{2016arXiv160602002N} and the numerical criterion
\cite[Def.~2.7]{2016arXiv160602002N}%
\begin{NB}
  where only $S_{\zeta^\pm}$, $T_{\zeta^\pm}$ with
  $S_{\zeta^-}\xrightarrow[\cong]{A_\zeta} S_{\zeta^+}$,
  $V_{\zeta^-}/T_{\zeta^-}\xrightarrow[\cong]{A_\zeta} V_{\zeta^+}/T_{\zeta^+}$
  are considered.
\end{NB}%
) that we have an isomorphism
\begin{NB}
$\det V_{i-1}^{\bw_{i-1}}\cong \det V_{i}^0$,
\end{NB}%
\begin{equation}\label{eq:bow7}
  \det V_i^{\bw_i}\cong \det V_{i+1}^0,
\end{equation}
\begin{NB}
  If $\dim V^{\bw_{i-1}}_{i-1} = \dim V_{i}^0$, $A_{i-1}$ is an isomorphism. If
  $m := \dim V_{i}^0 > n:= \dim V^{\bw_{i-1}}_{i-1}$, $A_{i-1}$ is injective and
  $V^0_{i}/\Ima A_{i-1}$ has a base which is the image of $a_{i}$,
  $B'_{i}a_{i}$, \dots, $(B'_{i})^{m-n-1} a_{i}$. The case
  $\dim V_{i}^0 < \dim V^{\bw_{i-1}}_{i-1}$ is similar.
\end{NB}%
hence the appearance of $\vk_{1,i+1}$
\begin{NB}
$\vk_{\bw_{i-1},i-1}$
\end{NB}%
in $\det V_i^{\bw_i}$
\begin{NB}
$\det V_i^{0}$
\end{NB}%
is natural.
Let us denote the corresponding GIT quotient by
$\cM_{\vk}(\underline{\bv},\underline{\bw})$, where $\vk$ should be
understood as $\vk_*\in\ZZ$, $(\vk_{\ialpha,i})\in\ZZ^\ell/\ZZ$. Let us
denote the corresponding line bundle by $\cL_{\vk}$. We have the
projective morphism
$\boldsymbol{\pi}\colon \cM_{\vk}(\underline{\bv},\underline{\bw})\to
\cM_{0}(\underline{\bv},\underline{\bw})$.
Let
$\BA^{\underline{\bv}} = \prod_{i=0}^{n-1} \BA^{\bv_i}/\mathfrak
S_{\bv_i}$. We have a factorization morphism
$\Psi\colon \cM_{\vk}(\underline{\bv},\underline{\bw})\to
\BA^{\underline{\bv}}$, given by eigenvalues of $B_i$ with
multiplicities, which are same as eigenvalues of $B'_i$ thanks to the
defining equation of bow varieties.
We can apply \ref{prop:flat} later, as
$\cM_{\vk}(\underline{\bv},\underline{\bw})$ is normal
(\cite[Th.~6.15]{2016arXiv160602002N}) and all fibers of $\Psi$ have the same
dimension (\cite[Prop.~6.13]{2016arXiv160602002N}), hence the condition of
\ref{rem:conditions} is satisfied.
Note that $\Psi$ factors through $\boldsymbol{\pi}$.

We have the factorization property
\begin{equation*}
    \cM_{\vk}(\underline{\bv},\underline{\bw})
    \times_{\BA^{\underline{\bv}}}
    (\BA^{\underline{\bv}'}\times\BA^{\underline{\bv}''})_{\mathrm{disj}}
    \cong
    \left(\cM_{\vk}(\underline{\bv}',\underline{\bw})
    \times\cM_{\vk}(\underline{\bv}'',\underline{\bw})\right)
  \times_{\BA^{\underline{\bv}'}\times\BA^{\underline{\bv}''}}
  (\BA^{\underline{\bv}'}\times\BA^{\underline{\bv}''})_{\mathrm{disj}}.
\end{equation*}
See \cite[Th.~6.9]{2016arXiv160602002N}. From its construction the line bundle
$\cL_{\vk}$ is compatible with the factorization, namely $\cL_{\vk}$
on $\cM_{\vk}(\underline{\bv},\underline{\bw})$ is sent to
$\cL_{\vk}\boxtimes\cL_{\vk}$ on
$\cM_{\vk}(\underline{\bv}',\underline{\bw})
\times\cM_{\vk}(\underline{\bv}'',\underline{\bw})$.
This is because $\cL_{\vk}$ is coming from the character
$\vk$, given by the product of determinants of $\GL(\bv_i)$ as in
\eqref{eq:bow1}, and it factors according to a decomposition
$V_i^\ialpha = (V_i^\ialpha)'\oplus (V_i^\ialpha)''$.
Note that this construction chooses an isomorphism between $\cL_\vk$
and $\cL_\vk\boxtimes\cL_\vk$ canonically. This choice will become
more explicit in the factorization formula \eqref{eq:bow2} of a
section $y_i^\ialpha$ later.
\begin{NB}\linelabel{NB:vague}
  This is a little vague, but added for a clarification.
\end{NB}%
This is a generalization of statements in \ref{factorization2}.

Let $\oA^{|\underline{\bv}|}$ denote the open subset of
$\BA^{|\underline{\bv}|}$ consisting of $w_i^k \neq w_i^l$
($k\neq l$), $w_i^k \neq w_{i+1}^l$, $w_i^k\neq 0$ (for $i$ with
$\bw_i\neq 0$). Let
$\oA^{\underline{\bv}} = \oA^{|\underline{\bv}|}/\prod\mathfrak
S_{\bv_i}$. It is the complement of union of all generalized root
hyperplanes of $(G,\bN)$ in the sense of \ref{sec:fixed-points}.

\begin{NB}
  The following is a record of an earlier attempt, which did not work
  out.

Let us define a rational section $y^\vk$ of $\cL_\vk$ defined over
$\Psi^{-1}(\oA^{\underline{\bv}})$.
By factorization $\cM_\vk(\underline{\bv},\underline{\bw})$ is
isomorphic to products of bow varieties with $\bv_i = 1$, $\bv_j = 0$
($j\neq i$) over $\oA^{|\underline{\bv}|}$. Since the line bundle
$\cL_\vk$ also factors, it is enough to define a rational section for
each factor, hence we may assume that $\bv_i=1$, $\bv_j=0$.
Those bow varieties are \cite[6.5.1]{2016arXiv160602002N} ($n=1$) and
\cite[6.5.3]{2016arXiv160602002N} ($n > 1$).

Let us first consider the case $n > 1$. Then the bow variety is
defined as
\begin{align*}
    \xymatrix@C=1.2em{
  & \CC \ar@(ur,ul)_{w_i} \ar@<-.5ex>[rr]_{C_{1,i}}
  && \CC \ar@<-.5ex>[ll]_{D_{1,i}} \ar@<-.5ex>[rr]_{C_{2,i}}
    && \cdots \ar@<-.5ex>[ll]_{D_{2,i}} \ar@<-.5ex>[rr]_{C_{N,i}}
  && \CC \ar@(ur,ul)_{w_i} \ar@<-.5ex>[ll]_{D_{N,i}} \ar[dr]_{b_i} &
  \\
\CC \ar[ur]_{a_i} && &&&&&& \CC.
      }
\end{align*}
We have
$\cM_0(\underline{\bv},\underline{\bw})\cong \cS_N = \{ yz=w^{N}\}$.
Here $a_i$ and $b_i$ are invertible thanks to the conditions
(S1),(S2), and we set $y = b_iC_{N,i}\cdots C_{1,i}a_i$,
$z = a_i^{-1}D_{1,i}\cdots D_{N,i}b_i^{-1}$.
Over $w_i\neq 0$, we have $C_{\ialpha,i}$, $D_{\ialpha,i}\neq 0$ from
the defining equations. Then the $\vk$-stability condition is
automatically satisfied, hence
$\cM_0(\underline{\bv},\underline{\bw})|_{w_i\neq 0}\cong
\cM_\vk(\underline{\bv},\underline{\bw})|_{w_i\neq 0}$.
We define a section $y_i^\ialpha$ ($\ialpha=1,\dots,N$) of the line
bundle $(\det V_i^{\ialpha})^*$ by
$y_i^\ialpha = b_i C_{N,i}\cdots C_{\ialpha+1,i}$.
This is nonvanishing over $w_i\neq 0$, as we have $C_{\ialpha,i}\neq 0$
as above.
We define a section $y^\vk$ of $\cL_\vk$ by
\begin{equation*}
  y^\vk \defeq
  (y_i^1)^{-\vk_{1,i}+\vk_{2,i}}\cdots
  (y_i^\ialpha)^{-\vk_{\ialpha,i}+\vk_{\ialpha+1,i}}\cdots
  (y_i^{\bw_i-1})^{-\vk_{\bw_i-1,i}+\vk_{\bw_i,i}}
  (y_i^{\bw_i})^{-\vk_{\bw_i,i}+\vk_{1,i+1}+\delta_{i+1,0}\vk_*}.
\end{equation*}

Let us next consider the case $n=1$. The bow variety is defined by
\begin{align*}
    \xymatrix@C=1.2em{
      &&&&
      \\
      & \CC \ar@(ur,ul)_{w} \ar[rr]^{A} \ar[dr]_{b}
      \ar@<-2pt> `l[lu] `[ur] `[1,3]_{D_{1}\cdots D_\bw} `_l[0,2] [0,2]
      \ar@{<-} @<+2pt> `l[lu] `[ur] `[1,3]^{C_{\bw}\cdots C_1} `_l[0,2] [0,2]
      && \CC \ar@(ur,ul)_{w'} &
  \\
  && \CC. \ar[ur]_a &&
      }
\end{align*}
We have $w' = -D_1C_1 = \cdots = -C_{\bw}D_{\bw} = w$.
Then $B_2A-AB_1 + ab = 0$ means $ab=0$.
Since $w\neq 0$, we have $C_\ialpha$, $D_\ialpha\neq 0$. Then the
stability condition depends only on $\vk_*$.
If $\vk_* = 0$, it is an affine algebro-geometric quotient, hence we
can take a representative $a=b=0$ in closed orbits.
If $\vk_* < 0$, we have the condition (C-S1) in
\cite[Prop.~6.4]{2016arXiv160602002N} must be satisfied. Thus we have $b\neq 0$
and hence $a=0$.
If $\vk_* > 0$, we have the condition (C-S2) instead. Hence $a\neq 0$, $b=0$.

, and in closed orbits we can  when $\bw\neq 0$.
We have $y = C_{\bw \cdots 1}$. Let us set $x = (-1)^\bw D_{1 \cdots \bw}$.
Then we get $xy = w^{\bw}$ because $w= -D_1C_1$, so
\begin{align*}
\CC[\cM(1, \bw)] = \CC[w, y, x]/ (xy -w^{\bw}).
\end{align*}
\end{NB}%

We order eigenvalues of $B_i$ (which are also eigenvalues of $B'_i$)
as $w_{i,1}$, \dots, $w_{i,\bv_i}$. We consider them as coordinates of
$\BA^{\bv_i}$, and functions on
$\cM_{\vk}(\underline{\bv},\underline{\bw})\times_{\BA^\bv}
\BA^{|\bv|}$. (Here $|\bv| = \sum \bv_i$.)
Define a section $y_{i,k}^\ialpha$ of the vector bundle $(V^{\ialpha}_i)^*$ by
\begin{equation*}
  y_{i,k}^\ialpha \defeq b_i \prod_{\substack{1\le l\le \bv_i\\ l\neq k}} (B_i - w_{i,l})
  C_{\bw_i,i} \cdots C_{\ialpha+1,i}
\end{equation*}
\begin{NB}
  another possible choice: a section of $V^{\ialpha}_i$ by
  \begin{equation*}
    z_{i,k}^\ialpha \defeq C_{\ialpha,i} C_{\ialpha-1,i} \cdots C_{1,i}
    \prod_{\substack{1\le l\le \bv_i\\ l\neq k}} (B'_i - w_{i,l}) a_i
  \end{equation*}
\end{NB}%
\begin{NB}
Note that these make sense even for $\ialpha=0$ as
$y^{0}_{i,k} = b_i\prod_{\substack{1\le l\le \bv_i\\ l\neq k}} (B_i - w_{i,l})
C_{\bw_i,i}\cdots C_{1,i}$.
\end{NB}%
and a rational section $y_i^\ialpha$ of the line bundle
$(\det V^{\ialpha}_i)^*$ defined over
$\Psi^{-1}(\oA^{\underline{\bv}})$ by
\begin{equation}\label{eq:bow4}
  y_i^\ialpha \defeq
  y_{i,1}^\ialpha\wedge y_{i,2}^\ialpha\wedge\cdots
  \wedge y_{i,\bv_i}^\ialpha   \prod_{k > l} (w_{i,k} - w_{i,l})^{-1}.
\end{equation}
Note that this is $\mathfrak S_{\bv_i}$-invariant, as signs from
$y_{i,1}^\ialpha\wedge y_{i,2}^\ialpha\wedge\cdots \wedge y_{i,\bv_i}^\ialpha$ and
$\prod_{k > l} (w_{i,k} - w_{i,l})$ cancel.

\begin{NB3}
  The following paragraph is added.
\end{NB3}%
We also define sections $z^0_{i+1,k}$ ($k=1,\dots,\bv_{i+1}$) of
$V^0_{i+1}$ and $z^0_{i+1}$ of $\det V^0_{i+1}$ by
\begin{equation*}
  \begin{gathered}
  z^0_{i+1,k} \defeq \prod_{\substack{1\le l\le \bv_{i+1}\\ l\neq k}}
  (B_{i+1}' - w_{i+1,l}') a_{i+1},\\
  z^0_{i+1} \defeq z^0_{i+1,1}\wedge z^0_{i+1,2} \wedge \cdots
  \wedge z^0_{i+1,\bv_{i+1}} \prod_{k > l} (w_{i+1,k}' - w_{i+1,l}')^{-1}.
  \end{gathered}
\end{equation*}
We regard $z^0_{i+1}$ as a section of $\det V^{\bw_i}_i$ via \eqref{eq:bow7}.

They are compatible with the factorization as follows. Let
$y_{i,k}^{\prime\ialpha}$ ($1\le k\le\bv'_i$), $y_i^{\prime\ialpha}$,
$y_{i,k}^{\prime\prime\ialpha}$ ($\bv'_i+1\le k\le \bv_i$),
$y_i^{\prime\prime\ialpha}$ be defined for
$\cM_{\vk}(\underline{\bv}',\underline{\bw})$,
$\cM_{\vk}(\underline{\bv}'',\underline{\bw})$ respectively.
As in \cite[Lem.~6.11]{2016arXiv160602002N}, we have
\begin{equation*}
  y_{i,k}^\ialpha =
   \begin{cases}
     y_{i,k}^{\prime\ialpha} \prod_{l=\bv'_i+1}^{\bv_i} (w_{i,k} - w_{i,l})
     & \text{if $1\le k\le \bv'_i$},
     \\
     y_{i,k}^{\prime\prime\ialpha} \prod_{l=1}^{\bv'_i} (w_{i,k} - w_{i,l})
     & \text{if $\bv'_i+1\le k\le \bv_i$},
   \end{cases}
\end{equation*}
and hence
\begin{equation}\label{eq:bow2}
   y_i^\ialpha =
   y_i^{\prime\ialpha} \wedge y_i^{\prime\prime\ialpha}
   \prod_{k=1}^{\bv'_i} \prod_{l=\bv'_i+1}^{\bv_i} (w_{i,k} - w_{i,l}).
\end{equation}
\begin{NB}
  We have
  \begin{equation*}
    y_{i,1}^a\wedge y_{i,2}^a\wedge\cdots \wedge y_{i,\bv_i}^a
    = (-1)^{\bv'_i \bv''_i} y_{i,1}^{\prime\ialpha}\wedge\cdots y_{i,\bv'_i}^{\prime\ialpha}
    \wedge
    y_{i,\bv'_i+1}^{\prime\prime\ialpha}\wedge\cdots y_{i,\bv_i}^{\prime\prime\ialpha}
    \prod_{k=1}^{\bv'_i} \prod_{l=\bv'_i+1}^{\bv_i} (w_{i,k} - w_{i,l})^2.
  \end{equation*}
\end{NB}%
We have a similar formula for $z^0_{i+1}$.

\begin{NB3}
Let $y^\vk$ be a section of $\cL_\vk$ given by
\begin{equation*}
  y^\vk \defeq \prod_{i=0}^{n-1} (y_i^1)^{\vk_{1,i}-\vk_{2,i}}\cdots
  (y_i^\ialpha)^{\vk_{\ialpha,i}-\vk_{\ialpha+1,i}}\cdots
  (y_i^{\bw_i-1})^{\vk_{\bw_i-1,i} - \vk_{\bw_i,i}}
  (y_i^{\bw_i})^{\vk_{\bw_i,i}-\vk_{1,i+1}-\delta_{i+1,0}\vk_*}.
\end{equation*}
(Compare with \ref{eq:bow1}.) It inherits the compatibility with the
factorization from \ref{eq:bow2}.

This is changed as follows, to include the case
$\vk_{\bw_i,i}-\vk_{1,i+1}-\delta_{i+1,0}\vk_* < 0$.
\end{NB3}%

Let $y^\vk$ be a section of $\cL_\vk$ given by
\begin{equation}\label{eq:bow5}
  \begin{split}
  y^\vk \defeq \prod_{i=0}^{n-1} (y_i^1)^{\vk_{1,i}-\vk_{2,i}} &\cdots
  (y_i^\ialpha)^{\vk_{\ialpha,i}-\vk_{\ialpha+1,i}}\cdots
  (y_i^{\bw_i-1})^{\vk_{\bw_i-1,i} - \vk_{\bw_i,i}}\\
  & \times
  \begin{cases}
    (y_i^{\bw_i})^{\vk_{\bw_i,i}-\vk_{1,i+1}-\delta_{i+1,0}\vk_*}
    & \text{if $\vk_{\bw_i,i}-\vk_{1,i+1}-\delta_{i+1,0}\vk_*\ge 0$},
    \\
    (z_{i+1}^0)^{-\vk_{\bw_i,i}+\vk_{1,i+1}+\delta_{i+1,0}\vk_*} &
    \text{otherwise}.
  \end{cases}
  \end{split}
\end{equation}
Note that powers in the first line are nonnegative by the assumption
\eqref{eq:bow3}. The power in the second line is nonnegative by the
definition.

By factorization $\cM_\vk(\underline{\bv},\underline{\bw})$ is
isomorphic to product of bow varieties with
$\dim V^0_i = \dim V^1_i = \cdots = \dim V^{\bw_i}_i = 1$,
$\dim V^\ialpha_j = 0$ ($j\neq i$) over
$\oA^{|\underline{\bv}|}$. Those bow varieties are
\cite[6.5.1]{2016arXiv160602002N} ($n=1$) and \cite[6.5.3]{2016arXiv160602002N}
($n > 1$). In either cases, they are locally isomorphic to
$\CC\times\CC^\times$, as we exclude $w_i^k = 0$.
We also see that $y_i^\ialpha$, $z^0_{i+1}$ are nonvanishing over
$\oA^{|\underline{\bv}|}$, hence $y^\vk$ also.

\begin{NB}
  Alternatively we define $y^\ialpha_{i}$ as follows. We order
  eigenvalues of $B_i$ as $w_{i,1}$, \dots, $w_{i,\bv_i}$ as above. We
  suppose that they are distinct and hence we have factorization to
  the above example with $\dim V_i^{\ialpha} = 1$, $V_j^\ialpha =
  0$. Let ${}'y^\ialpha_{i,k}$ be the section defined for the factor
  corresponding to $w_{i,k}$. (This can be made concrete. Later
  we will identify it with $v^{\bw_i-\ialpha}$ in the local model of
  type $A_{\bw_i-1}$ simple singularity.) Then we define $y^\ialpha_i$ by
  \begin{equation*}
    y^\ialpha_i \defeq {}'y^\ialpha_{i,1}\wedge \cdots\wedge {}' y^{\ialpha}_{i,\bv_i}
    \prod_{k < l} (w_{i,k} - w_{i,l}).
  \end{equation*}
  Note that the isomorphism
  $\det (V^\ialpha_{i,1}\oplus\cdots\oplus V^\ialpha_{i,\bv_i}) \cong
  V^\ialpha_{i,1}\otimes\cdots\otimes V^\ialpha_{i,\bv_i}$ involves the
  sign, hence $\wedge$ is included above.
\end{NB}%

Let us turn to the gauge theory side. We define the flavor symmetry as
follows: We consider the action of
$T(\underline{\bw}) = \prod_i T^{\bw_i}$ on $\bN$ induced from the
standard action of $T^{\bw_i}$ on $\CC^{\bw_i}$. Together with $G$, we
have an action of $(G\times T(\underline{\bw}))/\CC^\times$, where
$\CC^\times$ is embedded in $G\times\prod_i T^{\bw_i}$ as the diagonal
scalars. We have an extra $\CC^\times_{\dil}$ acting on $\bN$ by
scaling on the component $\Hom(\CC^{\bv_{n-1}},\CC^{\bv_0})$.
Let
$\tilde G = \CC^\times_{\dil}\times (G\times
T(\underline{\bw}))/\CC^\times$,
$G_F = \tilde G/G = \CC^\times_{\dil}\times
T(\underline{\bw})/\CC^\times$. Then
$
\begin{NB}
i_{\vk}^!\scAfor =   
\end{NB}%
H^{G_\cO}_*(\tilde\cR^{\vk})$ is a module over
$
\begin{NB}
i_0^!\scAfor =   
\end{NB}%
H^{G_\cO}_*(\cR) = \CC[\cM_C(G,\bN)]$ by the
construction in \ref{subsec:flav-symm-group2}.
\begin{NB}
  \ref{subsec:affG_flavor}, \ref{subsec:line-bundles-via}.
\end{NB}%
Here $\tilde\pi\colon\tilde\cR = \cR_{\tilde G,\bN}\to \Gr_{G_F}$ and
$\tilde\cR^\vk = \tilde\pi^{-1}(\vk)$ as before, and
$\vk = (\vk_*,\vk_{\ialpha,i})$
\begin{NB}
  Is it \emph{not} fine to denote this by $\lambda$, as $\lambda$ is
  usually a dominant weight $\sum_i \dim W_i\omega_i$ for a quiver
  gauge theory.
\end{NB}%
is a coweight of $G_F$, regarded as a point in $\Gr_{G_F}$.
We can also consider
$\cM_C^\vk(G,\bN) = \Proj(\bigoplus_{n\ge 0} H^{G_\cO}_*(\tilde\cR^{n\vk}))$,
which is endowed with a projective morphism
$\cM_C^\vk(G,\bN)\to \cM_C(G,\bN)$.
Let us use the standard basis of $\CC^{\bv_i}$ to take a maximal torus
$T$ of $G$ consisting of diagonal matrices.
We identify $\BA^{\underline{\bv}}$ with the spectrum of
$H^*_{G}(\mathrm{pt}) = H^*_T(\mathrm{pt})^{\prod\mathfrak
  S_{\bv_i}}$. We have
$\varpi\colon \cM_C(G,\bN) \to \BA^{\underline{\bv}}$ given by the
structural homomorphism $H^*_{G}(\mathrm{pt}) \to H^{G_\cO}_*(\cR)$
when $\vk = 0$.
We compose $\cM_C^\vk(G,\bN)\to \cM_C(G,\bN)$ with $\varpi$ to apply
\ref{prop:flat} to $\cM_C^\vk(G,\bN)$ later.

Let $\bNT$ denote $\bN$ regarded as a $T$-module. We have the
pushforward homomorphism
$\iota_*\colon H^{T_\cO}_*(\cR_{T,\bNT})\to H^{T_\cO}_*(\cR_{G,\bN}) =
H^{G_\cO}_*(\cR_{G,\bN})\otimes_{H^*_G(\mathrm{pt})}
H^*_T(\mathrm{pt})$ of the inclusion $\cR_{T,\bNT}\to \cR_{G,\bN}$
(see \ref{sec:bimodule}).
We put the flavor symmetry as above for $T$, i.e.,
$\tilde T \defeq \CC^\times_{\dil}\times (T\times
T(\bw))/\CC^\times$. We have
$\tilde T/T = \CC^\times_{\dil}\times T(\bw)/\CC^\times =
G_F$. We consider $\tilde\pi_T\colon \cR_{\tilde T,\bNT}\to \Gr_{G_F}$
as above, and $\tilde\pi_T^{-1}(\vk)$. We have a natural inclusion
$\tilde\pi_T^{-1}(\vk)\to \tilde\cR^{\vk} = \tilde\pi^{-1}(\vk)$, denoted again by
$\iota$, and the pushforward homomorphism
\[
  \iota_*\colon H^{T_\cO}_*(\tilde\pi_T^{-1}(\vk))\to
  H^{G_\cO}_*(\tilde\cR^{\vk})\otimes_{H^*_G(\mathrm{pt})} H^T_*(\mathrm{pt}).
\]
Let $\pi_T\colon \cR_{\tilde T,\bNT}\to \Gr_{\tilde T}$ be the
projection.

\begin{NB3}
Let $\pi_T\colon \cR_{\tilde T,\bNT}\to \Gr_{\tilde T}$ be the
projection. We lift the coweight $\vk = (\vk_*,\vk_{\ialpha,i})$ of
$G_F$ to $\tilde T$ by setting the $T$-component as
$(\underbrace{\vk_{1,i},\dots,\vk_{1,i}}_{\text{$\bv_i$
    times}})_{i=0}^{n-1}$.
\begin{NB}
  The shift $\vk_{\ialpha,i} + s$ ($s\in\ZZ$) changes $\vk_{1,i}$ also to
  $\vk_{1,i}+s$. Therefore this is well-defined.
\end{NB}%
Let us denote it by $\vk^0$. We consider the fundamental class of
$\pi_T^{-1}(\vk^0)$ and denote it by $\sy^\vk$. This is an analog of
$\sy_{i,k}$ considered in \cite[\S6.8.1]{2016arXiv160602002N}.
By the localization theorem, it is nonvanishing over
$\oA^{|\underline{\bv}|}$.

The following paragraph is changed so that the case 
$\vk_{\bw_i,i}-\vk_{1,i+1}-\delta_{i+1,0}\vk_* < 0$
is included.
\end{NB3}%

Next we introduce a class
$\sy^\vk\in H^{T_\cO}_*(\tilde\pi_T^{-1}(\vk))$, whose image under
$\iota_*$ will be identified with $y^\vk$.
We begin with $\sy_i^\alpha$, $\sz_{i+1}^0$, which will be identified
with $y_i^\alpha$, $z_{i+1}^0$ respectively.
For $\sy_i^\ialpha$ we choose $\vk$ so that
\begin{equation}\label{eq:bow6}
  \begin{minipage}{.8\linewidth}
    the corresponding component $-\vk_{\ialpha,i}+\vk_{\ialpha+1,i}$
    (or $-\vk_{\bw_i,i}+\vk_{1,i+1}+\delta_{i+1,0}\vk_*$ if
    $\ialpha = \bw_i$) is $-1$, and all others appearing in powers of
    \eqref{eq:bow1} are zero.
  \end{minipage}
\end{equation}
\begin{NB}
In other words, the $T(\underline{\bw})$-component is
$(1,1,\dots,1,0,\dots,0)$, and the $\CC^\times_{\dil}$-component is $-1$.
\end{NB}%
For $\sz_i^0$ we choose $\vk$ so that
$-\vk_{\bw_i,i}+\vk_{1,i+1}+\delta_{i+1,0}\vk_*$ is $1$ instead.
A choice of $\vk$ is unique up to overall shifts of
$T(\underline{\bw})$-component.
We lift this coweight $\vk$ of $G_F$ to $\tilde T$ by setting the
$T$-component as
\begin{equation}
  \label{eq:bow8}
  (\underbrace{\vk_{1,i},\dots,\vk_{1,i}}_{\text{$\bv_i$
    times}})_{i=0}^{n-1}.
\end{equation}
\begin{NB}
  The shift $\vk_{\ialpha,i} + s$ ($s\in\ZZ$) changes $\vk_{1,i}$ also to
  $\vk_{1,i}+s$. Therefore this is well-defined.
\end{NB}%
Let us denote the lift by $\tilde\vk$.
\begin{NB3}
  Corrected on June 20.
\end{NB3}%
We define $\sy_i^\alpha$ and $\sz_i^0$ as the fundamental class of
$\pi_T^{-1}(\tilde\vk)$ according to the choice of $\tilde\vk$.
This is an analog of $\sy_{i,k}$ considered in
\cite[\S6.8.1]{2016arXiv160602002N}.
By the localization theorem, it is nonvanishing over
$\oA^{|\underline{\bv}|}$.
For general $\vk$ we define $\sy^\vk$ as the product in
\eqref{eq:bow5} with $y_i^\alpha$, $z_{i+1}^0$ replaced by
$\sy_i^\alpha$, $\sz_{i+1}^0$.

We define a rational isomorphism
$\theta\colon
\Gamma(\cM_0(\underline{\bv},\underline{\bw}),\boldsymbol\pi_*(\cL_{\vk}))
\dasharrow H^{G_\cO}_*(\tilde\pi^{-1}({\vk}))$ by sending $y^\vk$ to
$\iota_* \sy^\vk$. It is $\prod \mathfrak S_{\bv_i}$-equivariant,
hence it is indeed an isomorphism as above.

\begin{NB}
  It seems that we will implicitly use that $\theta$ sends
  $y_i^\ialpha$ to the corresponding section in the Coulomb branch side
  defined similar to $\iota_*\sy^\vk$ by replacing $\vk$ by
  $\omega_{\ialpha,i}$, the $\ialpha$-th fundamental coweight of $\GL(\bw_i)$.
\end{NB}%

\begin{NB3}
Now the assumption is weakened.

We assume
\begin{equation}\label{eq:bow3'}
    \vk_{1,0}\ge \vk_{2,0}\ge\dots\ge \vk_{\bw_0,0}\ge \vk_{1,1} \ge \dots \ge
    \vk_{\bw_1,1} \ge \dots \ge \vk_{\bw_{n-1},n-1} \ge \vk_{1,0} + \vk_*.
\end{equation}
In particular, all powers appearing in \eqref{eq:bow1} are nonpositive.
\end{NB3}%

\begin{Theorem}\label{thm:lbowisom}
  Under the assumption \eqref{eq:bow3}
  $\theta$ extends to an isomorphism
  $\Gamma(\cM_0(\underline{\bv},\underline{\bw}),
  \boldsymbol\pi_*(\cL_{\vk}))\xrightarrow{\sim}
  H^{G_\cO}_*(\tilde\cR^{\vk})$ of
  $\CC[\cM_0(\underline{\bv},\underline{\bw})]= H^{G_\cO}_*(\cR)$-modules.
\end{Theorem}



\begin{proof}
As in the proofs of \ref{Coulomb_quivar}, \cite[Th.~6.18]{2016arXiv160602002N},
we need to study how the Coulomb branch and the bow variety look like
around the general points $t$ of the boundary divisor in
$\BA^{|\underline{\bv}|}$. In our case,
\begin{aenume}
      \item $w_{i-1,k}(t) = w_{i,l}(t)$ for some $i$, $k$, $l$, but all
    others are distinct. Moreover $w_{j,r}(t)\neq 0$ if $\bw_j\neq
    0$. (We understand $i\neq i-1$, hence $n \ge 2$.)
  \item $w_{i,k}(t) = w_{i,l}(t)$ for distinct $k$, $l$ and some $i$,
    but all others are distinct. Moreover $w_{j,r}(t)\neq 0$ if
    $\bw_j\neq 0$.
  \item All pairs like in (a),(b) are distinct, but $w_{i,k}(t) = 0$
    for $i$ with $\bw_i\neq 0$.
\end{aenume}
See the proof of \cite[Th.~6.18]{2016arXiv160602002N}. The gauge theory
$(G,\bN,\tilde G)$ with the flavor symmetry group $\tilde G$ is
replaced by $(Z_G(t),\bN^t,Z_{\tilde G}(t))$. In our case,
$Z_{\tilde G}(t) = \CC^\times_{\dil}\times(Z_G(t)\times
T(\underline{\bw}))/\CC^\times$, and
\(
    (Z_G(t),\bN^t) = (\GL(\underline{\bv}') \times
    T^{|\underline{\bv}''|}, \bN(\underline{\bv}',\underline{\bw}')),
\)
where $\underline{\bv}'$, $\underline{\bw}'$ are given below,
$\underline{\bv}'' = \underline{\bv} - \underline{\bv}'$ and
$T^{|\underline{\bv}''|}$ acts trivially on
$\bN(\underline{\bv}',\underline{\bw}')$:
\begin{aenume}
\item $\underline{\bw}' = 0$, $\bv_i' = 1 = \bv_{i-1}'$ and other
  entries are $0$.
  \item $\underline{\bw}' = 0$, $\bv_i' = 2$ and other
    entries are $0$.
    \item $\bv'_i = 1$, $\bw'_i = \bw_i$ and other entries are $0$.
\end{aenume}
The extra factor $T(\underline{\bw})$ acts trivially in (a),(b), while
it acts through $T(\underline{\bw}) \to T^{\bw_i}$ in (c).
On the other hand $\CC^\times_\dil$ acts trivially in (b),(c) and (a)
with $i\neq 0$.

By the same argument as in the proofs of \ref{Coulomb_quivar},
\cite[Th.~6.18]{2016arXiv160602002N}
both $y^\vk$ and $\sy^\vk$ are related to $y^{\prime\vk}$,
$\sy^{\prime\vk}$ by nonvanishing regular functions defined on a
neighborhood of $t$ in $\BA^{|\bv|}$ under the factorization.
Therefore it is enough to check that the isomorphism $\theta$ extends
for the local models (a),(b),(c) above.

Consider the case (a) with $n\ge 3$.
%
Let us consider the local model for the bow variety side. It is
\cite[6.5.6]{2016arXiv160602002N}:
\begin{NB}
\begin{align*}
    \xymatrix@C=1.2em{
      &\CC \ar@(ur,ul)_{w_{i-1}}\ar@<-.5ex>[rrrr]_{C_{\bw_{i-1},i-1} \cdots C_{1,i-1}} &&&& \CC \ar@(ur,ul)_{w_{i-1}} \ar[rr]^{A} \ar[dr]_{b_{i-1}} \ar@<-.5ex>[llll]_{D_{1,i-1}\cdots D_{\bw_{i-1},i-1}}
      && \CC \ar@(ur,ul)_{w_{i}} \ar@<-.5ex>[rrrr]_{C_{\bw_{i},i} \cdots C_{1,i}} &&&& \CC \ar@(ur,ul)_{w_{i}}\ar[dr]_{b_{i}} \ar@<-.5ex>[llll]_{D_{1,i}\cdots D_{\bw_{i},i}}
      &
  \\
\CC \ar[ur]_{a_{i-1}}  &&&&&& \CC \ar[ur]_{a_{i}} &&&&&& \CC
      }
\end{align*}
\end{NB}%
\begin{align*}
    \xymatrix@C=1.2em{
      &\CC \ar@(ur,ul)_{w_{i-1}} \ar[rr]^{A} \ar[dr]_{b_{i-1}}
      && \CC \ar@(ur,ul)_{w_{i}} \ar[dr]_{b_{i}}
      &
  \\
\CC \ar[ur]^{a_{i-1}}  && \CC \ar[ur]_{a_{i}} && \CC.
      }
\end{align*}
Since we assume $w_{i-1}$, $w_i\neq 0$, the relevant $C_{\ialpha,i-1}$,
$D_{\ialpha,i-1}$, $C_{\ibeta,i}$, $D_{\ibeta,i}$
($\ialpha=1,\dots,\bw_{i-1}$, $\ibeta=1,\dots,\bw_{i})$ are
isomorphisms, hence can be normalized by the group action and defining
equations. Thus they are omitted. It is also clear that the
$\vk$-stability condition is automatically satisfied, hence
$\cM_\vk(\underline{\bv}',\underline{\bw}')\cong
\cM_0(\underline{\bv}',\underline{\bw}')$.

We normalize $a_{i-1} = 1$, $b_i = 1$ thanks to the conditions
(S1),(S2). The defining equation is $(w_i - w_{i-1})A = a_i
b_{i-1}$. On the other hand, we have introduced functions $y_{i-1}$,
$y_i$, $y_{i-1,i}$ in \cite[6.5.6]{2016arXiv160602002N}, which are
$y_{i-1} = b_{i-1} a_{i-1} = b_{i-1}$, $y_i = b_i a_i = a_i$, $y_{i-1,i} =
A$. (We change $y_{i-1,i}$ in \cite{2016arXiv160602002N} by its inverse.)
The variety $\cM_0(\underline{\bv}',\underline{\bw}')$ is
$\{ (w_{i-1}, w_i, y_{i-1},y_i, y_{i-1,i}^{\pm 1}) \mid y_{i-1} y_i =
y_{i-1,i}(w_{i} - w_{i-1})\}$.
\begin{NB}
  It seems that there is a mistake on the sign in \cite{2016arXiv160602002N}.
\end{NB}%
In this case, line bundles $\det V_{i-1}^\ialpha$, $(\det V_i^\ibeta)^*$ are
trivialized by their nonvanishing sections
$C_{\ialpha,i-1}\cdots C_{1,i-1} a_{i-1} = y_{i-1}/y_{i-1}^\ialpha$,
$b_i C_{\bw_i,i}\cdots C_{\ibeta+1,i} = y_i^\ibeta$,
and sections $y_{i-1}^\ialpha$, $z_i^0$, $y_i^\ibeta$, $z_{i+1}^0$ are identified with
$y_{i-1}$, $y_i$, $1$, $1$ respectively.
Therefore
\begin{NB3}
\begin{equation*}
  y^\vk = y_{i-1}^{\vk_{1,i-1}-\vk_{1,i}-\delta_{i,0}\vk_*}.
\end{equation*}
\end{NB3}%
\begin{equation*}
  y^\vk =
  \begin{cases}
    y_{i-1}^{\vk_{1,i-1}-\vk_{1,i}-\delta_{i,0}\vk_*} &
    \text{if $\vk_{\bw_{i-1},i-1}-\vk_{1,i}-\delta_{i,0}\vk_*\ge 0$},\\
    y_{i-1}^{\vk_{1,i-1}-\vk_{\bw_{i-1},i-1}}
    y_i^{-\vk_{\bw_{i-1},i-1}+\vk_{1,i}+\delta_{i,0}\vk_*}
    & \text{otherwise}.
  \end{cases}
\end{equation*}

\begin{NB}
  The following argument is \emph{incomplete}, as we do not identify
  the above trivialization of line bundles in the Coulomb branch side.

  July 20:
  I change my mind. We can trivialize line bundles in both sides
  independently. Then we express sections as functions in both sides,
  and check that they are the same function. If both of
  trivializations are defined everywhere, the isomorphism extends.
\end{NB}%

Next let us consider the local model in the Coulomb branch side. The
group $T(\bw)$ acts trivially on
$\bN(\underline{\bv}',\underline{\bw}')$. The extra
$\CC^\times_\dil$-action appears when $i=0$, but it can be absorbed to
the $\GL(\bv_{i-1})$-action, as we assume $n \ge 3$.
We take an isomorphism $Z_{\tilde G}(t)
\begin{NB}
  = \CC_\dil^\times\times (\CC^\times\times\CC^\times\times T^{|\underline{\bv}''|}\times T(\bw))/\CC^\times
\end{NB}%
\cong \CC^\times\times\CC^\times\times T^{|\underline{\bv}''|}\times G_F$,
then
$H^{G_\cO}_*(\tilde\cR^{\vk}) \cong H^{G_\cO}_*(\cR)$. It means
that the line bundle is trivialized.
Then
\begin{NB3}
$\sy^\vk$  
\end{NB3}%
$\sy_{i-1}^\alpha$, $\sz_i^0$, $\sy_i^\beta$ or $\sz_{i+1}^0$
is the fundamental class of the fiber over the coweight
$(\vk_{1,i-1} - \delta_{i0}\vk_*, \vk_{1,i})$ of
$\GL(\bv_{i-1})\times \GL(\bv_i)$ according to a suitable
choice of $\vk$ as in \eqref{eq:bow6}.
(The ambiguity of shifts does not matter, as it only gives an
invertible function.) Now recall $y_{i-1}$, $y_i$, $y_{i-1,i}$ are
fundamental classes of fibers over $(1,0)$, $(0,1)$, $(1,1)$
respectively under
$\cM_C(\underline{\bv}',\underline{\bw}')\cong \{ y_{i-1}y_i =
y_{i-1,i}(w_i-w_{i-1})\}$ by \ref{abel}.
\begin{NB}
Recall again $y_{i-1,i}$ is invertible.
\end{NB}%
\begin{NB3}
Since $\vk_{1,i-1} - \vk_{1,i} - \delta_{i0}\vk_*\ge 0$ by our
assumption \eqref{eq:bow3}, $\sy^\vk$ is equal to
$y_{i-1}^{\vk_{1,i-1} - \vk_{1,i} - \delta_{i0}\vk_*}$ up to an
invertible function.
\begin{NB}
  Note that $y_{i-1,i}$ is invertible, hence we can strip $\pm (1,1)$
  from coweights.
\end{NB}%
Thus both $y^\vk$ and $\sy^\vk$ are identified with
$y_{i-1}^{\vk_{1,i-1}-\vk_{1,i}-\delta_{i,0}\vk_*}$ up to an
invertible function, and the isomorphism of line bundles extends over
$\cM_0(\underline{\bv}',\underline{\bw}')$.
\end{NB3}%
Thus $\sy_{i-1}^\alpha$, $\sz_i^0$, $\sy_i^\beta$, $\sz_{i+1}^0$ are
equal to $y_{i-1}$, $y_i$, $1$, $1$ up to invertible functions
respectively.
Since both $y^\vk$ and $\sy^\vk$ are defined as products, they are
equal up to an invertible function. Therefore the isomorphism of line
bundles extends over $\cM_0(\underline{\bv}',\underline{\bw}')$.

\begin{NB}
  The proof of the next case is \emph{not} complete. But it
  should be an easier case.

  It is completed on May 11.
\end{NB}%

For (a) with $n=2$, the gauge theory side is reduced to the case
$(\GL(\underline{\bv}'),\bN(\underline{\bv}',0)) =
(\CC^\times\times\CC^\times, \CC\oplus\CC)$
with the $\CC^\times\times\CC^\times$-action
$(t_0,t_1)\cdot (x,y) = (t_1 t_0^{-1}x, t_0 t_1^{-1} y)$
\begin{NB}
  ($x$ is in $\Hom(V_0,V_1)$, $y$ is in $\Hom(V_1,V_0)$)
\end{NB}%
and the flavor group $G_F$ remains only as the
$\CC^\times_\dil$-action by $t_*\cdot (x,y) = (x, t_*y)$ for
$t_*\in\CC^\times_\dil$. Since the diagonal subgroup
$\CC^\times\subset \CC^\times\times\CC^\times$ acts trivially on
$\CC\oplus\CC$, the action factors through the quotient
$\CC^\times\times\CC^\times\to \CC^\times; (t_0,t_1)\mapsto t_0
t_1^{-1}$. The Coulomb branch has the corresponding factor
$\CC\times\CC^\times = \cM_C(\CC^\times,0)$.
We can change the second summand $\CC$ of $\CC\oplus\CC$ by its dual
thanks to \ref{subsec:change_dual}. Hence we are reduced to the
situation in \ref{Klein via} with $V = \CC$, $W = \CC^2$. In
particular,
$\cM_C(\GL(\underline{\bv}'),\bN(\underline{\bv}',0)) =
\CC\times\CC^\times\times \cM_C(\CC^\times,\CC^2) =
\CC\times\CC^\times\times\cS_2$, and the corresponding
$\cM_C^\vk(\GL(\underline{\bv}'),\bN(\underline{\bv}',0))
\begin{NB}
  =
\Proj(\bigoplus_{n\ge 0}
H^{\GL(\underline{\bv}')_{\cO}}_*(\tilde\cR^{n\vk}))  
\end{NB}%
$
is $\CC\times\CC^\times\times T^*\proj^1$.
\begin{NB3}
The section $\sy^\vk$ is the fundamental class of the fiber over
$(\vk_{1,0} - \vk_{1,1}, -\vk_*,0)\in \Gr_{\tilde G}$ if we identify
$\Gr_{\tilde G}$ with the coweight lattice of
$\tilde G = \CC^\times\times\CC^\times\times\CC^\times/\CC^\times$, and also
with $\ZZ^3/\ZZ$. Note that $-\vk_*\ge 0$ by \eqref{eq:bow3}.
Noticing $\vk_{1,0} - \vk_{1,1}\le -\vk_*$ also by \eqref{eq:bow3}, we
find that $\sy^\vk$ is the product $(\sy')^{\vk_{1,0}-\vk_{1,1}}(\sy'')^{\vk_{1,1}-\vk_{1,0}-\vk_*}$ where $\sy'$ (resp.\
$\sy''$) is the fundamental class of the fiber over
$(1,1,0)$ (resp.\ $(0,1,0)$).
\begin{NB}
  As
  $(\vk_{1,0}-\vk_{1,1},-\vk_*,0) = (\vk_{1,0} - \vk_{1,1}, \vk_{1,0}
  - \vk_{1,1}, 0) + (0,\vk_{1,1}-\vk_{1,0}-\vk_*,0)$, and two terms
  have the same sign on each factor of $\CC\oplus \CC$.
\end{NB}%
\end{NB3}%
According to the choice of $\vk$ as in \eqref{eq:bow6}, the section
$\sy_0^\alpha$, $\sy_1^\beta$, $\sz_0^0$ or $\sz_1^0$ is the
fundamental class of fiber over
$(\vk_{1,0}-\vk_{1,1},-\vk_*,0)\in \Gr_{\tilde G}$ if we identify
$\Gr_{\tilde G}$ with the coweight lattice of
$\tilde G = \CC^\times\times\CC^\times\times\CC^\times/\CC^\times$,
and also with $\ZZ^3/\ZZ$. Concretely
$(\vk_{1,0}-\vk_{1,1},-\vk_*,0)$ is $(1,1,0)$, $(0,1,0)$, $(-1,-1,0)$,
$(0,-1,0)$ for $\sy_0^\alpha$, $\sy_1^\beta$, $\sz_1^0$, $\sz_0^0$
respectively.
\begin{NB3}
  See the note \verb+2020-06-19-Note-11-59.xopp+.
\end{NB3}

On the other hand, the local model of the bow variety is given in
\cite[6.5.4]{2016arXiv160602002N} with $\bw_1 = \bw_2 = 0$. Since
$A_0$ is an isomorphism by the conditions (S1),(S2), we can normalize
it to $1$. Then we can factor out $(w_1,A_1)\in\CC\times\CC^\times$,
and the remaining factor is $\cS_2$ and its resolution
$T^*\proj^1$. Line bundles are given by characters of $\CC^\times$
acting on $\CC$ on the right side:
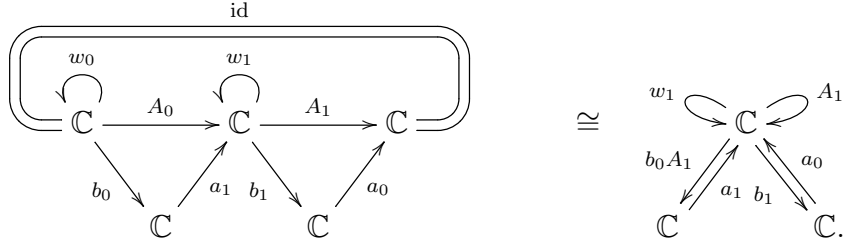
\begin{align*}
  \xymatrix@C=1.2em{
  &&&&&&&&&&&
  \\
  & \CC \ar@(ur,ul)_{w_0} \ar[rr]^{A_0} \ar[dr]_{b_0}
    \ar@{-} @<-2pt> `l[lu] `[ur] `[1,5] `_l[0,4] [0,4]
    \ar@{-} @<+2pt> `l[lu] `[ur] `[1,5]^{\id} `_l[0,4] [0,4]
    && \CC \ar@(ur,ul)_{w_1} \ar[rr]^{A_1} \ar[dr]_{b_1} && \CC &
          && \cong && \CC \ar@(ul,l)_{w_1} \ar@(ur,r)^{A_1}
                      \ar@<-2pt>[dl]_{b_0 A_1} \ar@<-2pt>[dr]_{b_1}&
  \\
  && \CC \ar[ur]_{a_1} && \CC \ar[ur]_{a_0} &&
  &&&\CC\ar@<-2pt>[ur]_{a_1} && \CC. \ar@<-2pt>[ul]_{a_0}        }
\end{align*}
Moreover $y_0^\ialpha$ (resp.\ $y_1^\ibeta$) is identified with
$b_0 A_1$ (resp.\ $b_1$).
\begin{NB3}
Hence $y^\vk
\begin{NB}
= y_0^{\vk_{1,0} - \vk_{1,1}} y_1^{\vk_{1,1} - \vk_{1,0} - \vk_*}
\end{NB}%
= (b_0 A_1)^{\vk_{1,0} - \vk_{1,1}} b_1^{\vk_{1,1} - \vk_{1,0} - \vk_*}.
$
\begin{NB}
  Recall we assume $\vk_{1,0}\ge\vk_{1,1}\ge\vk_{1,0}+\vk_*$.
\end{NB}%
\end{NB3}%
Since we identify the fundamental class of fiber over $(1,1,0)$
(resp.\ $(0,1,0)$) with $b_0A_1$ (resp.\ $b_1$) as in the end of
\ref{Klein via}, we conclude that the isomorphism of line bundles
extends.
For $\vk_*=1$, corresponding to the case $\sz_0^0$ or $\sz_1^0$, we
need to use the opposite stability condition, and $\proj^1$ is
replaced by the dual $\proj^1$.
We replace linear maps above by its transpose to apply \ref{Klein
  via}. Then the fundamental class for $(-1,-1,0)$ (resp.\ $(0,-1,0)$)
are identified with $a_1$ (resp.\ $a_0$), and hence $z_1^0$ (resp.\ $z_0^0$).
Therefore the assertion is true also in this case.
It is also true for general $\vk$ thanks to \ref{tensor Klein}.

Next consider the case (b). First suppose $n\ge 2$. The local model
for the bow variety is \cite[6.5.5]{2016arXiv160602002N}:
\begin{align*}
    \xymatrix@C=1.2em{
      &\CC^2 \ar@(ur,ul)_{B} \ar[dr]^{b}
      &
  \\
\CC \ar[ur]^{a}  && \CC,
                    }
                    \begin{NB}
                          \xymatrix@C=1.2em{
      & \CC^2 \ar@(ur,ul)_{B'} \ar@<-.5ex>[rr]_{C_{\bw \cdots 1}}
      && \CC^2 \ar@(ur,ul)_{B} \ar@<-.5ex>[ll]_{D_{1\cdots \bw}} \ar[dr]_b&
  \\
\CC \ar[ur]_a && && \CC
      }
                    \end{NB}%
\end{align*}
where we drop subscripts $i$. Linear maps $C_{\ialpha,i}$,
$D_{\ialpha,i}$ ($\ialpha=1,\dots,\bw_i$) are isomorphisms thanks to the
assumption that eigenvalues of $B$ are nonzero. Therefore they are
normalized by the group action and defining equations, and omitted.
We have
$\cM_\vk(\underline{\bv}',\underline{\bw}')\cong
\cM_0(\underline{\bv}',\underline{\bw}')$ as before.

Let $w_1$, $w_2$ be eigenvalues of $B$. Then
$\CC[\cM_0(\underline{\bv}',\underline{\bw}')\times_{\BA^{\underline{2}}}\BA^2]$
is
$\CC[w_1,w_2,{}'y_1^\pm,{}'y_2^\pm,\xi]/({}'y_1 - {}'y_2 = \xi(w_1 -
w_2))$ where ${}'y_1 = b(B-w_2)a$, ${}'y_2 = b(B-w_1)a$, $\xi = ba$.
Thanks to the conditions (S1),(S2) we trivialize the dual of the
vector bundle associated with $V = \CC^2$ by a frame $\{ b, bB \}$.
The factorization morphism is given by
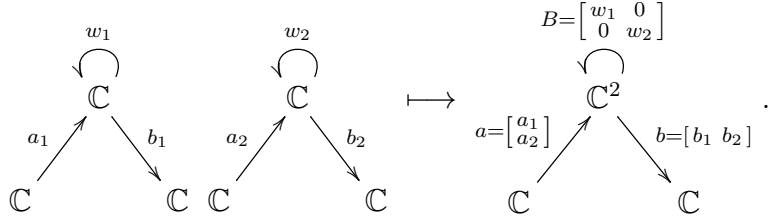
\begin{equation*}
    \xymatrix@C=1.2em{
      &\CC \ar@(ur,ul)_{w_1} \ar[dr]^{b_1}
      &
  \\
\CC \ar[ur]^{a_1}  && \CC
}
    \xymatrix@C=1.2em{
      &\CC \ar@(ur,ul)_{w_2} \ar[dr]^{b_2}
      &
  \\
  \CC \ar[ur]^{a_2}  && \CC}
\longmapsto
    \xymatrix@C=1.2em{
      &\CC^2 \ar@(ur,ul)_{B = \left[
          \begin{smallmatrix}
            w_1 & 0 \\ 0 & w_2
          \end{smallmatrix}\right]} \ar[dr]^{b = \left[\begin{smallmatrix}
            b_1 & b_2
          \end{smallmatrix}\right]}
      &
  \\
  \CC \ar[ur]^{a = \left[\begin{smallmatrix}
        a_1 \\ a_2
          \end{smallmatrix}\right]}  && \CC}.
\end{equation*}
Hence the trivialization $b\wedge bB$ of $\det V^*$ is $b_1 b_2 (w_1 - w_2)$
over the open subset $w_1\neq w_2$.
\begin{NB3}
Hence $\det V^*$ is trivialized by $b\wedge bB = b_1 b_2 (w_1 - w_2)$
over the open subset $w_1\neq w_2$.

Changed on June 19.
\end{NB3}
On the other hand the section $y^\ialpha$ of \eqref{eq:bow4} is
$b_1 b_2 (w_1 - w_2)$. (cf.\ \eqref{eq:bow2}.) Therefore
$y^\ialpha = b\wedge bB$. Thus $y^\ialpha$ extends to a nonvanishing
section over $\cM_0(\underline{\bv}',\underline{\bw}')$.
The same is true for $z^0$.

On the other hand, we have an isomorphism
$H^{\GL(2)_\cO}_*(\tilde\cR^{\vk})\cong
H^{\GL(2)_\cO}_*(\Gr_{\GL(2)})$ if we choose an isomorphism
$Z_{\tilde G}(t)
\begin{NB}
  = \CC^\times_\dil\times (\GL(2)\times T^{|\underline{\bv}''|}\times
  T(\bw))/\CC^\times
\end{NB}%
\cong \GL(2)\times T^{|\underline{\bv}''|}\times G_F$.
The homology class $\sy^\vk$ is identified with a power of
${}'y_1$, ${}'y_2$, which is an invertible function.
Therefore the isomorphism of line bundles extends over
$\cM_0(\underline{\bv}',\underline{\bw}')$.

\begin{NB}
  The proof of the next case is \emph{not} complete. But it
  should be an easier case.

  It is finished on May 11.
\end{NB}%

For (b) with $n=1$, we are reduced to the situation of \ref{Hilbert
  S0} if $\nu_* < 0$. Thus the local model $\cM_C^\vk(\GL(2),\gl(2))
\begin{NB}
= \Proj(\bigoplus_{n\ge 0}
H^{G_\cO}_*(\tilde\cR^{n\vk}))
\end{NB}%
$ is $\on{Hilb}^2(\cS_0)$, and
the line bundle is a power of the determinant line bundle. On the
other hand, the local model of the bow variety is given in
\cite[6.5.2]{2016arXiv160602002N} with $\bw = 0$. It coincides with the
description in \cite[\S1]{Lecture} with constraint $A$ being
invertible. It is nothing but $\on{Hilb}^2(\cS_0)$ and the relevant
line bundles coincide.
Moreover our definition of the section $y^\vk$ is compatible with the
open embedding
$\on{Sym}^2(\cS_0)\setminus\Delta_{\cS_0}
\hookrightarrow\on{Sym}^2(\BA^2)\setminus\Delta_{\BA^2}$ ($\Delta_?$
denotes the diagonal) as in \ref{Hilbert S0}. And the isomorphism is
unique up to a multiplicative scalar on
$\on{Sym}^2(\BA^2)\setminus\Delta_{\BA^2}$ by \ref{costalk
  Hilbert}(b). Therefore our isomorphism coincides with one in
\ref{Hilbert S0}, hence extends over $\Delta_{\cS_0}$.
If $\nu_* > 0$, we take transposes of linear maps to deduce the
assertion from the $\nu_* < 0$ case.
\begin{NB3}
  The new case $\nu_* > 0$ is added.
\end{NB3}

\begin{NB}
  The following case is studied carefully.
\end{NB}%

Let us consider the case (c). First suppose $n > 1$. The local model
for the bow variety side is \cite[6.5.3]{2016arXiv160602002N}:
\begin{align*}
    \xymatrix@C=1.2em{
  & \CC \ar@(ur,ul)_{w} \ar@<-.5ex>[rr]_{C_{1}}
  && \CC \ar@<-.5ex>[ll]_{D_{1}} \ar@<-.5ex>[rr]_{C_{2}}
    && \cdots \ar@<-.5ex>[ll]_{D_{2}} \ar@<-.5ex>[rr]_{C_{N}}
  && \CC \ar@(ur,ul)_{w} \ar@<-.5ex>[ll]_{D_{N}} \ar[dr]_b&
  \\
\CC \ar[ur]_a && &&&&&& \CC,
      }
\end{align*}
where we set $N = \bw_i$ and drop subscripts $i$.
We have
$\cM_0(\underline{\bv}',\underline{\bw}')\cong \cS_N = \{ YZ=W^{N}\}$.
Here $a$ and $b$ are normalized to $1$ thanks to the conditions
(S1),(S2), and we set $y = C_{N}\cdots C_{1}$,
$z = D_{1}\cdots D_{N}$. The section $y^\ialpha$ ($\ialpha=1,\dots,N$)
of the line bundle $(\det V^{\ialpha})^*$ is
$b C_{N}\cdots C_{\ialpha+1}$.
Sections $y^N$, $z^0$ are nowhere vanishing, as well as the
corresponding $\sy^N$, $\sz^0$. So let us ignore $y^N$, $\sy^N$,
$z^0$, $\sz^0$ hereafter.  In particular, we omit the second line in
\eqref{eq:bow5} for the definition of $y^\vk$.
\begin{NB3}
  Added on June 19.
\end{NB3}

After the normalization $a = b = 1$, it becomes a quiver variety of
type $A_{N-1}$. When $\vk_{1,i} > \vk_{2,i} > \cdots > \vk_{\bw_i,i}$,
it is easy to see that $\cM_\vk(\underline{\bv}',\underline{\bw}')$ is
the minimal resolution $\wit\cS_N$ of $yz=w^{N}$ so that $(\det V^1)^*$,
\dots, $(\det V^{N-1})^*$ correspond to line bundles $\cL_{\omega_1}$,
\dots, $\cL_{\omega_{N-1}}$, corresponding to weights $\omega_1$,
\dots, $\omega_{N-1}$ in \ref{line Klein}. On the other hand,
$(\det V^N)^*$ is the trivial line bundle $\shfO_{\wit\cS_N}$. (The
$\vk$-stability under the assumption
$\vk_{1,i} > \vk_{2,i} > \cdots > \vk_{\bw_i,i}$ coincides with the
stability used in \cite{Na-alg}.%
\begin{NB}%
  It means that there is no nonzero subspace
  $S^1\oplus\cdots\oplus S^{N-1}\subset V^1\oplus\cdots\oplus V^{N-1}$
  invariant under $C$, $D$ and $D_1(S^1) = 0$, $C_N(S^{N-1}) =
  0$. In particular, the chain of $\proj^1$'s is characterized by
  $C_1 = 0 = D_N$. Then $\ialpha$-th $\proj^1$ ($\ialpha=1,\dots,N-1$)
  is
\begin{align*}
    \xymatrix@C=1.2em{
  & \CC \ar@(ur,ul)_{w=0} 
  && \CC \ar[ll]_{D_{1}} 
  && \cdots \ar[ll]_{D_{2}} && \CC \ar[ll]_{D_{\ialpha}}
  \ar[rr]_{C_{\ialpha+1}} && \cdots \ar[rr]_{C_N}
  && \CC \ar@(ur,ul)_{w=0} 
     \ar[dr]_{b=1}&
  \\
\CC \ar[ur]_{a=1} &&&&&&&&&&&& \CC,
      }
\end{align*}
where arrows for vanishing linear maps are not written. The
homogeneous coordinates are given by
$[D_1\cdots D_\ialpha:C_N\cdots C_{\ialpha+1}]$.
Note that $D_1\cdots D_\ialpha$, $C_N\cdots C_{\ialpha+1}$ are sections
of $(\det V^\ialpha)^*$.
\end{NB}%
)
Moreover the section $y^\ialpha$ is $v^{N -\ialpha}$ under
the isomorphism
$\Gamma(\wit\cS_N,\cL_{\omega_{\ialpha}})\cong
\CC[\BA^2]^{\chi_{\omega_{\ialpha}}}$. (This holds even for
$\ialpha=N$.)
\begin{NB}
  Recall that $y=C_1\cdots C_N = v^N$, $z= D_1\cdots D_N = u^N$. So it
  seems that we should identify $C_\ialpha$ with $v$, $D_\ialpha$ with
  $u$.
\end{NB}%
\begin{NB}
  We have $v^{N-\ialpha} = y^\ialpha = C_N\cdots C_{\ialpha+1}$. This
  is $r^1$ in the notation in \ref{Klein via}. We have $r^m = y^{m-1} r^1
  = (C_1\cdots C_N)^{m-1} C_N\cdots C_{\ialpha+1}$ for $m > 0$.
  We also have
  $r^0 = u^\ialpha = u^N (uv)^{\ialpha-N} v^{N-\ialpha} = z
  w^{\ialpha-N} v^{N-\ialpha} = D_1\cdots D_\ialpha$. Then $r^m = z^{-m} r^0
  = (D_1\cdots D_N)^{-m} D_1\cdots D_\ialpha$ for $m\le 0$.

  On the other hand, $z^\alpha = C_\ialpha \cdots C_1$ is a section of
  $\det V^\ialpha$ such that
  $\langle C_\ialpha\cdots C_1, C_N\cdots C_{\ialpha+1}\rangle =
  v^N$. Therefore
  $(C_\ialpha\cdots C_1)^{-1} = v^{-N} C_N\cdots C_{\ialpha+1}$. It is
  a rational section $v^{-\ialpha}$.
\end{NB}%
This remains true if $\vk_{\ialpha-1,i} > \vk_{\ialpha,i}$, and other
inequalities may \emph{not} be strict if we replace $\wit\cS_N$ by a
partial resolution of $\cS_N$.
Thus $y^\vk$ is a section of the line bundle
$\cL_\vk = \bigotimes_{\ialpha=1}^{N-1} \cL_{\omega_\ialpha}^{\otimes
  (\vk_{\ialpha,i} - \vk_{\ialpha+1,i})}$, given by the product
$\bigotimes_{\ialpha=1}^{N-1} (v^{N-\ialpha})^{\otimes (\vk_{\ialpha,i} - \vk_{\ialpha+1,i})}$.

The gauge theory
$(\GL(\underline{\bv}'),\bN(\underline{\bv}',\underline{\bw}'))$ is
one studied in \ref{Klein via} with $N=\bw'_i$. We have an extra
$\CC^\times_{\dil}$ in the flavor symmetry group, but it acts
trivially on $\bN(\underline{\bv}',\underline{\bw}')$. Let us ignore
$\CC^\times_\dil$ from now on.
Recall $\sy^\vk$ is the fundamental class of $\pi_T^{-1}(\tilde\vk)$ where
$\tilde\vk = (\vk_{1,i},\vk_{1,i},\vk_{2,i},\dots,\vk_{N,i})$ is
a coweight of
$(\CC^\times\times T^N)/\CC^\times = (\GL(\bv'_i) \times T^{\bw'_i}) /\CC^\times$.
This is so for the lift $\tilde\vk$ of a particular $\vk$ as in
\eqref{eq:bow6}, but remains to be true for $\tilde\vk$ of arbitrary
$\vk$ with \eqref{eq:bow3} if we ignore the second line in
\eqref{eq:bow6}. See \ref{abel}.
\begin{NB3}
  Corrected on June 19.
\end{NB3}%
On the other hand, the fundamental class of
$\pi_T^{-1}(\tilde\omega_\ialpha)$ corresponds to $v^{N-\ialpha}$ by the
computation in~\ref{Klein via}, where
$\tilde\omega_\ialpha = (1,\underbrace{1,\dots,1}_{\text{$\ialpha$
    times}},\underbrace{0,\dots,0}_{\text{$N-\ialpha$ times}})$ is also
a coweight of $(\CC^\times\times T^N)/\CC^\times$.
Since
\begin{equation*}
  \sum_{\ialpha=1}^{N-1} (\vk_{\ialpha,i} - \vk_{\ialpha+1,i})\tilde\omega_\ialpha
  \begin{NB}
  = (\vk_{1,i} - \vk_{N,i},\vk_{1,i} - \vk_{N,i},
  \vk_{2,i} - \vk_{N,i},\dots, \vk_{N-1,i} - \vk_{N,i}, 0)
\end{NB}%
  = \vk
\end{equation*}
holds (up to shift), the class $\sy^\vk$ is equal to
$\bigotimes_{\ialpha=1}^{N-1} (v^{N-\ialpha})^{\otimes (\vk_{\ialpha,i} -
  \vk_{\ialpha+1,i})}$, which is nothing but $y^\vk$.
\begin{NB}
  Here we implicitly use that the product
  $i_\lambda\scAfor\otimes i_\mu\scAfor\to i_{\lambda+\mu}\scAfor$ is
  compatible with the tensor product of line bundles. The proof of
  this statement cannot be found anywhere.

  May 8: This is fixed.
\end{NB}%
This is nothing but the isomorphism normalized as in \ref{rem:ambiguity}.
Thus the isomorphism extends over
$\cM_\vk(\underline{\bv}',\underline{\bw}')$.

\begin{NB}
  The final part is \emph{not} complete either.

  Completed on May 11.
\end{NB}%

If $n=1$, we have
$\bN(\underline{\bv}',\underline{\bw}') = \End(\CC)\oplus
\Hom(\CC^{\bw'_i},\CC)$ and $\GL(\underline{\bv}') = \CC^\times$ acts
trivially on the summand $\End(\CC)$. On the other hand
$\CC^\times_\dil$ acts on $\End(\CC)$ by scaling and trivially on
$\Hom(\CC^{\bw'_i},\CC)$. Then we can separate $\End(\CC)$ and
$\Hom(\CC^{\bw'_i},\CC)$, and both are already treated.
\end{proof}

\subsection{Computation}
\label{computation}

For a later purpose we compute the case (a) with $n\ge 3$ in more
detail. Let us drop the assumption $\underline{\bw}'=0$ and study
general cases with $\bw'_i$, $\bw'_{i-1}$. Let us also write $j$
instead of $i-1$. Let us suppose $i\neq 0$ for brevity. Therefore we
ignore $\vk_*$. Let us also drop `${}'$' from dimension vectors.

\begin{NB}
  \renewenvironment{NB}{ \color{blue}{\bf NB2}. \footnotesize }{}
  \renewenvironment{NB2}{ \color{purple}{\bf NB3}. \footnotesize }{}
  Consider the case (a) with $n\ge 3$. In order to study relations
  between $\vk_{\ialpha,i}$ and $\vk_{\ibeta,i-1}$, let us drop the
  assumption $\underline{\bw}'=0$ and study general cases with
  $\bw'_i$, $\bw'_{i-1}$. Let us also write $j$ instead of $i-1$. Let
  us suppose $i\neq 0$ for brevity. Therefore we ignore $\vk_*$.
\end{NB}%
  
Let us consider the local model for the bow variety side. It is
\cite[6.5.6]{2016arXiv160602002N}:
\begin{align*}
    \xymatrix@C=2.5em{
  &\CC \ar@(ur,ul)_{w_{j}}\ar@<-.5ex>[r]_{C_{1,j}}
  & \CC \ar@<-.5ex>[l]_{D_{1,j}} \ar@<-.5ex>[r]_{C_{2,j}}
  & \cdots \ar@<-.5ex>[l]_{D_{2,j}} \ar@<-.5ex>[r]_{C_{\bw_{j},j}}
  & \CC \ar@(ur,ul)_{w_{j}} \ar[rr]^{A} \ar[dr]_{b_{j}} \ar@<-.5ex>[l]_{D_{\bw_{j},j}}
  && \CC \ar@(ur,ul)_{w_{i}} \ar@<-.5ex>[r]_{C_{1,i}}
  & \CC \ar@<-.5ex>[l]_{D_{1,i}} \ar@<-.5ex>[r]_{C_{2,i}}
  & \cdots \ar@<-.5ex>[l]_{D_{2,i}} \ar@<-.5ex>[r]_{C_{\bw_i,i}}
  & \CC \ar@(ur,ul)_{w_{i}}\ar[dr]_{b_{i}} \ar@<-.5ex>[l]_{D_{\bw_{i},i}}
      &
  \\
\CC \ar[ur]_{a_{j}}  &&&&& \CC \ar[ur]_{a_{i}} &&&&& \CC
      }
\end{align*}

Note that $\CC[\cM_0(\underline{\bv},\underline{\bw})]$ written in
\cite[6.5.6]{2016arXiv160602002N} is wrong, hence we will give a
detail.

We normalize $a_{j} = 1$, $b_i = 1$ thanks to the conditions
(S1),(S2). We also know that $A\neq 0$ thanks to (S1),(S2).
The defining equation for the middle triangle is
$(w_i - w_{j})A = a_i b_{j}$.

We introduce functions
\begin{equation*}
  \begin{gathered}
   z_j = D_{1,j} \cdots D_{\bw_j,j} A^{-1} a_i, \quad
   z_i = b_j A^{-1} D_{1,i} \cdots D_{\bw_i, i}, \quad
   z_{j,i} = D_{1,j} \cdots D_{\bw_j,j} A^{-1}D_{1,i} \cdots D_{\bw_i, i}, \\
  y_j = b_{j} C_{\bw_{j},j}\cdots C_{1,j}, \quad
  y_i = C_{\bw_i,i}\cdots C_{1,i} a_i, \quad
  y_{j,i} = C_{\bw_i,i} \cdots C_{1,i} A C_{\bw_{j},j}\cdots C_{1,j}.
  \\
  \end{gathered}
\end{equation*}
Then
\begin{equation*}
  \begin{gathered}
    z_j z_i
    \begin{NB}{\scriptstyle
      = D_{1,j} \cdots D_{\bw_j,j} A^{-1} a_i b_j A^{-1}
    D_{1,i} \cdots D_{\bw_i,i} = (w_i - w_j) D_{1,j} \cdots
    D_{\bw_j,j} A^{-1} D_{1,i} \cdots D_{\bw_i,i}}
    \end{NB}%
    = (w_i - w_j)
    z_{j,i},
    \\
    y_j y_i = (w_i - w_j) y_{j,i}, \quad
    z_{j,i} y_{j,i}
    \begin{NB}
      {\scriptstyle
= D_{1,j} \cdots D_{\bw_j,j} A^{-1}
    D_{1,i}\cdots D_{\bw_i,i} C_{\bw_i,i} \cdots C_{1,i} A
    C_{\bw_j,j}\cdots C_{1,j}    }\end{NB}%
  = w_i^{\bw_i} w_j^{\bw_j},\\
    z_i y_i
    \begin{NB}{\scriptstyle
= C_{\bw_i,i} \cdots C_{1,i} a_i b_j A^{-1} D_{1,i}
    \cdots D_{\bw_i,i}}
    \end{NB}%
    = (w_i - w_j) w_i^{\bw_i},\\
    z_j y_j
    \begin{NB}{\scriptstyle
= D_{1,j}\cdots D_{\bw_j,j} A^{-1} a_i b_j C_{\bw_j,j}
    \cdots C_{1,j}       }
    \end{NB}%
= (w_i - w_j) w_j^{\bw_j}\\
    z_i y_{j,i} = w_i^{\bw_i} y_j, \quad z_j y_{j,i} = w_j^{\bw_j} y_i, \quad
    y_i z_{j,i} = w_i^{\bw_i} z_j, \quad y_j z_{j,i} = w_j^{\bw_j} z_i.
  \end{gathered}
\end{equation*}
We have $\cM_0(\underline{\bv}, \underline{\bw})
\cong \{ (w_j, w_i, y_j,y_i,y_{j,i},z_j,z_i,z_{j,i}) \mid
\text{above equations}\}$.
On the other hand, this is isomorphic to the Coulomb branch, where
$y_j$, $y_i$, $y_{j,i}$ are fundamental classes of fibers over
$(1,0)$, $(0,1)$, $(1,1)$, and $z_j$, $z_i$, $z_{j,i}$ are those
over $(-1,0)$, $(0,-1)$, $(-1,-1)$.

Let us suppose $w_{j}$, $w_i\neq 0$. Then all
$C_{\ialpha,j}$, $D_{\ialpha,j}$, $C_{\ibeta,i}$, $D_{\ibeta,i}$
become isomorphisms.
Since $z_{j,i} y_{j,i} = w_i^{\bw_i} w_j^{\bw_j}$, $z_{j,i}$ and
$y_{j,i}$ are invertible. We can eliminate $z_{j,i}$,
$z_i = y_{j,i}^{-1} w_i^{\bw_i} y_j$, $z_j = y_{j,i}^{-1} w_j^{\bw_j} y_i$. Hence
$\cM_0(\underline{\bv},\underline{\bw})|_{w_j,w_i\neq 0} \cong
\{ (w_j^{\pm1}, w_i^{\pm1}, y_j, y_i, y_{j,i}^{\pm1}) \mid
y_j y_i = y_{j,i}(w_i-w_j)\}$.

On the other hand when $w_j\neq w_i$, we can eliminate
$y_{j,i} = (w_i - w_j)^{-1} y_j y_i$,
$z_{j,i} = (w_i - w_j)^{-1} z_j z_i$. Hence
$\cM_0(\underline{\bv},\underline{\bw})|_{w_j\neq w_i} \cong \{
(w_j, w_i, y_j, y_i, z_j, z_i) \mid y_jz_j = (w_i - w_j) w_j^{\bw_j},
y_i z_i = (w_i - w_j) w_i^{\bw_i}\}|_{w_i\neq w_j}$. This is an open
subset in the product of type $A_{\bw_j - 1}$ and $A_{\bw_i - 1}$
simple singularities.

Let us recall sections
$y_j^\ialpha = b_j C_{\bw_j,j} \cdots C_{\ialpha+1,j}$,
$y_i^\ibeta = b_i C_{\bw_i,i} \cdots C_{\ibeta+1,i}$ of
$(\det V_j^{\ialpha})^*$, $(\det V_i^\ibeta)^*$ respectively. We
consider other sections
\begin{equation*}
  \begin{gathered}
    {}'y_j^\ialpha = C_{\bw_i,i}\cdots C_{1,i} A C_{\bw_j,j}\cdots C_{\ialpha+1,j},
    \quad
    z_j^\ialpha \defeq D_{1,j}\cdots D_{\ialpha,j}, \\
    z_i^\ibeta \defeq b_j A^{-1} D_{1,i} \cdots D_{\ibeta,i},\quad
    {}'z_i^\ibeta \defeq D_{1,j}\cdots D_{\bw_j,j} A^{-1} D_{1,i}\cdots D_{\ibeta,i}.
  \end{gathered}
\end{equation*}
We have
\begin{equation*}
  \begin{gathered}
    y_j {}'y_j^\ialpha
    \begin{NB}
    = b_j C_{\bw_j,j} \cdots C_{1,j} C_{\bw_i,i}\cdots C_{1,i}
    A C_{\bw_j,j}\cdots C_{\ialpha+1,j}
    \end{NB}%
    = y_{j,i} y_j^\ialpha \\
  z_j y_j^\ialpha
  \begin{NB}
    = D_{1,j} \cdots D_{\bw_j,j} A^{-1} a_i
  b_j C_{\bw_j,j} \cdots C_{\ialpha+1}
  = (w_i - w_j) w_j^{\bw_j-\ialpha} D_{1,j}\cdots D_{\ialpha}
  \end{NB}%
  = (w_i - w_j) w_j^{\bw_j-\ialpha} z_j^\ialpha, \\
  y_i y_j^\ialpha
  \begin{NB}
    = C_{\bw_i,i} \cdots C_{1,i} a_i
  b_j C_{\bw_j,j} \cdots C_{\ialpha+1}
  \end{NB}%
  = (w_i - w_j)\,{}'y_j^\ialpha,\\
  z_{j,i} y_j^\ialpha
  \begin{NB}
    =   D_{1,j}\cdots D_{\bw_j,j} A^{-1} D_{1,i}\cdots D_{\bw_i,i}
    b_j C_{\bw_j,j}\cdots C_{\ialpha+1,j}
  \end{NB}%
  = w_j^{\bw_j-\ialpha} z_i z_j^\ialpha.
  \end{gathered}
\end{equation*}
Note
$z_i z_j^\ialpha = b_j A^{-1} D_{1,i}\cdots D_{\bw_i,i} D_{1,j}\cdots
D_{\ialpha,j}$. Similarly we have
\begin{equation*}
  \begin{gathered}
    y_j z_j^\ialpha
    \begin{NB}
      =  b_j C_{\bw_j,j}\cdots C_{1,j} D_{1,j}\cdots D_{\ialpha,j}
    \end{NB}%
    = w_j^{\ialpha} y_j^\ialpha,\\
    y_{j,i} z_j^\ialpha
    \begin{NB}
      = C_{\bw_i,i}\cdots C_{1,i} A C_{\bw_j,j}\cdots C_{1,j}
    D_{1,j}\cdots D_{\ialpha,j}
    \end{NB}%
    = w_j^{\ialpha}{}'y_j^\ialpha.
  \end{gathered}
\end{equation*}
\begin{NB}
Note that we do not have a formula for $z_i z_j^\ialpha$, similar to
$y_i y_j^\ialpha = (w_i - w_j){}'y_j^\ialpha$.
\end{NB}%

Let us consider the local model in the Coulomb branch side. Let us
take a coweight $(m,1^\ialpha,0^{\bw_j-\ialpha},n,0^{\bw_i})$
\begin{NB}
\[
   (m,\underbrace{1,\dots,1}_{\text{$\ialpha$
    times}},\underbrace{0,\dots,0}_{\text{$\bw_j-\ialpha$
    times}},n,\underbrace{0,\dots,0}_{\text{$\bw_i$ times}})
\]
\end{NB}%
of
$(\GL(V_j)\times T^{\bw_j}\times \GL(V_i)\times
T^{\bw_i})/\CC^\times$.
\begin{NB}
  $\sy_j^{\ialpha}$ is the case $m=1$, $n=0$.
\end{NB}%
Let ${}^\ialpha r^{m,n}$ denote the fundamental class of the fiber for
the projection $\tilde\cR\to \Gr_{\tilde G}$. We can compute products
of ${}^\ialpha r^{m,n}$ with $y_i$, $y_j$, $y_{j,i}$, $z_i$, $z_j$,
$z_{j,i}$ by the formula in \ref{sec:abelian}.
\begin{NB}
Pairings with weights
are $m-1$,\dots, $m-1$ ($\ialpha$ times) $m$, \dots, $m$
($\bw_j-\ialpha$ times) for $\Hom(W_j,V_j)$, $n$, \dots, $n$ ($\bw_i$
times) for $\Hom(W_i,V_i)$, and $n-m$ for $\Hom(V_j,V_i)$.
Recall that $y_j$, $y_i$ are represented by the fundamental
classes of the fibers of $(1,0)$, $(0,1)$. Therefore
\begin{equation*}
  \begin{split}
    & y_j {}^\ialpha r^{m,n} =
    \begin{cases}
      {}^\ialpha r^{m+1,n} & \text{if $m \ge 1$, $n\le m$}, \\
      (w_i - w_j) {}^\ialpha r^{m+1,n} & \text{if $m\ge 1$, $n > m$}, \\
      w_j^{\ialpha}\, {}^\ialpha r^{m+1,n} & \text{if $m=0$, $n \le m$}, \\
      w_j^{\ialpha}(w_i - w_j){}^\ialpha r^{m+1,n} & \text{if $m=0$, $n > m$}, \\
      w_j^{\bw_j} {}^\ialpha r^{m+1,n} & \text{if $m < 0$, $n\le m$}, \\
      (w_i - w_j) w_j^{\bw_j} {}^\ialpha r^{m+1,n} & \text{if
        $m < 0$, $n > m$},
    \end{cases}\\
    & y_i {}^\ialpha r^{m,n} =
    \begin{cases}
      {}^\ialpha r^{m,n+1} & \text{if $n \ge 0$, $n\ge m$}, \\
      (w_i - w_j) {}^\ialpha r^{m,n+1} & \text{if $n\ge 0$, $n < m$}, \\
      w_i^{\bw_i} {}^\ialpha r^{m,n+1} & \text{if $n < 0$, $n\ge m$}, \\
      (w_i - w_j) w_i^{\bw_i} {}^\ialpha r^{m,n+1} & \text{if $n < 0$, $n < m$}.
    \end{cases}
  \end{split}
\end{equation*}
Similarly $z_j$, $z_i$ are represented by the fundamental classes of the
fibers of $(-1,0)$, $(0,-1)$. Therefore
\begin{equation*}
  \begin{split}
    & z_j {}^\ialpha r^{m,n} =
    \begin{cases}
      {}^\ialpha r^{m-1,n} & \text{if $m \le 0$, $n\ge m$}, \\
      (w_i - w_j) {}^\ialpha r^{m-1,n} & \text{if $m\le 0$, $n < m$}, \\
      w_j^{\bw_j-\ialpha}\, {}^\ialpha r^{m-1,n} & \text{if $m=1$, $n \ge m$}, \\
      w_j^{\bw_j-\ialpha}(w_i - w_j){}^\ialpha r^{m-1,n} & \text{if $m=1$, $n < m$}, \\
      w_j^{\bw_j} {}^\ialpha r^{m-1,n} & \text{if $m > 1$, $n\ge m$}, \\
      (w_i - w_j) w_j^{\bw_j} {}^\ialpha r^{m-1,n} & \text{if
        $m > 1$, $n < m$},
    \end{cases}\\
    & z_i {}^\ialpha r^{m,n} =
    \begin{cases}
      {}^\ialpha r^{m,n-1} & \text{if $n \le 0$, $n\le m$}, \\
      (w_i - w_j) {}^\ialpha r^{m,n-1} & \text{if $n\le 0$, $n > m$}, \\
      w_i^{\bw_i} {}^\ialpha r^{m,n-1} & \text{if $n > 0$, $n\le m$}, \\
      (w_i - w_j) w_i^{\bw_i} {}^\ialpha r^{m,n-1} & \text{if $n > 0$, $n > m$}.
    \end{cases}
  \end{split}
\end{equation*}
Finally $y_{j,i}$, $z_{j,i}$ are represented by the fundamental classes
of the fibers over $(1,1)$, $(-1,-1)$. Hence
{\allowdisplaybreaks[4]
\begin{equation*}
  \begin{split}
    &
  y_{j,i} {}^\ialpha r^{m,n} =
  \begin{cases}
    {}^\ialpha r^{m+1,n+1} & \text{if $m\ge 1$, $n\ge 0$}, \\
    w_j^\ialpha {}^\ialpha r^{m+1,n+1} & \text{if $m=0$, $n\ge 0$}, \\
    w_j^{\bw_j} {}^\ialpha r^{m+1,n+1} & \text{if $m < 0$, $n\ge 0$}, \\
    w_i^{\bw_i} {}^\ialpha r^{m+1,n+1} & \text{if $m \ge 1$, $n < 0$}, \\
    w_i^{\bw_i} w_j^\ialpha {}^\ialpha r^{m+1,n+1} & \text{if $m = 0$, $n < 0$}, \\
    w_i^{\bw_i} w_j^{\bw_j} {}^\ialpha r^{m+1,n+1} & \text{if $m < 0$,
      $n < 0$},
  \end{cases}
  \\ &
  z_{j,i} {}^\ialpha r^{m,n} =
  \begin{cases}
    {}^\ialpha r^{m-1,n-1} & \text{if $m\le 0$, $n\le 0$}, \\
    w_j^{\bw_j-\ialpha} {}^\ialpha r^{m-1,n-1} & \text{if $m=1$, $n\le 0$}, \\
    w_j^{\bw_j} {}^\ialpha r^{m-1,n-1} & \text{if $m > 1$, $n\le 0$}, \\
    w_i^{\bw_i} {}^\ialpha r^{m-1,n-1} & \text{if $m \le 0$, $n > 0$}, \\
    w_i^{\bw_i} w_j^{\bw_j-\ialpha} {}^\ialpha r^{m-1,n-1}
    & \text{if $m = 1$, $n > 0$}, \\
    w_i^{\bw_i} w_j^{\bw_j} {}^\ialpha r^{m-1,n-1} & \text{if $m > 1$,
      $n > 0$}.
  \end{cases}
  \end{split}
\end{equation*}
Therefore}
\end{NB}%
A calculation shows that
\begin{equation}\label{eq:rmn}
  {}^\ialpha r^{m,n} =
  \begin{cases}
    y_j^{m-n-1} y_{j,i}^n y_j^\ialpha & \text{if $m > n \ge 0$},\\
    z_i^{-n} y_j^{m-1} y_j^\ialpha & \text{if $m > 0 \ge n$},\\
    y_i^{n-m} y_{j,i}^{m-1} {}'y_j^\ialpha & \text{if $n \ge m > 0$},\\
    y_i^n z_j^{-m} z_j^\ialpha & \text{if $n\ge 0\ge m$},\\
    z_i^{m-n} z_{j,i}^{-m} z_j^\ialpha & \text{if $0\ge m\ge n$},\\
    z_j^{n-m} z_{j,i}^{-n} z_j^\ialpha & \text{if $0\ge n \ge m$}
  \end{cases}
\end{equation}
gives an isomorphism of
$\CC[\cM_0(\underline{\bv},\underline{\bw})]$-modules.
\begin{NB}
\begin{equation*}
  y_j^\ialpha = {}^\ialpha r^{1,0}, \quad
  {}' y_j^\ialpha = {}^\ialpha r^{1,1}, \quad
  z_j^\ialpha = {}^\ialpha r^{0,0}
\end{equation*}
\end{NB}%


\section{Determinant line bundles on convolution diagram over the affine
Grassmannian}
\label{det slice}
In this section we identify the determinant line bundles on the convolution
diagrams over slices in the affine Grassmannian, or rather global sections of
their pushforwards to the slices, with the modules over the Coulomb branches
of the corresponding quiver gauge theories arising from the construction
of \ref{subsec:flav-symm-group2}.
\begin{NB}
\ref{sec:sheav-affine-grassm}.
\end{NB}%

\subsection{Slices revisited}
\label{revisited}
Recall the setup and notations of~\ref{subsec:bd-slices}.
We define the iterated convolution diagram $\wit\CW{}^{\unl\lambda}_\mu$ as the
moduli space of the following data:

\textup{(a)} a collection of $G$-bundles
$\scP_{\on{triv}}=\scP_0,\scP_1,\ldots,\scP_N$ on $\BP^1$;

\textup{(b)} a collection of rational isomorphisms
$\sigma_s\colon\scP_{s-1}\to\scP_s,\ 1\leq s\leq N$, regular over
$\BP^1\setminus\{0\}$, with a pole of degree $\leq \omega_{i_s}$ at $0$;

\textup{(c)} a $B$-structure $\phi$ on $\scP_N$ of degree $w_0\mu$ having fiber
$B_-\subset G$ at $\infty\in\BP^1$ (with respect to the trivialization
$\sigma:=\sigma_N\circ\ldots\circ\sigma_1$).

We have an evident proper birational projection
$\bpi\colon \wit\CW{}^{\unl\lambda}_\mu\to\ol\CW{}^\lambda_\mu$ (where
$\lambda=\sum_{s=1}^N\omega_{i_s}$), sending
$(\scP_0,\ldots,\scP_N,\sigma_1,\ldots,\sigma_N,\phi)$ to $(\scP_N,\sigma,\phi)$.

More generally, we will need an evident generalization
$\bpi\colon \wit\CW{}^{\unl{\unl\lambda}}_\mu\to\ol\CW{}^\lambda_\mu$ for an
arbitrary sequence of dominant coweights
$\unl{\unl\lambda}=(\lambda_1,\ldots,\lambda_n),\ \sum_{s=1}^n\lambda_s=\lambda$,
in place of $(\omega_{i_1},\ldots,\omega_{i_N})$.

Now recall the setup and notations of~\ref{factorization}; in particular,
we set $\alpha=\lambda-\mu$. We pick $\BN[Q_0]\ni\gamma\leq\alpha$, and set
$\beta=\alpha-\gamma$.

\begin{Proposition}
\label{prop:factoriz}
We have a factorization isomorphism
of the varieties over $(\BG_m^{\beta^*}\times\BA^{\gamma^*})_{\on{disj}}$:
$$(\BG_m^{\beta^*}\times\BA^{\gamma^*})_{\on{disj}}\times_{\BA^{\alpha^*}}\wit\CW{}^{\unl\lambda}_{\mu}\iso
(\BG_m^{\beta^*}\times\BA^{\gamma^*})_{\on{disj}}\times_{\BA^{\beta^*}\times\BA^{\gamma^*}}
(\oZ^{\beta^*}\times\wit\CW{}^{\unl\lambda}_{\lambda-\gamma}).$$ It is compatible with
the factorization isomorphism of zastava (see~\ref{zastava}) under projection
$s^\lambda_\mu\circ\bpi$.
\end{Proposition}

\begin{proof}
The same argument as in the proof of~\cite[Proposition~2.4]{bfgm}.
\end{proof}

We fix $i\in Q_0$; recall that $\alpha_i$ is the corresponding simple coroot.
In what follows we will use a particular case of~\ref{prop:factoriz}
similar to~\ref{prop:factor}, where $\gamma=\alpha_i$ and $\beta=\alpha-\alpha_i$.
Here we are additionally able to identify $\wit\CW{}^{\unl\lambda}_{\mu}$
with the minimal resolution of the Kleinian surface $\cS_{\langle\lambda,\alphavee_i\rangle}$.
Recall the birational isomorphism of~\ref{factorization}
\begin{equation*}\varphi\colon
(\BG_m^{\beta^*}\times\BA^1)_{\on{disj}}\times_{\BA^{\alpha^*}}\oW^\lambda_{\mu}\dasharrow
(\BG_m^{\beta^*}\times\BA^1)_{\on{disj}}\times_{\BA^{\beta^*}\times\BA^1}(\oZ^{\beta^*}\times
  \cS_{\langle\lambda,\alphavee_i\rangle}).
\end{equation*}

\begin{Proposition}
\label{lem:factoriz}
The birational isomorphism $\varphi$ extends to a regular isomorphism
of the varieties over $(\BG_m^{\beta^*}\times\BA^1)_{\on{disj}}$:
$$(\BG_m^{\beta^*}\times\BA^1)_{\on{disj}}\times_{\BA^{\alpha^*}}\wit\CW{}^{\unl\lambda}_{\mu}\iso
(\BG_m^{\beta^*}\times\BA^1)_{\on{disj}}\times_{\BA^{\beta^*}\times\BA^1}(\oZ^{\beta^*}\times
\wit\cS_{\langle\lambda,\alphavee_i\rangle}).$$
\end{Proposition}

\begin{proof}
Like in the proof of~\ref{prop:factor}, it suffices to prove the claim over
$\bZ^{\alpha^*}$. So we restrict to this open subset without further mentioning
this and introducing new notations for the corresponding open subsets in the
convolution diagrams over slices. Like in~\ref{prop:blow}, we will identify
$\wit\CW{}^{\unl\lambda}_\mu$ with a certain blowup of $\ol\CW{}^\lambda_\mu$.
To this end we consider a convolution diagram $\ol\Gr{}_G^{\lambda_1}\wit\times
\ldots\wit\times\ol\Gr{}_G^{\lambda_n}\to\ol\Gr{}_G^\lambda,\
\sum_{s=1}^n\lambda_s=\lambda$, and denote it by
$\bpi\colon\wit\Gr{}_G^{\unl{\unl\lambda}}\to\ol\Gr{}_G^\lambda$. Then just as
in~\ref{nondom}, we have $\wit\CW{}^{\unl{\unl\lambda}}_{\mu}=
\wit\Gr{}^{\unl{\unl\lambda}}_G\times_{'\!\on{Bun}_G(\BP^1)}\on{Bun}_B^{w_0\mu}(\BP^1)$.
The sequences $\unl{\unl\lambda}$ we need will have at most one term not equal
to a fundamental coweight, so that $\unl{\unl\lambda}=(\omega_{j_1},\ldots,
\omega_{j_{d-1}},\lambda_d,\omega_{j_{d+1}},\ldots,\omega_{j_n})$.
In fact, we can choose a collection of sequences
$(\lambda)={}^{(0)}\!\unl{\unl\lambda},{}^{(1)}\!\unl{\unl\lambda},\ldots,
{}^{(a)}\!\unl{\unl\lambda}=\unl\lambda=(\omega_{i_1},\ldots,\omega_{i_N})$
such that for any $b<a$ the sequence ${}^{(b+1)}\!\unl{\unl\lambda}$
is obtained from the sequence ${}^{(b)}\!\unl{\unl\lambda}$ by the procedure
${}^{(b)}\!\unl{\unl\lambda}\leadsto{}^{(b)}\!\unl{\unl\lambda}'=:
{}^{(b+1)}\!\unl{\unl\lambda}$ described in three cases (i--iii) below.

\noindent (i) In case $\lambda_d$ is not a fundamental coweight, but
$\langle\lambda_d,\alphavee_j\rangle=1$ for certain vertex $j$ (which may or
may not happen to coincide with our chosen vertex $i$), we set
$$n'=n+1,\ \lambda'_d=\lambda_d-
\omega_j,\ \unl{\unl\lambda}'=(\omega_{j_1},\ldots,
\omega_{j_{d-1}},\lambda'_d,\omega_j,\omega_{j_{d+1}},\ldots,\omega_{j_n}).$$
Then the convolution morphism
$\varpi\colon\wit\Gr{}^{\unl{\unl\lambda}'}_G\to\wit\Gr{}^{\unl{\unl\lambda}}_G$
is an isomorphism up to codimension 2, and hence the convolution morphism
$\varpi\colon\wit\CW{}^{\unl{\unl\lambda}'}_\mu\to\wit\CW{}^{\unl{\unl\lambda}}_\mu$
is an isomorphism (recall that we restricted ourselves to the open subset
over $\bZ^{\alpha^*}$).

\noindent (ii) If $\unl{\unl\lambda}=(\omega_{j_1},\ldots,\omega_{j_e},\omega_{j_{e+1}},\ldots,
\omega_{j_n})$, we set $n'=n,\ \unl{\unl\lambda}=(\omega_{j_1},\ldots,
\omega_{j_{e+1}},\omega_{j_e},\ldots,\omega_{j_n})$, i.e.\ we just swap two
neighbouring fundamental coweights. It follows from (i) above that
$\wit\CW{}^{\unl{\unl\lambda}'}_\mu=\wit\CW{}^{\unl{\unl\lambda}}_\mu$ (over
$\bZ^{\alpha^*}$).

\noindent (iii) In case $\langle\lambda_d,\alphavee_j\rangle\geq2$, we set
$$n'=n+2,\
d'=d+1,\ \lambda'_{d'}=\lambda_d-2\omega_j,\ \unl{\unl\lambda}'=(\omega_{j_1},
\ldots,\omega_{j_{d-1}},\omega_j,\lambda'_{d'},\ \omega_j,\omega_{j_{d+1}},\ldots,
\omega_{j_n}).$$ We also set $n''=n,\ \lambda''_d=\lambda_d-\alpha_j,\
\unl{\unl\lambda}''=(\omega_{j_1},\ldots,
\omega_{j_{d-1}},\lambda''_d,\omega_{j_{d+1}},\ldots,\omega_{j_n})$.

We have an open subvariety $^\circ_j\!\ol\Gr{}^{\lambda_d}_G:=
\Gr^{\lambda_d}_G\sqcup\Gr^{\lambda_d-\alpha_j}_G\subset\ol\Gr{}^{\lambda_d}_G$, and also
an open subvariety $^\circ_j\!\wit\Gr{}^{\unl{\unl\lambda}}_G:=\Gr_G^{\omega_{j_1}}
\wit\times\ldots\wit\times{}^\circ_j\!\ol\Gr{}^{\lambda_d}_G\wit\times\ldots
\wit\times\Gr_G^{\omega_{j_n}}\subset\wit\Gr{}_G^{\unl{\unl\lambda}}$.
We have a closed subvariety
$^\circ_j\!\wit\Gr^{\unl{\unl\lambda}''}_G:=\Gr_G^{\omega_{j_1}}
\wit\times\ldots\wit\times\Gr{}^{\lambda''_d}_G\wit\times\ldots
\wit\times\Gr_G^{\omega_{j_n}}\subset{}^\circ_j\!\wit\Gr{}_G^{\unl{\unl\lambda}}$.
We will denote the restriction of the convolution morphism
$\varpi\colon\wit\Gr{}^{\unl{\unl\lambda}'}_G\to\wit\Gr{}^{\unl{\unl\lambda}}_G$ to
$^\circ_j\!\wit\Gr{}^{\unl{\unl\lambda}}_G\subset\wit\Gr{}_G^{\unl{\unl\lambda}}$ by
$\varpi\colon{}^\circ_j\!\wit\Gr{}^{\unl{\unl\lambda}'}_G\to
{}^\circ_j\!\wit\Gr{}^{\unl{\unl\lambda}}_G$. Similarly, if $j\ne i$ but
$\lambda_d-\alpha_i$ is dominant, we define the open subsets
$^\circ_i\!\ol\Gr{}^{\lambda_d}_G:=
\Gr^{\lambda_d}_G\sqcup\Gr^{\lambda_d-\alpha_i}_G\subset\ol\Gr{}^{\lambda_d}_G$ and
$^\circ_i\!\wit\Gr{}^{\unl{\unl\lambda}}_G\subset\wit\Gr{}_G^{\unl{\unl\lambda}}$.
Then (if $j\ne i$) the convolution morphism
$\varpi\colon{}^\circ_i\!\wit\Gr{}^{\unl{\unl\lambda}'}_G\to
{}^\circ_i\!\wit\Gr{}^{\unl{\unl\lambda}}_G$ is an isomorphism, while
$\varpi\colon{}^\circ_j\!\wit\Gr{}^{\unl{\unl\lambda}'}_G\to
{}^\circ_j\!\wit\Gr{}^{\unl{\unl\lambda}}_G$ is the blowup of
$^\circ_j\!\wit\Gr{}^{\unl{\unl\lambda}}_G$ along the closed subvariety
$^\circ_j\!\wit\Gr^{\unl{\unl\lambda}''}_G\subset
{}^\circ_j\!\wit\Gr{}_G^{\unl{\unl\lambda}}$.

Indeed, \'etale-locally, $^\circ_j\!\wit\Gr{}^{\unl{\unl\lambda}}_G$ splits as a
product $^\circ_j\!\wit\Gr^{\unl{\unl\lambda}''}_G\times\cS_{N_j}$ where
$N_j:=\langle\lambda_d,\alphavee_j\rangle$, and $\varpi$ splits as a
product $\on{Id}\times\ol\varpi$ where $\ol\varpi\colon\cS'_{N_j}\to\cS_{N_j}$
is the restriction of $\varpi$ to any slice $\cS_{N_j}$. Now $\cS'_{N_j}$ is
a normal surface, smooth if $N_j=2$, and the fiber of $\ol\varpi$ over
$0\in\cS_{N_j}$ is the projective line if $N_j=2$. Furthermore, if $N_j>2$, then
the fiber of $\ol\varpi$ over $0\in\cS_{N_j}$ is a union of two projective lines
intersecting at a point; this point in $\cS'_{N_j}$ has Kleinian
$A_{N_j-3}$-singularity (in particular, it is smooth if $N_j=3$).
The check reduces to the case of rank $1$ by the argument
of~\cite[Section~3]{mov}. In rank $1$ it follows e.g.\ from~\cite{MR1968260}.
We conclude
that $\ol\varpi\colon\cS'_{N_j}\to\cS_{N_j}$ is the blowup of $\cS_{N_j}$ at
$0\in\cS_{N_j}$ (in effect, the minimal resolution $\wit\cS{}'_{N_j}$ of
$\cS'_{N_j}$ must coincide with the minimal resolution $\wit\cS_{N_j}$ of
$\cS_{N_j}$, hence $\cS'_{N_j}$ must be obtained from $\wit\cS_{N_j}$ by blowing
down all the exceptional divisor components except for the two outermost ones),
and hence $\varpi\colon{}^\circ_j\!\wit\Gr{}^{\unl{\unl\lambda}'}_G\to
{}^\circ_j\!\wit\Gr{}^{\unl{\unl\lambda}}_G$ is the blowup of
$^\circ_j\!\wit\Gr{}^{\unl{\unl\lambda}}_G$ along the closed subvariety
$^\circ_j\!\wit\Gr{}^{\unl{\unl\lambda}''}_G\subset
{}^\circ_j\!\wit\Gr{}_G^{\unl{\unl\lambda}}$.

We define $^\circ_j\!\wit\CW{}_\mu^{\unl{\unl\lambda}}:=
{}^\circ_j\!\wit\Gr{}_G^{\unl{\unl\lambda}}\times_{'\!\on{Bun}_G(\BP^1)}
\on{Bun}_B^{w_0\mu}(\BP^1),\ ^\circ_i\!\wit\CW{}_\mu^{\unl{\unl\lambda}}:=
{}^\circ_i\!\wit\Gr{}_G^{\unl{\unl\lambda}}\times_{'\!\on{Bun}_G(\BP^1)}
\on{Bun}_B^{w_0\mu}(\BP^1)$, and we define
$\varpi\colon{}^\circ_j\!\wit\CW{}_\mu^{\unl{\unl\lambda}'}\to
{}^\circ_j\!\wit\CW{}_\mu^{\unl{\unl\lambda}}\supset
{}^\circ_j\!\wit\CW{}_\mu^{\unl{\unl\lambda}''},\
\varpi\colon{}^\circ_i\!\wit\CW{}_\mu^{\unl{\unl\lambda}'}\to
{}^\circ_i\!\wit\CW{}_\mu^{\unl{\unl\lambda}}$ similarly. By the argument used in the
proof of~\ref{lem:CMBD}, the morphisms
$\wit\Gr{}^{\unl{\unl\lambda}}_G\stackrel{p\circ\bpi}{\longrightarrow}{}'
\!\on{Bun}_G(\BP^1)\leftarrow\on{Bun}_B^{w_0\mu}(\BP^1)$ are Tor-independent,
hence $\varpi\colon{}^\circ_j\!\wit\CW{}^{\unl{\unl\lambda}'}_\mu\to
{}^\circ_j\!\wit\CW{}^{\unl{\unl\lambda}}_\mu$ is the blowup of
$^\circ_j\!\wit\CW{}^{\unl{\unl\lambda}}_\mu$ along the closed subvariety
$^\circ_j\!\wit\CW{}^{\unl{\unl\lambda}''}_\mu\subset
{}^\circ_j\!\wit\CW{}_\mu^{\unl{\unl\lambda}}$, while
$\varpi\colon{}^\circ_i\!\wit\CW{}^{\unl{\unl\lambda}'}_\mu\to
{}^\circ_i\!\wit\CW{}^{\unl{\unl\lambda}}_\mu$ is an isomorphism (if $j\ne i$).

In case $j\ne i$, the open subvariety $(\BG_m^{\beta^*}\times\BA^1)_{\on{disj}}
\times_{\BA^{\alpha^*}}{}^\circ_i\!\wit\CW{}^{\unl{\unl\lambda}}_\mu
\subset(\BG_m^{\beta^*}\times\BA^1)_{\on{disj}}
\times_{\BA^{\alpha^*}}\wit\CW{}_\mu^{\unl{\unl\lambda}}$ coincides with the whole of
$(\BG_m^{\beta^*}\times\BA^1)_{\on{disj}}
\times_{\BA^{\alpha^*}}\wit\CW{}_\mu^{\unl{\unl\lambda}}$.
Hence the convolution morphism $\varpi\colon(\BG_m^{\beta^*}\times\BA^1)_{\on{disj}}
\times_{\BA^{\alpha^*}}\wit\CW{}^{\unl{\unl\lambda}'}_\mu\to
(\BG_m^{\beta^*}\times\BA^1)_{\on{disj}}
\times_{\BA^{\alpha^*}}\wit\CW{}^{\unl{\unl\lambda}}_\mu$ is an isomorphism.

In case $j=i$, the open subvariety $(\BG_m^{\beta^*}\times\BA^1)_{\on{disj}}
\times_{\BA^{\alpha^*}}{}^\circ_j\!\wit\CW{}^{\unl{\unl\lambda}}_\mu
\subset(\BG_m^{\beta^*}\times\BA^1)_{\on{disj}}
\times_{\BA^{\alpha^*}}\wit\CW{}_\mu^{\unl{\unl\lambda}}$ coincides with the whole of
$(\BG_m^{\beta^*}\times\BA^1)_{\on{disj}}
\times_{\BA^{\alpha^*}}\wit\CW{}_\mu^{\unl{\unl\lambda}}$. Furthermore,
the closed subvariety $(\BG_m^{\beta^*}\times\BA^1)_{\on{disj}}
\times_{\BA^{\alpha^*}}{}^\circ_j\!\wit\CW{}^{\unl{\unl\lambda}''}_\mu
\subset(\BG_m^{\beta^*}\times\BA^1)_{\on{disj}}
\times_{\BA^{\alpha^*}}\wit\CW{}_\mu^{\unl{\unl\lambda}}$ coincides with the
singular locus (with its reduced scheme structure) of
$(\BG_m^{\beta^*}\times\BA^1)_{\on{disj}}
\times_{\BA^{\alpha^*}}\wit\CW{}_\mu^{\unl{\unl\lambda}}$.
Arguing by induction, we conclude that
$(\BG_m^{\beta^*}\times\BA^1)_{\on{disj}}\times_{\BA^{\alpha^*}}\wit\CW{}^{\unl\lambda}_{\mu}$
coincides with $\on{Bl}_{\lfloor\frac{\langle\lambda,\alphavee_i\rangle}{2}\rfloor}$, where
$\on{Bl}_0:=(\BG_m^{\beta^*}\times\BA^1)_{\on{disj}}
\times_{\BA^{\alpha^*}}\oW^\lambda_{\mu}\cong
(\BG_m^{\beta^*}\times\BA^1)_{\on{disj}}\times_{\BA^{\beta^*}\times\BA^1}(\oZ^{\beta^*}
\times\cS_{\langle\lambda,\alphavee_i\rangle})$, and $\on{Bl}_b$ is the result of
blowup of $\on{Bl}_{b-1}$ at its singular locus,
$b=1,\ldots,\lfloor\frac{\langle\lambda,\alphavee_i\rangle}{2}\rfloor$.
Hence, $\on{Bl}_{\lfloor\frac{\langle\lambda,\alphavee_i\rangle}{2}\rfloor}\cong
(\BG_m^{\beta^*}\times\BA^1)_{\on{disj}}\times_{\BA^{\beta^*}\times\BA^1}(\oZ^{\beta^*}
\times\wit\cS_{\langle\lambda,\alphavee_i\rangle})$.

The proposition is proved.
\begin{NB}
  This is a record of an earlier manuscript, which was criticized by
  HN as isomorphisms of fibers may not extend to a global isomorphism.

Given~\ref{prop:factor} and the fact that
$\wit\cS_{\langle\lambda,\alphavee_i\rangle}$ is the minimal resolution of
$\cS_{\langle\lambda,\alphavee_i\rangle}$, it suffices to check two claims:

(i) $(\BG_m^{\beta^*}\times\BA^1)_{\on{disj}}\times_{\BA^{\alpha^*}}\wit\CW{}^{\unl\lambda}_{\mu}$
is smooth;

(ii) the fibers of $\bpi$ in the LHS of the desired isomorphism are the
same as the fibers of the minimal resolution in the RHS.

We denote the convolution diagram $\ol\Gr{}_G^{\omega_{i_1}}\wit\times\ldots
\wit\times\ol\Gr{}_G^{\omega_{i_N}}\to\ol\Gr{}_G^\lambda$ by
$\bpi\colon\wit\Gr{}_G^{\unl\lambda}\to\ol\Gr{}_G^\lambda$. Then just as
in~\ref{nondom}, we have $\wit\CW{}^{\unl\lambda}_{\mu}=
\wit\Gr{}^{\unl\lambda}_G\times_{'\!\on{Bun}_G(\BP^1)}\on{Bun}_B^{w_0\mu}(\BP^1)$.
Now (i) follows from the fact that $\wit\CW{}^{\unl\lambda}_{\mu}$ is smooth in
codimension $2$, that in turn follows by the argument used in the proof
of~\ref{lem:CMBD}: the morphisms
$\wit\Gr{}^{\unl\lambda}_G\stackrel{p\circ\bpi}{\longrightarrow}{}'
\!\on{Bun}_G(\BP^1)\leftarrow\on{Bun}_B^{w_0\mu}(\BP^1)$ are Tor-independent,
and $\wit\Gr{}^{\unl\lambda}_G$ is smooth in codimension $2$, since
$\ol\Gr{}^{\omega_{i_s}}_G$ is smooth in codimension $2$.

(ii) follows as well, since the fibers of
$\bpi\colon \wit\CW{}^{\unl\lambda}_\mu\to\ol\CW{}^\lambda_\mu$ are the same as the
fibers of $\bpi\colon\wit\Gr{}_G^{\unl\lambda}\to\ol\Gr{}_G^\lambda$, and the
latter fibers are known e.g.\ by~\cite{Lus-ast}.
\end{NB}
\end{proof}

\subsection{Determinant line bundles}
\label{dlb}
Note that we have a whole collection of morphisms from
$\wit\CW{}^{\unl\lambda}_\mu$ to $\Gr_G$: for $1\leq s\leq N$ we set
$p_s(\scP_0,\ldots,\scP_N,\sigma_1,\ldots,\sigma_N,\phi):=
(\scP_s,\sigma_s\circ\ldots\circ\sigma_1)$. Recall the determinant line bundle
$\CL$ on $\Gr_G$ (see e.g.~\ref{Determinant}). For $1\leq s\leq N$ we
define the relative determinant line bundle $\CalD_s$ on
$\wit\CW{}^{\unl\lambda}_\mu$ as $\CalD_s:=p_s^*\CL\otimes p_{s-1}^*\CL^{-1}$
(where $p_0^*\CL$ is understood as a trivial line bundle). For a collection of
integers $\varkappa=(k_1,\ldots,k_N)\in\BZ^N$, we define a line bundle
$\CalD^\varkappa$ on $\wit\CW{}^{\unl\lambda}_\mu$ as
$\bigotimes_{s=1}^N\CalD_s^{\otimes k_s}$. In other words, for the obvious
projection $\bp\colon \wit\CW{}^{\unl\lambda}_\mu\to\wit\Gr{}_G^{\unl\lambda}$
and similarly defined line bundle $\CalD^\varkappa_\Gr$ on the Grassmannian
convolution diagram $\wit\Gr{}_G^{\unl\lambda}$, we have
$\CalD^\varkappa=\bp^*\CalD^\varkappa_\Gr$. In particular,
$\CalD^{(1,1,\ldots,1)}=p_N^*\CL$ is trivial.

For $i\in Q_0$, we set
$N_i=\langle\lambda,\alphavee_i\rangle=\sharp\{s : \omega_{i_s}=\omega_i\}$.
We order the set of indices $s$ such that $\omega_{i_s}=\omega_i\colon
s^{(i)}_1<\ldots<s^{(i)}_{N_i}$.
We associate to $\varkappa\in\BZ^N$ a collection of coweights
$\varkappa^{(i)}=
\sum_{n=1}^{N_i-1}(k_{s^{(i)}_n}-k_{s^{(i)}_{n+1}})\omega_n,\ i\in Q_0$,
of $\PGL(W_i)$.
We will denote by $\Lambda^+_F\subset\BZ^N$ the set of all $\varkappa$
such that $k_{s^{(i)}_1}\geq k_{s^{(i)}_2}\geq\ldots\geq k_{s^{(i)}_{N_i}}$ for any
$i\in Q_0$. We will denote by $\Lambda^{++}_F\subset\Lambda^+_F$ the set of
all $\varkappa$ such that $k_{s_1}\geq k_{s_2}$ for any $1\leq s_1<s_2\leq N$.

\begin{NB3}
  I would like to make whether these assumptions are used or not.
\end{NB3}%

\begin{Proposition}
\label{prop:factoriz det}
The factorization isomorphism of~\ref{prop:factoriz} lifts to a canonical
(in the sense explained during the proof) isomorphism of line bundles
$$\left((\BG_m^{\beta^*}\times\BA^{\gamma^*})_{\on{disj}}\times_{\BA^{\alpha^*}}
\wit\CW{}^{\unl\lambda}_{\mu},\ \CO_{(\BG_m^{\beta^*}\times\BA^{\gamma^*})_{\on{disj}}}\otimes
\CalD^\varkappa\right)\iso$$ $$\left((\BG_m^{\beta^*}\times\BA^{\gamma^*})_{\on{disj}}
\times_{\BA^{\beta^*}\times\BA^{\gamma^*}}
(\oZ^{\beta^*}\times\wit\CW{}^{\unl\lambda}_{\lambda-\gamma}),\
\CO_{(\BG_m^{\beta^*}\times\BA^{\gamma^*})_{\on{disj}}}\otimes\CO_{\oZ^{\beta^*}}\boxtimes
\CalD^\varkappa\right).$$
\begin{NB3}
  This holds for any $\varkappa\in\ZZ^N$.
\end{NB3}%
\end{Proposition}

\begin{NB}\linelabel{NB:HNq}
  I hope that the canonicity will be made into an actual condition
  like in l.\lineref{NB:vague}.
\end{NB}%

\begin{proof}
The factorization isomorphism of~\ref{prop:factoriz} associates to the data
of $(\scP_0,\ldots,\\ \scP_N,\sigma_1,\ldots,\sigma_N,\phi)$ the data of
$(\scP_0^{(1)}=\ldots=\scP_N^{(1)},\sigma_1^{(1)}=\ldots=\sigma_N^{(1)}=\on{id},
\phi^{(1)})$ and
$(\scP_0^{(2)},\ldots,\scP_N^{(2)},\sigma_1^{(2)},\ldots,\sigma_N^{(2)},\phi^{(2)})$.
By construction, the relative determinant of $\scP_s$ and $\scP_{s-1}$
coincides with the relative determinant of $\scP_s^{(2)}$ and $\scP_{s-1}^{(2)}$.
\end{proof}

\begin{NB}
To study the above line bundles, we recall the setup and notations
of~\ref{divisors}. We will denote the divisor $(\iota^\lambda_\mu)^{-1}(\oE'_i)$
by $\oE''_i\subset\ol\CW{}^\lambda_\mu$ for short. The following description of
$\oE''_i$ was used in the proof of~\ref{lem:dete}. We have a diagram of
morphisms $\ol\CW{}^\lambda_\mu\stackrel{\bp}{\longrightarrow}
\ol\Gr{}^\lambda_G\supset\Gr^\lambda_G\stackrel{\on{pr}}{\longrightarrow}\CB$.
Here as in the above mentioned proof we assume that $\lambda$ is regular
(the general case is similar but requires introducing more notations, e.g.\
for partial flag varieties). The complement to the open $B_-$-orbit in $\CB$
is the union of Schubert divisors $D^-_i\subset\CB,\ i\in Q_0$, and
$\oE''_i$ is the closure of the preimage $\bp^{-1}\on{pr}^{-1}(D^-_i)$.
Now the convolution morphism
$\bpi\colon\wit\Gr{}_G^{\unl\lambda}\to\ol\Gr{}_G^\lambda$ is one-to-one over
$\Gr_G^\lambda\subset\ol\Gr{}_G^\lambda$, and we have
$\CalD^\varkappa_\Gr|_{\Gr_G^\lambda}\cong
\CO_{\Gr_G^\lambda}(\sum_{s=1}^Nk_s\on{pr}^*D_{i_s}^-)$. This isomorphism is unique
up to an invertible constant since $\Gamma(\Gr^\lambda_G,\CO^\times)=\BC^\times$.

$\CL^\varkappa_\Gr|_{\Gr_G^\lambda}\cong\CO_{\Gr_G^\lambda}
(\sum_{i\in Q_0}\sum_{n=1}^{N_i}\chi_n^{(i)}\on{pr}^*D^-_i)$ as can be seen by
comparing the $T$-weights in the fibers of both sides at the $T$-fixed points.
\begin{NB2}
    HN : Is the isomorphism unique ? It seems so as $\Gr_G^\lambda$ is
    a vector bundle over $G/P_\lambda$. And if this computation is
    correct (I would like to see the detail), then I am puzzled as
    $k_N$ seems to be irrelevant to $\cL^\varkappa$ when we pull back
    the line bundle to $\wit\CW{}^{\unl\lambda}_\mu$.
\end{NB2}%
Here for $i\in Q_0$, we set
$N_i=\langle\lambda,\alphavee_i\rangle=\sharp\{s : \omega_{i_s}=\omega_i\}$;
we order the set of indices $s$ such that $\omega_{i_s}=\omega_i\colon
s^{(i)}_1<\ldots<s^{(i)}_{N_i}$, and define a partition
$\chi^{(i)}=(\chi^{(i)}_1\geq\chi^{(i)}_2\geq\ldots\geq\chi^{(i)}_{N_i})$ as follows:
$\chi^{(i)}_n:=\sum_{m\geq n}k_{s^{(i)}_m}$. We will store all the partitions
$\chi^{(i)},\ i\in Q_0$, as a sequence $\chi\in\BN^N$: for $1\leq s\leq N$ we
find $i\in Q_0$ and $n\leq N_i$ such that $s=s_n^{(i)}$, and set
$\chi_s=\chi^{(i)}_n$. We will write $\chi=\chi(\varkappa)$. Note that
$\varkappa\mapsto\chi(\varkappa)$ is an embedding $\BN^N\hookrightarrow\BN^N$;
we denote the image by $\Lambda^+_F\subset\BN^N$, and we denote the inverse
bijection $\Lambda^+_F\iso\BN^N$ by $\chi\mapsto\varkappa(\chi)$.
\begin{NB2}
  HN:

  The sequence $\chi^{(i)}$ will be identified with a part of the
  cocharacter $\chi$ of $G_F$, corresponding to $\GL(W_i)$. Therefore
  the construction should be independent of over all shift of
  $\chi^{(i)}_n$ independent of $n$ \emph{and} $i$. On the other hand,
  the most of argument works for a shift of $\chi^{(i)}_n$
  \emph{depending on} $i$. In particular, it is not clear where in the
  current argument one cannot shift $\chi^{(i)}_n$ depending on $n$.

  In the current definition of $\chi^{(i)}$, a shift of $\chi^{(i)}_n$
  depending on $n$ corresponds to adding $s_i$ to $k_{s_{N_i}^{(i)}}$
  (the last entry for $i$). If $s_i$ is independent of $i$, it is an
  overall shift.

  On the other hand, only $k_N$ is irrelevant in the definition of
  $\cL^\varkappa$.
\end{NB2}%

In particular, we obtain a trivialization
$\CL^\varkappa_\Gr|_{\Gr^\lambda_G\setminus\bigcup_{i\in Q_0}\on{pr}^{-1}(D^-_i)}\cong
\CO_{\Gr^\lambda_G\setminus\bigcup_{i\in Q_0}\on{pr}^{-1}(D^-_i)}$, and
$\CL^\varkappa|_{\CW^\lambda_\mu\setminus\bigcup_{i\in Q_0}\overset{\circ}{E}{}''_i}\cong
\CO_{\CW^\lambda_\mu\setminus\bigcup_{i\in Q_0}\overset{\circ}{E}{}''_i}$. Here $\CW^\lambda_\mu:=
\Gr^\lambda_G\times_{'\!\on{Bun}_G(\BP^1)}\on{Bun}_B^{w_0\mu}(\BP^1)$ is an open
subvariety of $\ol\CW{}^\lambda_\mu$, and the convolution morphism
$\bpi\colon \wit\CW{}^{\unl\lambda}_\mu\to\ol\CW{}^\lambda_\mu$ is
one-to-one over $\CW^\lambda_\mu$. Note that
$\Gr^\lambda_G\setminus\bigcup_{i\in Q_0}\on{pr}^{-1}(D^-_i)=
\Gr^\lambda_G\cap T_{w_0\lambda}=\ol\Gr{}^\lambda_G\cap T_{w_0\lambda}=
\ol\Gr{}^\lambda_G\setminus\bigcup_{i\in Q_0}\ol{\on{pr}^{-1}(D^-_i)}$
(open intersection with a semiinfinite orbit). Hence
$\ol\CW{}^\lambda_\mu\setminus\bigcup_{i\in Q_0}\oE''_i=
\CW^\lambda_\mu\setminus\bigcup_{i\in Q_0}\oE''_i$. Now recall from~\ref{lem:full}
that the restriction of $s^\lambda_\mu\colon\ol\CW{}^\lambda_\mu\to Z^{\alpha^*}$
to $\ol\CW{}^\lambda_\mu\setminus(\bigcup_{i\in Q_0}\oE''_i\cup\bigcup_{i\in Q_0}\oE_i)$
establishes its isomorphism with
$\bpi_{\alpha^*}^{-1}(\BG_m^{\alpha^*})\subset\oZ^{\alpha^*}$.
All in all, we obtain a trivialization $\tau$ of $\CL^\varkappa$ restricted to
$(\bpi_{\alpha^*}\circ s^\lambda_\mu\circ\bpi)^{-1}\BG_m^{\alpha^*}
\subset\wit\CW{}^{\unl\lambda}_\mu$.
\end{NB}%

We consider the Kleinian surface resolution
$\wit\cS_{N_i}\stackrel{\bpi}{\longrightarrow}\cS_{N_i}
\stackrel{\varPi}{\longrightarrow}\BA^1$ with a line bundle $\CL_{\varkappa^{(i)}}$.
See \ref{line Klein}.
\begin{NB3}
  Reference added on June 12 by HN.
\end{NB3}%
\begin{NB}
Recall that $\cS_{N_i}$ is given by the equation $xy=w^{N_i}$, and $\varPi$ is
the projection onto the line with $w$ coordinate. Note that
$\bpi$ is one-to-one over $\varPi^{-1}(\BG_m)\subset\cS_{N_i}$, and the
restriction of $\CL_{\chi^{(i)}}$ to $\bpi^{-1}\varPi^{-1}(\BG_m)$ is trivialized
since $\CL_{\chi^{(i)}}^{\otimes N_i}=\CO_{\wit\cS_{N_i}}(D)$ for a divisor $D$ supported
on $\bpi^{-1}(0)$, but $\bpi^{-1}(0)\cap\bpi^{-1}\varPi^{-1}(\BG_m)=\emptyset$.
Tensoring with the identity automorphism of $\CO_{\oZ^{\beta^*}}$ we obtain a
trivialization $$\tau'\colon \left(\bpi_{\beta^*}^{-1}(\BG_m^{\beta^*})\times
\bpi^{-1}\varPi^{-1}(\BG_m),\ \CO\right)\iso
\left(\bpi_{\beta^*}^{-1}(\BG_m^{\beta^*})\times\bpi^{-1}\varPi^{-1}(\BG_m),\
\CO_{\bpi_{\beta^*}^{-1}(\BG_m^{\beta^*})}\boxtimes\CL_{\chi^{(i)}}\right).$$
\end{NB}

\begin{Corollary}
\label{cor:factoriz det}
The factorization isomorphism of~\ref{lem:factoriz}
lifts to a canonical (in the sense explained during the proof) isomorphism of
line bundles
$$\left((\BG_m^{\beta^*}\times\BA^1)_{\on{disj}}\times_{\BA^{\alpha^*}}
\wit\CW{}^{\unl\lambda}_{\mu},\ \CO_{(\BG_m^{\beta^*}\times\BA^1)_{\on{disj}}}\otimes
\CalD^\varkappa\right)\iso$$
$$\left((\BG_m^{\beta^*}\times\BA^1)_{\on{disj}}\times_{\BA^{\beta^*}\times\BA^1}
(\oZ^{\beta^*}\times\wit\cS_{N_i}),\ \CO_{(\BG_m^{\beta^*}\times\BA^1)_{\on{disj}}}\otimes
\CO_{\overset{\circ}{Z}{}^{\beta^*}}\boxtimes\CL_{\varkappa^{(i)}}\right).$$
\begin{NB3}
  This holds for any $\varkappa\in\ZZ^N$.
\end{NB3}%
\end{Corollary}

\begin{NB}\linelabel{NB:HNq2}
  The canonicity should be fixed when an answer to l.\lineref{NB:HNq}
  will be given.

  Anyway the proof is not precise, as you do not define the torus
  action on the line bundle $\CalD^\varkappa$ before.
\begin{NB2} I hope to type down the calculation of the torus action in the
fixed points when I am back to Moscow in the end of August.
\end{NB2}

I complained as there is no explanation of an action on
$\wit\CW{}^{\unl\lambda}_{\lambda-\alpha_i}$, as far as I read.
\end{NB}%

\begin{proof}
Due to~\ref{prop:factoriz det}, it suffices to construct an isomorphism
$(\wit\CW{}^{\unl\lambda}_{\lambda-\alpha_i},\CalD^\varkappa)\iso
(\wit\cS_{N_i},\CL_{\varkappa^{(i)}})$. This reduces to the case of rank $1$ by the
argument of~\cite[Section~3]{mov}. In rank $1$ we compare the weights of
the Cartan torus in the fixed points.

Namely, $G=\GL(2),\ \omega$ is the fundamental weight $(1,0),\ \unl{\lambda}$
is a sequence $(\omega,\ldots,\omega)\ (N$ times), $\alpha=(1,-1)$ is the
simple root, $\lambda=N\omega=(N,0),\ \lambda-\alpha=(N-1,1)$, and we will write
$\wit\CW$ for $\wit\CW{}^{\unl\lambda}_{\lambda-\alpha}$. Then $\wit\CW$ is a locally
closed subvariety of the convolution diagram
$\Gr^\omega_G\wit{\times}\ldots\wit{\times}\Gr^\omega_G\ (N$ times).
The latter convolution diagram is the moduli space of flags of lattices
$L_0\supset L_1\supset\ldots\supset L_N$ where $L_0=V\otimes\BC[[z]],\
V=\BC e_1\oplus\BC e_2$, and $\dim L_n/L_{n+1}=1$ for any $n=0,\ldots,N-1$.
The fixed points $\wit\CW{}^T=\{p_0,\ldots,p_{N-1}\}$ (where
$T\subset\GL(2)=\GL(V)$ is the diagonal torus) are as follows:
$p_r=(L_0^{(r)}\supset\ldots\supset L_N^{(r)})$
where $L_n^{(r)}$ is spanned by $z^ne_1,e_2$ for $0\leq n<r$, and by
$z^{n-1}e_1,ze_2$ for $r\leq n\leq N-1$. In particular, $L_0^{(r)}=L_0$, and
$L_N^{(r)}=z^{N-1}\BC[[z]]e_1\oplus z\BC[[z]]e_2$. The fiber of $\CalD_s$ at $p_r$
is $\BC z^{s-1}e_1$ for $1\leq s\leq r$, and $\BC e_2$ for $s=r+1$, and
$\BC z^{s-2}e_1$ for $r+1<s\leq N$.
\begin{NB} The loop rotation $\BC^\times$ acts on $\CalD_s$
naturally: the character of a fiber isomorphic to $\BC z^le_{1,2}$ is $q^l$.
\end{NB}%
Let $T_1$ be the image of $T\subset\GL(2)$ in $\on{PGL}(2)$. The natural action
of $T$ on the convolution diagram factors through $T_1$, and the action of $T_1$
lifts to an action on $\CalD_s$: the character of the fiber (at a fixed point)
isomorphic to $\BC z^le_1$ (resp.\ $\BC z^le_2$) is $1$ (resp.\ $x_1^{-1}$).
Here $x_1$ is the generator of $X^*(T_1)$.
Recall the action of $\BC^\times\times\BC^\times$ on $\wit\cS_N$
in~\ref{line Klein}. We will be interested in the action of the first copy
of $\BC^\times$. It factors through the quotient modulo the subgroup of
$N$-th roots of unity: $\BC^\times\twoheadrightarrow\BC^\times/\sqrt[N]{1}$.
We identify $\BC^\times/\sqrt[N]{1}$ with $T_1$ so that the pullback of
$x_1\in X^*(T_1)$ to $\BC^\times$ coincides with $x^N$. Then the identification
$\wit\cS_N\simeq\wit\CW$ is $\BC^\times\twoheadrightarrow T_1$-equivariant,
it takes $p_r\in\wit\cS_N$ to $p_r\in\wit\CW$,
and the characters of $\BC^\times$ in the fibers of $\CL_{\omega_s-\omega_{s-1}}$
and $\CalD_s$ at the respective fixed points in $\wit\cS_N$ and $\wit\CW$ match
up to an overall twist (independent of a fixed point) by the character $x$ of
$\BC^\times$.

This defines the desired isomorphism
$(\wit\CW,\CalD_s)\iso(\wit\cS_N,\CL_{\omega_s-\omega_{s-1}})$ up to
multiplication by an invertible constant, and hence
$(\wit\CW,\CalD^\varkappa)\iso(\wit\cS_N,\CL_\varkappa)$ (also up to multiplication
by an invertible constant).
This is the only ambiguity in the choice of isomorphism of corollary.
\begin{NB}
Like in the proof of~\ref{prop:factor}, it suffices to prove the claim over
$\bZ^{\alpha^*}$. So we restrict to this open subset without further mentioning
this and introducing new notations for the corresponding open subsets in the
convolution diagrams over slices. We denote the irreducible components of the
exceptional divisor of
$$(\BG_m^{\beta^*}\times\BA^1)_{\on{disj}}\times_{\BA^{\alpha^*}}
\wit\CW{}^{\unl\lambda}_{\mu}\stackrel{\bpi}{\longrightarrow}
(\BG_m^{\beta^*}\times\BA^1)_{\on{disj}}\times_{\BA^{\alpha^*}}
\ol\CW{}^{\unl\lambda}_{\mu}$$ by $E^\CW_j,\ 1\leq j\leq N_i-1$; they correspond
to the irreducible components of the exceptional divisor of
$$(\BG_m^{\beta^*}\times\BA^1)_{\on{disj}}\times_{\BA^{\beta^*}\times\BA^1}
(\oZ^{\beta^*}\times\wit\cS_{N_i})\stackrel{\bpi}{\longrightarrow}
(\BG_m^{\beta^*}\times\BA^1)_{\on{disj}}\times_{\BA^{\beta^*}\times\BA^1}
(\oZ^{\beta^*}\times\cS_{N_i})$$ denoted by $E^\cS_j$, cf.~\ref{line Klein}.
The unit section of the structure sheaf of
$\bpi_{\beta^*}^{-1}(\BG_m^{\beta^*})\times\bpi^{-1}\varPi^{-1}(\BG_m)$ under
isomorphism $\tau'$ vanishes to the order $\chi^{(i)}_j-\chi^{(i)}_{j+1}$
at $E^\cS_j$. We have to check that the unit section of the structure
sheaf of $(\bpi_{\alpha^*}\circ s^\lambda_\mu\circ\bpi)^{-1}\BG_m^{\alpha^*}$ under
isomorphism $\tau$ vanishes to the order $\chi^{(i)}_j-\chi^{(i)}_{j+1}=k_{s_j^{(i)}}$
at $E^\CW_j$. By the construction of $\tau$, the latter claim can be checked
for the trivialization of $\CL^\varkappa_\Gr$ on
$\Gr^\lambda_G\setminus\bigcup_{i\in Q_0}\on{pr}^{-1}(D^-_i)$ and the irreducible
component $E^{\Gr,i}_j$ of the exceptional divisor of
$\wit\Gr{}^{\unl\lambda}\stackrel{\bpi}{\longrightarrow}\ol\Gr{}^\lambda$.
This check reduces to the case of rank $1$ by the argument
of~\cite[Section~3]{mov}. In rank $1$ it follows e.g.\ from~\cite{MR1968260}.
\end{NB}%
\end{proof}

\subsection{Sections of determinant line bundles}
\label{sdlb}
For $1\leq s\leq N$, we set $\lambda_s:=\omega_{i_1}+\ldots+\omega_{i_s}$.
Then the projection $p_s\colon \wit\CW{}^{\unl{\lambda}}_\mu\to\Gr_G$ lands into
$\ol\Gr{}_G^{\lambda_s}$. The determinant line bundle
$\CL|_{\ol\Gr{}_G^{\lambda_s}}\simeq\CO_{\ol\Gr{}_G^{\lambda_s}}(\sum_{i\in Q_0}\langle\lambda_s,
\alphavee_i\rangle S_{\lambda_s-\alpha_i}\cap\ol\Gr{}_G^{\lambda_s})$ has a canonical
section $z_{\lambda_s}$ vanishing to the order $\langle\lambda_s,\alphavee_i\rangle$ at
the semiinfinite orbit $S_{\lambda_s-\alpha_i}$ intersecting $\ol\Gr{}_G^{\lambda_s}$
in codimension 1. For $\varkappa=(k_1\geq\ldots\geq k_N)\in\Lambda^{++}_F$, the
line bundle $\CalD^\varkappa=\bigotimes_{s=1}^Np_s^*\CL^{\otimes (k_s-k_{s+1})}$ (we set
$k_{N+1}=0$)
\begin{NB3}
  Changed to $\otimes$ on June 12 by HN. We used
  $\CalD^\varkappa = \bigotimes_{s=1}^N (p_s^*\CL\otimes
  p_{s-1}^*\CL^{-1})^{\otimes k_s} = \bigotimes_{s=1}^N
  p_s^*\CL^{\otimes (k_s - k_{s+1})}$.
\end{NB3}%
has a section $z^\varkappa:=\bigotimes_{s=1}^Np_s^*z_{\lambda_s}^{k_s-k_{s+1}}$.
In particular, recall that $\CalD^{(1,1,\dots,1)} = p_N^*\CL$ is trivial, but the section
$z^{(1,1,\dots,1)} = p_N^* z_{\lambda_N}\ne1$ since it vanishes along some semiinfinite orbits.
\begin{NB3}
  Recall that $\CalD^{(1,1,\dots,1)} = p_N^*\CL$ is trivial. On the
  other hand, we have $z^{(1,1,\dots,1)} = p_N^* z_{\lambda_N}$.
  But I guess that $z_{\lambda_N}$ is \emph{not} $1$, right ? \linelabel{zlN}
  \begin{NB2}
    Inserted in the text.
    \end{NB2}%
\end{NB3}%


\subsection{Example}
\label{example}
We consider $G=\on{SL}(3),\ \mu=0,\ \unl{\lambda}=(\omega_j,\omega_i),\
\lambda=\omega_i+\omega_j=\alpha_i+\alpha_j$. The slice $\ol\CW{}^\lambda_\mu$
is the closure of the minimal nilpotent orbit in ${\mathfrak{sl}}_3$, and
$\widetilde\CW{}^{\unl{\lambda}}_\mu$ is the cotangent bundle $T^*\BP^2$ where
$\BP^2=\BP(V)$, and $V$ has a basis $b_1,b_2,b_3$, and $V^*$ has the dual basis
$a_1,a_2,a_3$. We assume that these bases are eigenbases for a Cartan torus $T$,
and the weight of $a_1$ equals $\omega_i,\ \wt(a_2)=\omega_i-\alpha_i,\
\wt(a_3)=-\omega_j$. The zastava $Z^\lambda$ is given by equation
$y_iy_j=(w_i-w_j)y_{j,i}$, and the open zastava $\oZ^\lambda\subset Z^\lambda$ is
given by $y_{j,i}\ne0$. The weights
$\wt(y_i)=\alpha_i,\ \wt(y_j)=\alpha_j,\ \wt(w_i)=\wt(w_j)=0,\
\wt(y_{j,i})=\lambda$.

We have the canonical projections
$\widetilde\CW{}^{\unl{\lambda}}_\mu\to\ol\CW{}^\lambda_\mu\to Z^\lambda$, and a section
$\oZ^\lambda\hookrightarrow\ol\CW{}^\lambda_\mu$. We consider the incidence
quadric $Q\subset V\times V^*$ given by $a_1b_1+a_2b_2+a_3b_3=0$. Its categorical
quotient modulo the `hyperbolic' $\BC^\times$-action is $\ol\CW{}^\lambda_\mu$,
and the composed projection $Q\to Z^\lambda$ acts as
$$y_i=a_1b_2,\ y_j=a_2b_3,\ y_{j,i}=a_1b_3,\ w_i=-a_1b_1,\ w_j=a_3b_3,\
w_i-w_j=a_2b_2.$$
The preimage of the open zastava $\oZ^\lambda\subset Z^\lambda$ is given by
$a_1\ne0\ne b_3$. The composition $\oZ^\lambda\hookrightarrow\CW^\lambda_\mu
\hookrightarrow\Gr^\lambda_{\on{SL}(3)}\to\CB$ (the flag variety of $\on{SL}(3)$)
is nothing but the evaluation at $0\in\BP^1$ morphism (viewing $\oZ^\lambda$
as based maps from $\BP^1$ to $\CB$).

The Picard group $\on{Pic}(\widetilde\CW{}^{\unl{\lambda}}_\mu)\simeq\BZ$, generated
by the first determinant bundle $\CL_1=\CalD_1$ that coincides with the
pullback of $\CO(1)$ from $\BP^2$. The global sections
$\Gamma(\widetilde\CW{}^{\unl{\lambda}}_\mu,\CalD_1)$ are the functions on the incidence quadric
$Q$ having weight 1 with respect to the hyperbolic $\BC^\times$.
In particular, this line bundle has $T$-eigensections $a_1,a_2,a_3$. The restriction of $a_1$ to
$\oZ^\lambda\subset\widetilde\CW{}^{\unl{\lambda}}_\mu$ is nowhere vanishing.
The restriction of $a_2$ vanishes along the divisor
$\on{div}(y_j)\subset\oZ^\lambda$, and the restriction of $a_3$ vanishes along
the divisor $\on{div}(w_j)\subset\oZ^\lambda$. Note that
$a_3=p_1^*z_{\lambda_1}\in\Gamma(\widetilde\CW{}^{\unl{\lambda}}_\mu,\CalD_1)$.
Furthermore, the section of the trivial line bundle (i.e.\ a function)
$p_2^*z_{\lambda_2}=-a_3b_1$.


\begin{NB3}
  `,' is added on June 12.
\end{NB3}%
Comparing with~\ref{computation}, we conclude that
(in our situation $\ialpha=1$)
$$z_i=-a_2b_1,\ z_j=a_3b_2,\ z_{j,i}=-a_3b_1,\ y_j^1=a_2,\ 'y_j^1=a_1,\ z_j^1=a_3.$$
From~\eqref{eq:rmn} we conclude that
$^1r^{0,0}=z_j^1=a_3$ (the fundamental class of the preimage of the cocharacter
$(0,1,0,0)\in X_*(\GL(V_j)\times T^{{\mathbf w}_j}\times\GL(V_i)\times
T^{{\mathbf w}_i})$).

\begin{NB3}
  I am afraid that we have a slight confusion here. HN introduces
  \emph{functions} $z_i$, $z_j$ in \ref{computation}. On the other
  hand, MF introduces a \emph{section} $z_{\lambda_s}$ in \ref{sdlb},
  which looks very similar.

  Anyway $z_{\lambda_1}$ should be $z^1_j$. And what is
  $z_{\lambda_2}$ ?  See l.\lineref{zlN}. In \ref{det via hom slice}
  MF will identify it with the class associated with
  $\varkappa = (1,1)$, lifted to $\tilde T$ by setting the $T$
  component to be $0$. This is how HN understands the
  convention. Since the overall shift is irrelevant, this will be
  identified with the function corresponding to $(-1,-1,0,0)$. It is
  equal to $z_{j,i}$ in \ref{computation}.
  \begin{NB2}
    Inserted into the previous paragraph.
    \end{NB2}%
\end{NB3}%

Similarly, the fundamental class of the preimage of the cocharacter
$(0,0,0,1)\in X_*(\GL(V_j)\times T^{{\mathbf w}_j}\times\GL(V_i)\times
T^{{\mathbf w}_i})$ is the section $b_1$ of the pullback of $\CO(1)$ from
$\BP(V^*)$ to $T^*\BP(V^*)=\widetilde\CW{}^{\unl{\lambda}'}_\mu$ where
$\unl{\lambda}'=(\omega_i,\omega_j)$.

More generally, the fundamental class of the preimage of the cocharacter
$(\min(k_1,k_2),\ k_1,\\ \min(k_1,k_2),\ k_2)$ restricted to $\oZ^\lambda$ vanishes
to the order $k_1-k_2$ at the divisor $w_j=0$ if $k_1\geq k_2$, and to the
order $k_2-k_1$ at the divisor $w_i=0$ if $k_1\leq k_2$, and is invertible
elsewhere, in particular at $w_i=w_j$. Hence for $k\geq0$ the fundamental
class of the preimage of the cocharacter
$(\min(k_1,k_2)-k,k_1,\min(k_1,k_2)-k,k_2)$ restricted to $\oZ^\lambda$
is invertible off the zero divisors of $w_i$ and $w_j$.
This follows from~\ref{eq:rmn} (note that $z_{j,i}$ is invertible at
the generic point of the divisor $w_i=w_j$).

\subsection{Determinant sheaves on slices via homology groups of fibers}
\label{det slice via}
We recall the setup of~\ref{quivar} and~\ref{defo}.
We set $G=\GL(V),\ G_F=T(W),\ \tilde{G}=G\times G_F$. The group $\tilde{G}$
acts on $\bN^\lambda_\mu$. According to~\ref{Coulomb_quivar}, the Coulomb branch
$\CM_C(G,\bN)$ is identified with $\ol\CW{}^{\lambda^*}_{\mu^*}$.
Our choice of basis of the character
lattice of $T(W)$ defines a cone of dominant coweights of $\GL(W)\supset T(W)$.
It is nothing but $\Lambda^+_F$ introduced in~\ref{dlb}.
For $\varkappa\in\Lambda^+_F$, the homology $H_*^{G_\CO}(\tilde\CR{}^\varkappa)$ forms
a module over the algebra $H_*^{G_\CO}(\CR)$, and for
$\varkappa\in\Lambda^{++}_F$ we want to identify this
module with $\Gamma(\wit\CW{}^{\unl\lambda^*}_{\mu^*},\CalD^{\varkappa})=
\Gamma(\ol\CW{}^{\lambda^*}_{\mu^*},\bpi_*\CalD^{\varkappa})$.
Here the assumption $\varkappa\in\Lambda^{++}_F$ is \emph{not}
essential, as we can renumber $i_1$, \dots, $i_N$ so that
$k_1\ge k_2\ge\dots\ge k_N$.
\begin{NB3}
    Add on June 12 by HN.
\end{NB3}%

First we consider the case $\varkappa=(1,\ldots,1,0,\ldots,0)$,
i.e.\ $\CalD^\varkappa=p_s^*\CL$. We allow $s=N$, so that $\varkappa=(1,\ldots,1)$.
\begin{NB3}
  Is the case $s=N$ included ? Then there will be no $0$.
  \begin{NB2}
    Corrected.
    \end{NB2}%
\end{NB3}%
Let $\bN_T$ denote $\bN$ regarded as a $T$-module. We have the
pushforward homomorphism
$\iota_*\colon H^{T_\cO}_*(\cR_{T,\bNT})\to H^{T_\cO}_*(\cR) =
H^{G_\cO}_*(\cR)\otimes_{H^*_G(\mathrm{pt})}
H^*_T(\mathrm{pt})$ of the inclusion $\cR_{T,\bNT}\to \cR$
(see \ref{sec:bimodule}). We set $\tilde{T}:=T\times G_F=T\times T(W)$.
We consider $\tilde\pi_T\colon \cR_{\tilde T,\bNT}\to \Gr_{G_F}$,
and the fiber $\tilde\pi_T^{-1}(\varkappa)$. We have a natural inclusion
$\tilde\pi_T^{-1}(\varkappa)\to \tilde\pi^{-1}(\varkappa)=\tilde\CR{}^\varkappa$, denoted again by
$\iota$, and the pushforward homomorphism
\[
  \iota_*\colon H^{T_\cO}_*(\tilde\pi_T^{-1}(\varkappa))\to
  H^{G_\cO}_*(\tilde\CR{}^\varkappa)\otimes_{H^*_G(\mathrm{pt})} H^T_*(\mathrm{pt}).
\]

Let $\pi_T\colon \cR_{\tilde T,\bNT}\to \Gr_{\tilde T}$ be the projection. We lift
the coweight $\varkappa$ of $G_F$ to $\tilde T$ by setting the $\sw_{i,r}$-coordinate of
the $T$-component to be 0
\begin{NB3}
  This is \emph{not} true if $s=N$.
  \begin{NB2}
    Corrected on June 18, 2020.
    \end{NB2}%
\end{NB3}%
for any $\sw_{i,r},\ i\in Q_0,\ 1\leq r\leq a_i$.
Let us denote it by $\tilde\varkappa$. We consider the fundamental class of
$\pi_T^{-1}(\tilde\varkappa)$ and denote it by $\sz^\varkappa$.
\begin{NB3}
  Corrected from $\vk^0$ to $\tilde\vk$. June 20.
\end{NB3}%
By the localization theorem, it is nonvanishing over $\oA^{|\alpha|}$.
Note that the lift $\tilde\vk$ has different $T$ component from one in
\eqref{eq:bow8}. The class $\sz^\vk$ is different from
$\sy^\vk$ used in \ref{line Cherkis}.

We define a rational isomorphism
$\theta\colon \Gamma(\ol\CW{}^{\lambda^*}_{\mu^*},\bpi_*\CalD^{\varkappa})\dasharrow
H_*^{G_\CO}(\tilde\CR{}^\varkappa)$
by sending $z^\varkappa$ to $\iota_* \sz^\varkappa$. It is $S_\alpha$-equivariant,
hence it is indeed an isomorphism as above.

\begin{NB}
Recall that $\on{Spec}H^*_{G_\CO}(\on{pt})=\BA^\alpha\leftarrow\BA^{|\alpha|}=
\on{Spec}H^*_{T(V)_\CO}(\on{pt})$, and the base change under
$\BA^\alpha\leftarrow\BA^{|\alpha|}$ gives
$H_*^{T(V)_\CO}(\tilde\pi^{-1}(\varkappa))$, where $T(V)$ is a Cartan torus of
$G=\GL(V)$. If we further localize to $\oG_m^{|\alpha|}\subset\BA^{|\alpha|}$,
we have a localization isomorphism
$\bz^*\iota_*^{-1}\colon H_*^{T(V)_\CO}(\tilde\pi^{-1}(\varkappa))_{\on{loc}}\iso
H_*^{T(V)_\CO}(\hat\pi^{-1}(\varkappa))_{\on{loc}}$ where
$\hat\pi\colon\Gr_{T(V)\times G_F}\to\Gr_{G_F}$ is the obvious projection.
But $H_*^{T(V)_\CO}(\hat\pi^{-1}(\varkappa))\cong H_*^{T(V)_\CO}(\Gr_{T(V)})=
\BC[\BA^{|\alpha^*|}\times\BG_m^{|\alpha^*|}]$ by~\ref{rem:W-cover}(2).
All in all, we obtain an $S_\alpha$-equivariant trivialization
\begin{equation}
\label{costalk triv slice}
i^!_\varkappa\tilscA^{\on{for}}|_{\overset{\circ}\BG{}_m^{|\alpha|}}\cong
\CO_{\overset{\circ}\BG{}_m^{|\alpha|}\times\BG_m^{|\alpha|}}.
\end{equation}
Composing with the trivialization $\tau$, we obtain a rational
isomorphism of $\BC[\ol\CW{}^{\lambda^*}_{\mu^*}]$-modules
$\theta\colon \Gamma(\ol\CW{}^{\lambda^*}_{\mu^*},\bpi_*\CalD^{\varkappa})\dasharrow
i^!_\varkappa\tilscA^{\on{for}}$.
\end{NB}%

\begin{Theorem}
  \label{det via hom slice}
  \begin{NB3}
    Here $\varkappa = (1,\dots,1,0,\dots,0)$.
  \end{NB3}%
The rational isomorphism $\theta\colon \Gamma(\ol\CW{}^{\lambda^*}_{\mu^*},
\bpi_*\CalD^{\varkappa})\dasharrow H_*^{G_\CO}(\tilde\CR{}^\varkappa)$
extends to the regular isomorphism of $\BC[\ol\CW{}^{\lambda^*}_{\mu^*}]$-modules
$\theta\colon \Gamma(\ol\CW{}^{\lambda^*}_{\mu^*},
\bpi_*\CalD^{\varkappa})\iso H_*^{G_\CO}(\tilde\CR{}^\varkappa)$.
\end{Theorem}

\begin{proof}
We follow the standard scheme, see e.g.\ the proof of~\ref{Coulomb_quivar}.
We have to check that $\theta$ extends through the general points of the
boundary divisor $\BA^{|\alpha|}\setminus\oG_m^{|\alpha|}$. Namely,

\begin{aenume}
      \item $w_{j,s}(t) = w_{i,r}(t)$ for some $i\ne j$ connected by an edge, $r$,
$s$, but all
    others are distinct. Moreover $w_{k,p}(t)\neq 0$ if $N_k\ne0$.
  \item $w_{i,r}(t) = w_{i,s}(t)$ for distinct $r$, $s$ and some $i$,
    but all others are distinct. Moreover $w_{j,p}(t)\ne0$ if $N_j\ne0$.
  \item All pairs like in (a),(b) are distinct, but $w_{i,r}(t) = 0$
    for $i$ with $N_i\ne0$.
\end{aenume}
The gauge theory $(G,\bN,\tilde G)$ with the flavor symmetry group $\tilde G$ is
replaced by $(Z_G(t),\bN^t,Z_{\tilde G}(t))$. In our case,
$Z_{\tilde G}(t)=Z_G(t)\times T(W)$, and
\(
    (Z_G(t),\bN^t) = (\GL(V') \times T'', \bN(V',W')),
\)
where $V'$, $W'$ are given below, $V=V'\oplus V''$ and $T''$ acts trivially on
$\bN(V',W')$:
\begin{aenume}
\item $W' = 0$, $V'_j=\BC=V'_i$ and other
  entries are $0$.
  \item $W' = 0$, $V'=\BC^2$ and other
    entries are $0$.
    \item $V'_i = 1$, $W'_i=\BC^{N_i}$ and other entries are $0$.
\end{aenume}
The extra factor $T(W)$ acts trivially in (a),(b), while
it acts through $T(W) \to T(W_i)$ in~(c).

By the same argument as in the proofs of \ref{Coulomb_quivar},
both $z^\varkappa$ and $\sz^\varkappa$ are related to $z^{\prime\varkappa}$,
$\sz^{\prime\varkappa}$ by nonvanishing regular functions defined on a
neighborhood of $t$ in $\BA^{|\alpha|}$ under the factorization.
Therefore it is enough to check that the isomorphism $\theta$ extends
for the local models (a),(b),(c) above.

(a) According to~\ref{example}, both $z^{\prime\varkappa}$ and $\sz^{\prime\varkappa}$
are invertible at the general points of the divisor $w_i=w_j$ (recall that
we assume $w_i\ne0\ne w_j$).

(b) The zero divisor of $z^{\prime\varkappa}$ is the union of the zero divisors
of $w_{i,1}$ and $w_{i,2}$; in particular, $z^{\prime\varkappa}$ is invertible at the
general points of the divisor $w_{i,1}=w_{i,2}$ (recall that we assume
$w_{i,1}\ne0\ne w_{i,2}$). The homology class $\sz^{\prime\varkappa}$ is invertible
as well.

(c) We make use of the $\BC^\times\times\BC^\times$-action on $\wit\cS_{N_i}$
of~\ref{line Klein}. A dominant weight $\lambda$ of~\ref{line Klein} is now
$\varkappa^{(i)}=\omega_n=(1,\ldots,1,0,\ldots,0)$ ($n$ 1's, and we allow the possibility
$n=N_i$, when $\varkappa^{(i)}=(1,\ldots,1)$).
\begin{NB3}
  Corrected on June 18, 2020.
  \end{NB3}%
The fundamental class
$\sz^{\prime\varkappa^{(i)}}$ is an eigenvector of $\BC^\times\times\BC^\times$ with the
eigencharacter $x^nt^n$.
Since all the eigenspaces are 1-dimensional, it suffices to check that
$z^{\prime\varkappa^{(i)}}$ has the same eigencharacter.
Now the $x$-character of $z^{\prime\omega_n}$ is $x^n$ since $z^{\prime\omega_n}$ is
a highest vector of the irreducible $\GL(2)$-module with highest weight $(n,0)$.
The exponent of the $t$-character of $z^{\prime\omega_n}$ is minimal among all such
exponents with the fixed $x$-character. Hence the $t$-character of
$z^{\prime\omega_n}$ is $t^n$.

For the sake of completeness, note that the divisor of $z^{\prime\omega_n}$ is the
union of $E_1,\ldots,E_{n-1}$ and the strict transform of $\{Z=0\}$ (notation of the beginning
of~\ref{sec:mult} and \ref{line Klein}).
\begin{NB3}
  Corrected on June 18, 2020. And added on June 19 by HN.
\end{NB3}%

We conclude by an application of~\ref{prop:flat} and~\ref{rem:conditions}.
The condition $\varPi_*\bpi_*\CalD^{\varkappa}\iso
j_*\varPi_*\bpi_*\CalD^{\varkappa}|_{(\wit\CW{}^{\unl\lambda^*}_{\mu^*})^\bullet}$
of~\ref{rem:conditions} is satisfied since $\wit\CW{}^{\unl\lambda^*}_{\mu^*}$ is
Cohen-Macaulay, and the complement of
$(\wit\CW{}^{\unl\lambda^*}_{\mu^*})^\bullet$
in $\wit\CW{}^{\unl\lambda^*}_{\mu^*}$ is of codimension $2$. The latter claim
follows from the semismallness of
$\bpi\colon \wit\CW{}^{\unl\lambda^*}_{\mu^*}\to\ol\CW{}^{\lambda^*}_{\mu^*}$ as in
the proof of~\ref{semismall}, and the Cohen-Macaulay property is proved the
same way as in~\ref{lem:CMslice} and~\ref{lem:CMBD}.
\end{proof}

\begin{NB}
The following is moved to the beginning.

  We choose isomorphisms 
$\Gamma(\cS_N,\CF_\lambda)\iso i^!_\lambda\tilscA^{\on{for}}$ for any $\lambda$
(defined uniquely up to multiplication by a scalar).

\begin{Lemma}
\label{tensor surj old}
The multiplication morphism $\Gamma(\cS_N,\CF_\lambda)\otimes\Gamma(\cS_N,\CF_\mu)
\to\Gamma(\cS_N,\CF_{\lambda+\mu})$ (resp.\ $i^!_\lambda\tilscA^{\on{for}}
\otimes i^!_\mu\tilscA^{\on{for}}\to i^!_{\lambda+\mu}\tilscA^{\on{for}}$) is surjective
for any dominant $\lambda,\mu$.
\end{Lemma}

\begin{proof}
It suffices to consider the case $\mu=\omega_n=(1,\ldots,1,0,\ldots,0)$.
Recall that the $\BC^\times\times\BC^\times$-character of 
$\Gamma(\cS_N,\CF_{\lambda+\mu})$ given by~\ref{monopole Klein} is multiplicity
free. So it suffices to represent each summand
$x^{\sum_{i=1}^N((\lambda+\omega_n)_i-m)}t^{\sum_{i=1}^N|(\lambda+\omega_n)_i-m|}$ as a product
of two summands $x^{\sum_{i=1}^N(\lambda_i-m')}t^{\sum_{i=1}^N|\lambda_i-m'|}$ and
$x^{\sum_{i=1}^N((\omega_n)_i-m'')}t^{\sum_{i=1}^N|(\omega_n)_i-m''|}$. Now if 
$m\geq\lambda_n+1$, we take $m'=m-1,\ m''=1$, and if $m\leq\lambda_n$, we take
$m'=m,\ m''=0$. The same argument works for the costalks of $\tilscA^{\on{for}}$
due to the monopole formula.
\end{proof}

\begin{Lemma}
\label{tensor Klein old}
The diagram 
$$\begin{CD}
\Gamma(\cS_N,\CF_\lambda)\otimes_{\BC[\cS_N]}\Gamma(\cS_N,\CF_\mu) @>\sim>>
i^!_\lambda\tilscA^{\on{for}}\otimes_{i^!_0\tilscA^{\on{for}}}i^!_\mu\tilscA^{\on{for}} \\
@VVV @VVV \\
\Gamma(\cS_N,\CF_{\lambda+\mu}) @>\sim>> i^!_{\lambda+\mu}\tilscA^{\on{for}}
\end{CD}$$
commutes up to multiplication by a scalar for any dominant $\lambda,\mu$.
\end{Lemma}

\begin{proof}
The kernels of both vertical morphisms coincide with the torsion in the upper
row. Thus it suffices to check the claim generically. But generically all the
four modules in question are free of rank 1. So it suffices to check the
commutativity for a single $\BC^\times\times\BC^\times$-eigensection of 
$\Gamma(\cS_N,\CF_\lambda)\otimes_{\BC[\cS_N]}\Gamma(\cS_N,\CF_\mu)$, and this 
follows from the multiplicity free property of $i^!_{\lambda+\mu}\tilscA^{\on{for}}$.
\end{proof}
\end{NB}%

Now we construct an isomorphism 
$\theta_\varkappa\colon \Gamma(\ol\CW{}^{\lambda^*}_{\mu^*},
\bpi_*\CalD^{\varkappa})\iso H_*^{G_\CO}(\tilde\CR{}^\varkappa)$ for arbitrary
$\varkappa\in\Lambda^{++}$ inductively, with~\ref{det via hom slice} as the
base of induction. More precisely, we write $\varkappa=\sum_l\varkappa_l$,
where each $\varkappa_l$ is of the form $(1,\ldots,1,0,\ldots,0)$ considered
in~\ref{det via hom slice}.

\begin{Theorem}
  \label{det via hom slice general}
  There is a unique isomorphism
  $\theta_\varkappa\colon \Gamma(\ol\CW{}^{\lambda^*}_{\mu^*},
  \bpi_*\CalD^{\varkappa})\iso H_*^{G_\CO}(\tilde\CR{}^\varkappa)$ making the
  following diagram commutative:
  \begin{equation*}
    \begin{CD}
{\bigotimes\limits^l}_{\BC[\ol\CW{}^{\lambda^*}_{\mu^*}]}\Gamma(\ol\CW{}^{\lambda^*}_{\mu^*},
  \bpi_*\CalD^{\varkappa_l}) @>\sim>{\bigotimes\limits^l\theta_{\varkappa_l}}>
{\bigotimes\limits^l}_{H_*^{G_\CO}(\CR)}H_*^{G_\CO}(\tilde\CR{}^{\varkappa_l}) \\
@VVV @VVV \\
\Gamma(\ol\CW{}^{\lambda^*}_{\mu^*},
  \bpi_*\CalD^{\varkappa}) @>\sim>> H_*^{G_\CO}(\tilde\CR{}^\varkappa).
  \end{CD}
  \end{equation*}
\end{Theorem}

\begin{proof}
Assume $\varkappa=\varkappa'+\varkappa''$, and 
$\theta_{\varkappa'},\theta_{\varkappa''}$ are already constructed.
Then we restrict to $(\ol\CW{}^{\lambda^*}_{\mu^*})^\bullet\stackrel{j}
\hookrightarrow\ol\CW{}^{\lambda^*}_{\mu^*}$, and note that
$j^*\bpi_*\CalD^\varkappa$ is the quotient of 
$j^*\bpi_*\CalD^{\varkappa'}\otimes j^*\bpi_*\CalD^{\varkappa''}$ modulo torsion,
due to factorization and~\ref{tensor surj}. Similarly,
$j^*H_*^{G_\CO}(\tilde\CR{}^\varkappa)$ is the quotient of
$j^*H_*^{G_\CO}(\tilde\CR{}^{\varkappa'})\otimes j^*H_*^{G_\CO}(\tilde\CR{}^{\varkappa''})$
modulo torsion. So we define $j^*\theta_\varkappa$ as the quotient of
$j^*\theta_{\varkappa'}\otimes j^*\theta_{\varkappa''}$ modulo torsion. Finally,
we define $\theta_\varkappa$ as $j_*j^*\theta_\varkappa$.
The resulting diagram commutes thanks to \ref{tensor Klein}.
\begin{NB3}
  Added by HN on June 19.
\end{NB3}%
\end{proof}


\bibliographystyle{myamsalpha}
\bibliography{nakajima,mybib,coulomb}

\end{document}